\documentclass[11pt,reqno]{amsart}
\usepackage[utf8]{inputenc}
\usepackage{pgfplots}
\usepackage{fancyhdr}
\usepackage{tikz}
\usepackage{systeme}
\usepackage{bbm}
\usepackage{graphicx}
\usepackage[rightcaption]{sidecap}
\graphicspath{ {images/} }
\usepackage[utf8]{inputenc}
\usepackage{mathtools}
\usepackage{amsmath,dsfont,amssymb,mathrsfs,amsthm,hyperref,color}
\usepackage{graphicx}
\usepackage{subcaption}
\usepackage[english]{babel}

\usepackage{pst-node}
\usepackage{tikz-cd} 
\usepackage{circuitikz}
\usepackage{amsthm}
\usepackage{hyperref}
\usepackage{cleveref} 

\newtheorem{theorem}{Theorem}[section]
\newtheorem{corollary}[theorem]{Corollary}
\newtheorem{lemma}[theorem]{Lemma}
\newtheorem{proposition}[theorem]{Proposition}
\theoremstyle{definition}
\newtheorem{definition}[theorem]{Definition}

\theoremstyle{remark}    
\newtheorem{remark}[theorem]{Remark}

\newtheoremstyle{boldcondition}
  {\topsep}   
  {\topsep}   
  {\bfseries} 
  {}          
  {\bfseries} 
  {.}         
  {.5em}      
  {}          

\newtheorem{condition}{Condition}

\title{A Minimax Bilinear Transport Problem and Nash-Monge-Kantorovich  Maps} 

     \author{Rene Cabrera and Edward Huynh}

     \thanks{RC was partially supported by the NSF DMS RTG 1840314 during Spring 2025 when this project began to take shape. EH was funded under the National Defense Science and Engineering Graduate (NDSEG) Fellowship administered through the Department of the Air Force (AFRL/Space Force).}
     \address{Department of Mathematics\\
The University of Texas at Austin\\
2515 Speedway Stop C1200\\
Austin, Texas 78712-1202}
\email{rene.cabrera@math.utexas.edu}
\address{
Oden Institute for Computational Engineering and Sciences\\
The University of Texas at Austin\\
201 East 24th Street\\
Austin, Texas 78712
}
\email{edhuynh@utexas.edu}
 
\DeclareMathOperator*{\argmax}{argmax} 
\DeclareMathOperator*{\argmin}{argmin} 
  \pgfplotsset{compat=1.18}

  \date{\today}

\begin{document}

\begin{abstract}
   We study a min–max bilinear transport problem arising from a two-player zero-sum game with quadratic kinetic and interaction costs. Starting from a dynamic path space formulation, we establish existence of minimax and maximin plans and prove a minimax theorem. We show that the equilibrium induces a finite-dimensional stationary problem via an endpoint cost on transport plans, which is well defined below a critical interaction strength and yields a Nash equilibrium over couplings. In the quadratic interaction case, we derive an explicit endpoint cost and a dual formulation. The resulting Nash–Monge–Kantorovich (NMK) plans admit Monge solutions, recovering classical structures in optimal transport, with optimal maps given by gradients of convex or concave functions when they exist. Our analysis highlights duality and cyclical (anti-)monotonicity for nonstandard costs and links the equilibrium maps to coupled nonlinear PDEs, bridging optimal transport, zero-sum games, and Monge–Ampère-type equations.
\end{abstract}
\maketitle 

\tableofcontents

\section{Introduction and Main Results}
Optimal transport theory (OT) provides a powerful framework for comparing probability measures by minimizing a transport cost. Since the foundational works of Monge and Kantorovich (MK), the theory has been extended in many directions, including dynamical formulations, multi-marginal problems, and interactions between transported masses. In this paper, we introduce a new class of optimal transport problems in which \emph{two transport plans interact through a bilinear cost} and the optimization is formulated as a \emph{minimax problem}. This naturally connects optimal transport with ideas from game theory, where competing agents optimize against each other.

More precisely, given probability measures $(\mu_1,\nu_1)$ on a domain $X$ and $(\mu_2,\nu_2)$ on a domain $Y$, we consider pairs of transport plans $(\pi_1,\pi_2)$ and study the \emph{minimax bilinear transport problem}
\begin{gather}\label{inf sup problem}
    \iint c(\gamma, \xi)\, d\pi_1(\gamma)\, d\pi_2(\xi)
    = \inf_{\widetilde{\pi}_{1}\in\Pi_{\text{path}}(\mu_{1}, \nu_{1})}
      \sup_{\widetilde{\pi}_{2}\in\Pi_{\text{path}}(\mu_{2},\nu_{2})}
      \iint c(\gamma, \xi)\, d\widetilde{\pi}_1(\gamma)\, d\widetilde{\pi}_2(\xi),
\end{gather}
where $\Pi_{\text{path}}(\mu_i,\nu_i)$ denotes admissible transport plans on a path space. The cost $c(\gamma,\xi)$ couples two trajectories, leading to an interacting transport problem that differs fundamentally from classical optimal transport.

The purpose of this paper is to develop a theory for \eqref{inf sup problem} and its stationary counterpart. Our main results can be summarized as follows.
Under mild assumptions on the cost function and the data, we prove that minimax solutions exist. Moreover, we establish a minimax theorem showing that, under suitable conditions, the minimax and maximin values coincide. The corresponding optimal pair of transport plans defines a \emph{Nash equilibrium transport plan}. We also make a reduction from paths to endpoints. For a class of path-dependent costs, we show that solutions of the path space problem project to solutions of a stationary problem with an effective endpoint cost.

In the quadratic-type setting, we prove that bilinear plans are induced by transport maps, which are uniquely characterized by gradients of convex or concave potentials, in analogy with Brenier's theorem. In addition,  we derive a dual formulation of the problem and show, formally, the equivalence between the dynamical (path space) formulation and the stationary one via a change of variables.

We now state the main results illustrating these contributions.
\begin{theorem}\label{thm:minimaxExistence}
Suppose $\mu_1,\nu_1 \in \mathcal{P}(X)$ and $\mu_2,\nu_2 \in \mathcal{P}(Y)$ have compact support and $c$ satisfies Condition \ref{condition 1}. Then there exist at least one minimax solution of (\ref{inf sup problem}).
\end{theorem}

Under certain conditions on the cost function, the plans obtained from Theorem \ref{thm:minimaxExistence} also yield \textit{strong duality}, which will be referred to when the minimax equals the maximin problem. 

\begin{theorem}\label{thm:minimax=maximin}
    Let $\pi_{1}\in \Pi_{\text{path}}(\mu_{1},\nu_{1})$ and $\pi_{2}\in \Pi_{\text{path}}(\mu_{2}, \nu_{2})$. Given the bilinear functional $\sigma_{0}:\mathcal{P}(\Omega_{1})\times \mathcal{P}(\Omega_{2})\to \mathbb{R}$ defined by (\ref{sigma functional}), there exists a Nash Equilibrium Transport Plan (NETP) (Definition \ref{def: NETP}) $(\pi_1^*,\pi_2^*)$ such that we have 
     \begin{align}\label{eq: minimax}
    \sigma_0(\pi_1^*,\pi_2^*) = \inf_{\pi_{1}\in \Pi_{\text{path}}(\mu_{1},\nu_{1})}\sup_{\pi_{2}\in \Pi_{\text{path}}(\mu_{2}, \nu_{2})}\sigma_{0}(\pi_{1}, \pi_{2})= \sup_{\pi_{2}\in \Pi_{\text{path}}(\mu_{2}, \nu_{2})}\inf_{\pi_{1}\in\Pi_{\text{path}}( \mu_{1}, \nu_{1})}\sigma_{0}(\pi_{1}, \pi_{2}).
\end{align}
\end{theorem}

We specialize the results to a cost function involving the difference of kinetic energies plus a quadratic interaction term between paths, although our theory supports the more general costs that have Fourier transforms of finite, positive measure, which are functions of positive type. We show that the solution to the path space problem \eqref{inf sup problem} projects to a solution to a stationary formulation as was done by Cabrera in \cite{cabrera2022optimaltransportationprincipleinteracting}-\cite{Cabrera2022} (in the case of path and quadratic congestion) via an endpoint cost function defined similarly as to \cite{villani2008optimal}. 

\begin{theorem}\label{thm:pathtostationary}
    Let $c$ be as defined in Condition \ref{condition 2}. If $(\pi_{1}^{*}, \pi_{2}^{*})$ solves (\ref{inf sup problem}) on paths, then (1) $(\pi_{1}^{*}, \pi_{2}^{*})$ is supported on $c_{\pi_{1}^{*}, \pi_{2}^{*}}$-minimaximal paths and (2) the couple $(e_{0}, e_{1})_{\sharp}\pi_{1}^{*}, (e_{0}, e_{1})_{\sharp} \pi_{2}^{*}$ solve the Nash-Monge-Kantorovich problem for the pair $c_{e,\pi_{1}^{*}}$ and $c_{e,\pi_{2}^{*}}$, defined by \eqref{eqn: effective costs}. 
\end{theorem}

In particular, for quadratic type interactions, the problem exhibits strong structural properties analogous to classical optimal transport.

\begin{theorem}\label{thm:NashBrenier}
    Let the cost function $c$ be given by Condition \ref{condition 2}. Then \[c(x,y,x',y') = \frac12|x-y|^2 - \frac12|x'-y'|^2 + \frac\alpha3\left[(x-x')^2 + (x-x')(y-y') + (y-y')^2 \right],\]
    where $\alpha < \frac{\pi^{2}}{2}.$ Suppose $\mu_1,\nu_1,\mu_2,\nu_2 \in \mathcal{P}_2(\mathbb{R}^d)$ are absolutely continuous measures with respect to the Lebesgue measure. Then for $\alpha \neq 3$ the solution to the Nash equilibrium stationary problem is unique and admits transport maps $T_1(x)$ and $T_2(x')$. In particular when $\alpha < 3$, then the maps are given as gradients of convex potentials. Similarly, if $\alpha > 3$, then the maps are given as gradients of concave potentials. 
    
    Furthermore, if $(\pi_{1}^{*}, \pi_{2}^{*})$ is a pair of NETPs (Def. \ref{def: NETP}) for the end point cost functions $c_{e, \pi_{i}^{*}}$ \eqref{eqn: effective costs}, then the unique coupled solution $(\pi_{1}^{*},\pi_{2}^{*})$ to \eqref{inf sup problem} is given by coupled minimaximal maps of paths $\Gamma^{i}:\text{spt}\left(\mu_{i}\right)\to \Omega_{i}$, for $i=1,2.$
\end{theorem}

In standard optimal transport theory, one automatically gets a \emph{dual problem}. Analogously, we get a \emph{coupled dual problem} as well corresponding to the bilinear transport problem.

\begin{theorem}\label{thm:dualproblm}
    If $(\pi_1^*,\pi_2^*)$ solves \eqref{inf sup problem}, then \eqref{inf sup problem} can be characterized by a coupled dual problem.
\end{theorem}

\subsection{Relation to Existing Work.}
Our formulation is closely related to several directions in the literature. The bilinear structure is reminiscent of Gromov--Wasserstein distances \cite{zhang2024gradient}, which involve quadratic interactions in the transport plan. Unlike those formulations, however, our problem is posed as a minimax optimization over two distinct plans---\emph{bilinear plans}. The quadratic optimal transport problem studied by Jeong in \cite{jeong2025quadraticoptimaltransportationproblem} corresponds to a single-plan bilinear cost and exhibits non-uniqueness phenomena that we extend to the minimax setting. 

Our work is also inspired by path-dependent and congestion formulations of optimal transport. In \cite{cabrera2022optimaltransportationprincipleinteracting} interacting path costs are introduced in a purely minimizing framework. Here we incorporate such ideas into a minimax setting, leading to new connections with differential games and saddle-point problems.

The functional \eqref{inf sup problem} can be interpreted as a two-player zero-sum game, where each player selects a transport plan. A minimax solution corresponds to a strategy that minimizes the worst-case cost, while a maximin solution maximizes the best guaranteed payoff. When these coincide, the solution defines a Nash equilibrium. This perspective naturally explains the emergence of mixed strategies (transport plans) and, in certain regimes, pure strategies (transport maps).

A key feature of our approach is the incorporation of path-dependent costs. Given trajectories $\gamma$ and $\xi$, we consider costs of the form
\begin{align}\label{Lagrangian cost}
c(\gamma,\xi) := \int_0^1 L(\gamma(t),\dot\gamma(t),\xi(t),\dot\xi(t),t)\, dt,
\end{align}
which gives rise to the functional
\begin{align}\label{sigma functional}
\sigma_0(\pi_1,\pi_2) := \iint c(\gamma,\xi)\, d\pi_1(\gamma)\, d\pi_2(\xi).
\end{align}
This formulation connects optimal transport with optimal control \cite{villani2021topics}, \cite{santambrogio2015optimal} and differential games \cite{isaacs1965differential}, and allows us to capture interactions along trajectories rather than only at endpoints.

\medskip

\noindent\textbf{Organization of the paper.}
The remainder of the paper is organized as follows. In Section~\ref{prelims} we provide a technical account of optimal transport theory and how it is associated to both a quadratic and bilinear transport setting. Furthermore, we give the definitions of a solution to \eqref{inf sup problem} and a \emph{Nash Equilibrium Transport Plan} and \emph{maps}, and apply them to elementary calculations. In Section \ref{section: path formulation}  we develop the path space formulation and prove existence results ending the section with a proof of a minimax value coinciding with a maximin one. Section \ref{NMK maps} establishes a stationary calculation of the cost in order to obtain the Nash-Monge-Kantorovich transport maps. In particular, Section~\ref{section: Nash-Monge-Kantorovich Transport Map} analyzes the stationary problem and proves the Brenier-type structure result. Lastly, Section~\ref{section: dual} establishes duality. In addition, we prove a coupled dual problem corresponding to \eqref{inf sup problem} in which we use to construct the NMK maps which come from the path space.  Furthermore, Section \ref{sec: effective cost and cyclical monotonicity} exhibits monotonicity to acquire bilinear plans given my maps which are uniquely determined.   Additional technical details and comparisons are provided throughout.

\section{Preliminaries}\label{prelims}
In standard optimal transport theory, we transport commodities in the most efficient way from one location onto another. The problem is famously known as the \textbf{Monge-Kantorovich problem}. First for the \textbf{Monge Problem}, a \textit{transport map} between two probability measures $\mu$ and $\nu$, both in $X\subset\mathbb{R}^{d}$ and in $Y\subset \mathbb{R}^{d}$, respectively, is a map $T: X \to Y$ such that it is the \textit{push forward} of $\mu$ onto $\nu$, $T_{\sharp}\mu=\nu$. That is, $T:X\to Y$ is such that for any measurable subset $E$ of $\mathbb{R}^{d}$, $\nu(E):=T_{\sharp}\mu(E)=\mu\left(T^{-1}(E)\right)$. Given a measurable \textit{cost function} $c:X\times Y\to\mathbb{R}$, the \textit{optimal transport problem} or the \textit{Monge problem} transporting $\mu$ onto $\nu$ is to find a measurable map $T$ which satisfies 
\begin{align}\label{monge problem}
    \int_{\mathbb{R}^{d}}c(x, T(x))d\mu(x)=\inf_{\widetilde{T}_{\sharp}\mu=\nu}\int_{\mathbb{R}^{d}}c(x, \widetilde{T}(x))d\mu(x).
\end{align}
However, this problem may not be feasible. To wit, no admissible mapping $T$ may exist; simply take a Dirac delta $\delta = \mu$ for the source measure while the target measure $\nu$ is not \cite{ambrosio2013user}. Hence, $T$ is not measure preserving. Essentially, in the \emph{Monge problem} the mass is not allowed to split. In contrast, the Kantorovich relaxation allows for splitting of the masses, making the problem feasible. 
   
In the \textbf{Kantorovich Problem}, we instead look for a \textit{transport plan}, $\pi$, which is a probability measure $\pi$ on $X\times Y\subset\mathbb{R}^{d}\times\mathbb{R}^{d}$ whose marginals are given as $\mu$ and $\nu$ through the following characterization. For all bounded and continuous functions $h:\mathbb{R}^{d}\to\mathbb{R}$,
\begin{align}\label{transport plan}
    \int_{\mathbb{R}^{d}\times \mathbb{R}^{d}}h(x)d\pi(x,y)=\int_{\mathbb{R}^{d}}h(x)d\mu(x),\quad \int_{\mathbb{R}^{d}\times \mathbb{R}^{d}}h(y)d\pi(x,y)=\int_{\mathbb{R}^{d}}h(y)d\nu(y).
\end{align}
The set of all such transport plans with source and target marginals, respectively, $\mu$ and $\nu$ is denoted by $\Pi(\mu, \nu)$.

The \textit{Kantorovich problem} is the following. We seek a transport plan $\pi \in \Pi(\mu, \nu)$ that satisfies
\begin{align}\label{Kantorovich problem}
\int_{\mathbb{R}^{d}\times\mathbb{R}^{d}}c(x,y)d\pi(x,y)=\inf_{\widetilde{\pi}\in\Pi(\mu,\nu)}\int_{\mathbb{R^{d}\times\mathbb{R}^{d}}}c(x,y)d\widetilde{\pi}(x,y).
\end{align}
It is well known in the literature, such as in Ambrosio and Gigli's guide \cite{ambrosio2013user}, that if $c(x,y)=|x-y|^{2}$, then Kantorovich's problem is a relaxation of Monge's. In addition, the connection between Monge's and Kantorovich's problem is that any transport map $T$ induces the transport plan defined by $(\text{Id}\times T)_{\sharp}\mu$, and which is concentrated in the graph of $T$. On the other hand, if a transport plan is concentrated on the graph of a measurable mapping $T$, then it is induced by this map. Moreover, if the marginals $\mu, \nu$ are absolutely continuous with respect to Lebesgue measure, then Brenier \cite{brenier1991polar} showed that the \textbf{Monge-Kantorovich problem} (MKP), (\ref{monge problem}), (\ref{Kantorovich problem}), has a solution $T$ which is given by the gradient of a convex function and this function is uniquely determined. Bilinear transport theory shares many application of classical optimal transport theory such as ``game theory".

The result of Theorem \ref{thm:NashBrenier} suggests from the perspective of game theory that the ``game" admits pure Nash strategies as opposed to simply mixed strategies. This furthermore can be easily pulled back into path space to show existence of pure strategies over paths using the results of Theorem \ref{thm:pathtostationary}.

The minimax \emph{stationary} primal problem of (\ref{inf sup problem}) is thus
\begin{align}\label{eqn:inf sup euclidean}
    \iint c(x,y,x',y')d\pi_{1}(x,y)d\pi_{2}(x',y')=\inf_{\widetilde{\pi}_{1}\in \Pi(\mu_{1}, \nu_{1})}\sup_{\widetilde{\pi}_{2}\in \Pi(\mu_{2}, \nu_{2})}\iint c(x,y,x', y')d\widetilde{\pi}_{1}(x,y)d\widetilde{\pi}_{2}(x',y'),
\end{align}
which can be viewed as a relaxation of the \textit{bilinear minimax Monge problem},
\begin{align}\label{min max monge}
    \iint c(x, T(x), y, T(y))d\mu_{1}(x)d\mu_{2}(y)=\inf_{\overline{T}_{\sharp}\mu_{1}=\nu_{1}}\sup_{\overline{T}_{\sharp}\mu_{2}=\nu_{2}}\iint c(x,\overline{T}(x), y, \overline{T}(y))d\mu_{1}(x)d\mu_{2}(y).
\end{align}

In this article, we will be interested in a variant of an optimal transport  problem, namely a \emph{bilinear transport problem}. That is, instead of merely minimizing, we \textit{mini-maximize}, (\ref{inf sup problem}). But first we set the stage and introduce the \textit{bilinear transport problem} \eqref{inf sup problem} through (\ref{qot monge problem})-(\ref{quadratic ot}). Let $X\subset \mathbb{R}^{d}$ be a simply connected, bounded domain. In particular, let $X:=\overline{B}_{R}(0):=\{x\in \mathbb{R}^{d}:\|x\|\leq R\}$ denote the closed ball centered at the origin of radius $R>0$. Fix a cost function $c:(X\times X)^{2}\to \mathbb{R}$, assume $\mu$ and $\nu$ are probability measures on $X\subset\mathbb{R}^{d}$, and $T: X\to X$ a measurable map satisfying the push-forward condition between probability measures $\mu\in \mathcal{P}(X), \nu\in\mathcal{P}(X)$. Then the \textit{quadratic Monge problem} is to find a measurable map $T$ satisfying 
\begin{align}\label{qot monge problem}
    \iint c(x,T(x),y, T(y))d\mu(x)d\mu(y)=\inf_{\widetilde{T}_{\sharp}\mu=\nu}\iint c(x,\widetilde{T}(x),y,\widetilde{T}(y))d\mu(x)d\mu(y).
\end{align}

Just as in the classical theory of optimal transport in which there is a relaxation version of (\ref{monge problem}), which is referred to as the primal problem, (\ref{Kantorovich problem}), there is also a relaxation version of the quadratic Monge problem (\ref{qot monge problem}); also referred as the \textit{primal problem}. 

Indeed, the \textit{Kantorovich bilinear transport  problem} (\ref{quadratic ot}), which, again, will be referred to as the primal problem, is to find a plan $\pi$ that satisfies

\begin{gather}\label{quadratic ot}
    \iint c(x,y,x^{\prime},y^{\prime})\ d\pi(x,y)\ d\pi(x^{\prime},y^{\prime}) = \inf_{\widetilde{\pi}\in\Pi(\mu, \nu)}\iint c(x,y,x^{\prime},y^{\prime})\ d\widetilde{\pi}(x,y)\ d\widetilde{\pi}(x^{\prime},y^{\prime}).
\end{gather}
The above relaxation serves as the analogue of relaxing the Monge problem (\ref{monge problem}) in classical optimal transport to the \textit{Kantorovich problem} (\ref{Kantorovich problem}).

In this novel problem, we will consider the minimization of the maximization, or \textit{mini-maximization}, of the \textit{stationary bilinear transport problem}. To that end, the \textit{minimax primal problem} is to find a pair of transport plans $\pi_{1}, \pi_{2}$ that satisfy the following:
\begin{align}\label{eqn: infsup stationary}
\iint c(x,y,x^{\prime},y^{\prime})\ d\pi_1(x,y)\ d\pi_2(x^{\prime},y^{\prime}) = \inf_{\widetilde{\pi}_{2}\in\Pi(\mu_{2}, \nu_{2})}\sup_{\widetilde{\pi}_{1}\in\Pi(\mu_{1},\nu_{1})} \iint c(x,y,x^{\prime},y^{\prime})\ d\widetilde{\pi}_1(x,y)\ d\widetilde{\pi}_2(x^{\prime},y^{\prime}).
\end{align}
Notice that we are minimizing/maximizing over the transport plans, $\pi_{1}$ and $\pi_{2}$ which have marginals, $\mu_{1}, \nu_{1}$ and $\mu_{2}, \nu_{2}$, respectively. We can similarly formulate a maximin problem. In either case, we study the path space equivalent of these above problems. 

We define a solution to (\ref{inf sup problem}).

\begin{definition}\label{def: MinimaxTP}
    If $(\pi_1^*,\pi_2^*)$ is a pair that solves the minimax problem (\ref{inf sup problem}), we call it a \textit{Minimax Transport Plan}. Similarly, a pair that solves the maximin problem will be called a \textit{Maximin Transport Plan}. Collectively, when a solution is either a minimax solution or a maximin solution, we will refer to them as \emph{bilinear transport plans}.
\end{definition}
When the value of the minimax problem equals the maximin, then we give a special name to this pair:
\begin{definition}\label{def: NETP}
    If $(\pi_1^*,\pi_2^*)$ is a pair that solves the minimax problem (\ref{inf sup problem}) and also the maximin problem, we call it a \textit{Nash Equilibrium Transport Plan (NETP)}. 
\end{definition}
The reason we refer to a solution of (\ref{inf sup problem}) as well as the maximin problem as a NETP is that, in \textit{game theory}, the functional (\ref{inf sup problem}) models a particular instance of a \textit{zero-sum game}. When strong duality holds, the solution characterizes a saddle point equilibrium and is oftentimes referred to as a \textit{Nash equilibrium}. 

A \textit{Nash equilibrium} is a strategy profile in which no player can improve their own payoff by unilaterally deviating from their chosen strategy while the other players' strategies remain fixed. In the context of a two-player zero-sum game, the \textit{minimax} problem of (\ref{inf sup problem}) represents the fact that if the other player plays their best response strategy (which may potentially be a mixed strategy) then one player attempts to minimize the \textbf{worst-case outcome}. Conversely, a solution to the maximin problem is a symmetric scenario where one player attempts to maximize for the \textbf{best-possible profit} assuming the other player plays their best response strategy. Note that a mixed strategy as opposed to a pure strategy represents a probability distribution.

If each player has chosen a strategy – an action plan based on what has happened so far in the game – and no one can increase one's own expected payoff by changing one's strategy while the other players keep theirs unchanged, then the current set of strategy choices constitutes a Nash equilibrium.

\begin{definition}\label{def: NMK}
   A \textit{Nash-Monge-Kantorovich map} (NMK) is a mapping $(T_{1}, T_{2})$ which induces a NETP $(\pi_{1}, \pi_{2})$ through push-forward on maps $T_{1}: X \to Y, \;T_{2}: X'\to Y'$, and the coupled map $(\text{Id}, T_{1}):X\to X\times Y$ such that $\pi_{1}:=(\text{Id}, T_{1})_{\#}\mu_{1}$ and $\pi_{2}:= (\text{Id}, T_{2})_{\#}\mu_{2}$.
\end{definition}
Definition \ref{def: NMK} refers precisely to a pure Nash equilibrium for the game \eqref{eqn: infsup stationary} in the following sense. Suppose $(T_1^*,T_2^*)$ is an NMK map and consider $(T_1,T_2^*)$ where $T_1$ is any other admissible map such that $(T_1)_\#\mu_1 = \nu_1$. Form the measures $\pi_1^*(x,y) = (Id\otimes T_1^*)_{\#}\mu_{1}(x)$, $\pi_2^*(x',y') = (Id\otimes T_2^*)_{\#}\mu_{2}(x')$ and $\pi_1 = (Id \otimes T_1)_{\#}\mu_{1}(x)$. Then by definition of \eqref{eqn: infsup stationary} we have
\begin{align*}
    \iint c(x,y,x',y')\ d\pi_1^*(x,y)d\pi_2^*(x',y') &\leq \iint c(x,y,x',y')\ d\pi_1(x,y)d\pi_2^*(x',y')\\
    &= \iint c(x,y,x',y')\ d\left((Id \otimes T_1)_{\#}\mu_{1}(x) \right)d\pi_2^*(x',y')\\
    &= \iint c(x,T_1(x), x',y')\ d\mu_1(x)d\pi_2^*(x',y').
\end{align*}
On the other hand, we have
\begin{gather*}
    \iint c(x,y,x',y')\ d\pi_1^*(x,y)d\pi_2^*(x',y') = \iint c(x,T_1^*(x), x',y')\ d\mu_1(x)d\pi_2^*(x',y')
\end{gather*}
and hence
\begin{gather*}
    \iint c(x,T_1^*(x), x',y')\ d\mu_1(x)d\pi_2^*(x',y') \leq \iint c(x,T_1(x), x',y')\ d\mu_1(x)d\pi_2^*(x',y').
\end{gather*}
A similar argument holds for
\begin{gather*}
    \iint c(x,y, x',T_2(x'))\ d\mu_2(x')d\pi_1^*(x,y) \leq \iint c(x,y, x',T_2^*(x'))\ d\mu_2(x')d\pi_1^*(x,y).
\end{gather*}
for all admissible maps $T_2$ such that $(T_2)_\#\mu_2 = \nu_2$. Thus, we have proved the following proposition:
\begin{proposition}\label{prop:NE}
    Suppose $(T_1,T_2)$ is a Nash-Monge-Kantorovich mapping for the problem \eqref{eqn: infsup stationary}. Then $(T_1,T_2)$ corresponds to a pure Nash equilibrium and solves the minimax Monge problem \eqref{min max monge}
    \begin{gather*}
        \inf\sup\left\{\iint c(x,\widetilde{T}_1(x), x', \widetilde{T}_2(x') d\mu_1(x)d\mu_2(x')\;\Big| \left(\widetilde{T}_{1}\right)_{\#}\mu_1 = \nu_1,\ \left(\widetilde{T}_2\right)_{\#}\mu_2 = \nu_2 \right\}
    \end{gather*}
    or equivalently
    \begin{gather*}
        \sup\inf\left\{\iint c(x,\widetilde{T}_1(x), x', \widetilde{T}_2(x') d\mu_1(x)d\mu_2(x')\;\Big| \left(\widetilde{T}_1\right)_{\#}\mu_1 = \nu_1,\ \left(\widetilde{T}_2\right)_{\#}\mu_2 = \nu_2 \right\}.
    \end{gather*}
\end{proposition}
However, this map may not exist in general for the same reasons as in standard optimal transport. Therefore, we consider the Kantorovich relaxation \eqref{eqn: infsup stationary} and the existence of NETP solutions which is the same as the existence of a mixed Nash equilibrium.   
We will explore this further and determine how we can construct such maps. See Section \ref{section: dual} for the construction of these maps through paths.

We do not stop there, however. Instead of looking at a cost function $c(x,y, x^{\prime}, y^{\prime})$, representing how much it costs to send a unit of mass at point $x$ to point $y$, and then at $x^{\prime}$ to $y^{\prime}$, we  consider continuous paths $\gamma$ and $\xi$ such that $\gamma(0)=x$ indicates the initial point along $\gamma$ and $\gamma(1)=y$ the arrival point along $\gamma$, and $\xi(0)=x^{\prime}$ which indicates the initial point along $\xi$ and $\xi(1)=y^{\prime}$ the arrival point along $\xi$. Then associate to such paths a cost \eqref{Lagrangian cost} indicating \emph{how much transportation along those paths costs}. Here, $L$ is some Lagrangian. More concretely, if $\Omega$ denotes a path space of continuous rectifiable paths, then the total cost of transporting along all paths in the plan is the functional \eqref{sigma functional}.
 
 This can be thought of as an explicit path-dependent version of the MKP, and it is reminiscent of the optimal transport principle on paths in \cite{cabrera2022optimaltransportationprincipleinteracting}. Indeed, the use of a Lagrangian to model the cost function over paths is quite natural in the context of optimal control problems and dynamic game theory. In particular, several interesting problems in game theory can be reformulated or relaxed to problems of the form \eqref{eqn: infsup stationary}; see for instance problems that arise in the theory of zero-sum games in \cite{isaacs1965differential}. For example, Figure \ref{fig:MotivatingExample} depicts a model problem based on the classic pursuit-evasion problem. In this problem, the players (or agents) are generalized to the level of distributions and the strategies are chosen on the level of paths. We note that this figure shows the potential for mass splitting to occur and hence the question of (NMK) transport maps becomes quite subtle. In that case, the relaxation problem \eqref{eqn: infsup stationary} allows for the existence of \textit{distributions over paths} which corresponds precisely to the notion of the existence of a mixed Nash equilibrium.

 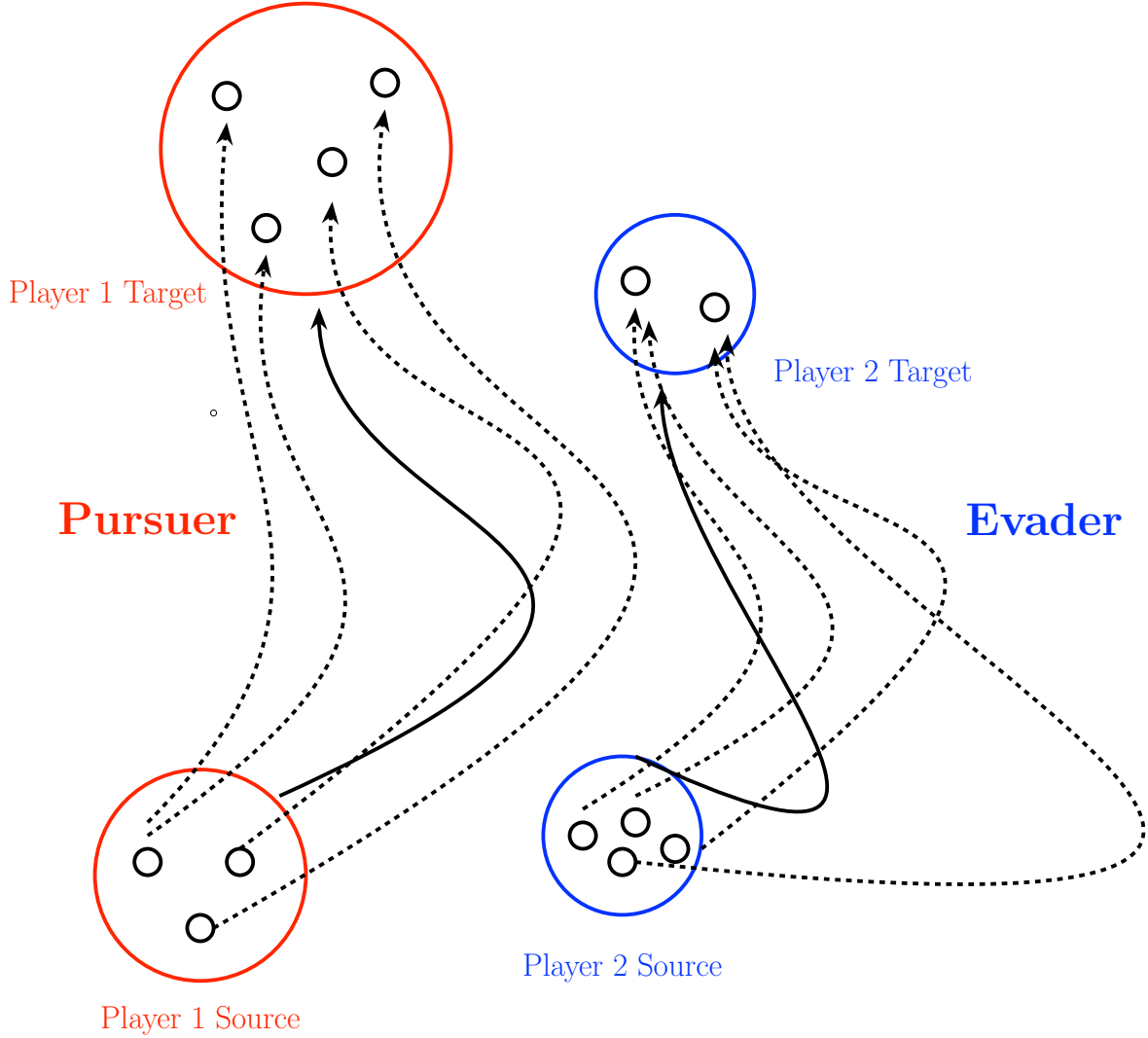
\begin{figure}[!ht]
\centering
\resizebox{1.5\textwidth}{!}{%
\begin{circuitikz}
\tikzstyle{every node}=[font=\Huge]
\draw (-8.75,19.25) to[short, -o] (-8.75,19.25) ;
\draw [ line width=2pt ] (-9.75,11.25) circle (0cm);
\draw [ line width=2pt ] (-8.25,10.75) circle (0.25cm);
\draw [ line width=2pt ] (-10,10.75) circle (0.25cm);
\draw [ line width=2pt ] (-9,9.5) circle (0.25cm);
\draw [ line width=2pt ] (-5.25,19.5) circle (0cm);
\draw [ line width=2pt ] (-10.75,13) circle (0cm);
\draw [ color={rgb,255:red,255; green,38; blue,0} , line width=2pt ] (-9,10.5) circle (2cm);
\draw [ line width=2pt ] (-9.5,23.25) circle (0cm);
\draw [ line width=2pt ] (-8.5,25.25) circle (0.25cm);
\draw [ line width=2pt ] (-5.5,25.5) circle (0.25cm);
\draw [ line width=2pt ] (-6.5,24) circle (0.25cm);
\draw [ line width=2pt ] (-7.75,22.75) circle (0.25cm);
\draw [ color={rgb,255:red,255; green,38; blue,0} , line width=2pt ] (-7,24.25) circle (2.75cm);
\draw [ line width=2pt ] (-6.5,18.5) circle (0cm);
\draw [ line width=2pt ] (-6.75,18.5) circle (0cm);
\draw [ line width=2pt ] (-3.75,16) circle (0cm);
\draw [line width=2pt, ->, >=Stealth] (-7.5,12) .. controls (3,16.75) and (-7,16.75) .. (-6.75,21.25) ;
\draw [ line width=2pt ] (-1.75,11.25) circle (0.25cm);
\draw [ line width=2pt ] (-0.75,11.5) circle (0.25cm);
\draw [ line width=2pt ] (-1,10.75) circle (0.25cm);
\draw [ line width=2pt ] (0,11) circle (0.25cm);
\draw [ color={rgb,255:red,4; green,51; blue,255} , line width=2pt ] (-1,11.25) circle (1.5cm);
\draw [ line width=2pt ] (-0.75,21.75) circle (0.25cm);
\draw [ line width=2pt ] (0.75,21.25) circle (0.25cm);
\draw [ line width=2pt ] (-6.75,24) circle (0cm);
\draw [ color={rgb,255:red,4; green,51; blue,255} , line width=2pt ] (0,21.5) circle (1.5cm);
\draw [ line width=2pt ] (0,20.25) circle (0cm);
\draw [ line width=2pt ] (0,20) circle (0cm);
\draw [line width=2pt, ->, >=Stealth] (-0.75,12.75) .. controls (7.25,9) and (-0.5,16.25) .. (-0.25,19.75) ;
\node [font=\LARGE, color={rgb,255:red,255; green,38; blue,0}] at (-9,7.75) {Player 1 Source};
\node [font=\LARGE, color={rgb,255:red,255; green,38; blue,0}] at (-10.75,21.5) {Player 1 Target};
\node [font=\LARGE, color={rgb,255:red,4; green,51; blue,255}] at (-1,8.75) {Player 2 Source};
\node [font=\LARGE, color={rgb,255:red,4; green,51; blue,255}] at (3.75,20) {Player 2 Target};
\node [font=\LARGE] at (-11,17.25) {};
\node [font=\LARGE] at (-11,17.25) {};
\node [font=\Huge, color={rgb,255:red,4; green,51; blue,255}] at (4.75,16.75) {\textbf{}};
\node [font=\LARGE] at (-11,17.25) {};
\node [font=\Huge, color={rgb,255:red,255; green,38; blue,0}] at (-10,17.25) {\textbf{Pursuer}};
\draw [line width=2pt, ->, >=Stealth, dashed] (-8.25,11) .. controls (4.75,20.75) and (-7.25,17.25) .. (-6.5,23.25);
\draw [line width=2pt, ->, >=Stealth, dashed] (-8.75,9.5) .. controls (8,19) and (-7,17.25) .. (-5.5,25);
\draw [line width=2pt, ->, >=Stealth, dashed] (-10,11.25) .. controls (-2.5,16.5) and (-8.75,16.75) .. (-7.75,22.25);
\draw [line width=2pt, ->, >=Stealth, dashed] (-1.75,11.75) .. controls (5.5,16.25) and (-1.25,16.5) .. (-0.75,21.25);
\draw [line width=2pt, ->, >=Stealth, dashed] (0.5,11) .. controls (10.75,19) and (0.75,15.75) .. (0.75,20.5);
\draw [line width=2pt, ->, >=Stealth, dashed] (-0.75,12) .. controls (7.5,16) and (-0.5,16.5) .. (-0.5,21);
\draw [line width=2pt, ->, >=Stealth, dashed] (-0.75,10.75) .. controls (20.25,8.5) and (0.25,15.75) .. (1,20.75);
\draw [line width=2pt, ->, >=Stealth, dashed] (-10,11.5) .. controls (-5.25,16.25) and (-9.25,18.25) .. (-8.5,24.75);
\node [font=\Huge, color={rgb,255:red,4; green,51; blue,255}] at (7,17.25) {\textbf{Evader}};
\end{circuitikz}
}%
\caption{Pursuit-Evasion Example of Minimax Bilinear Transport}
\label{fig:MotivatingExample}
\end{figure}

\section{Path Formulation for a Minimax Bilinear Problem}\label{section: path formulation} 
Let $X,Y,X',Y'\subset \mathbb{R}^{d}$ be simply connected, bounded domains. In particular, $X,Y,X',Y'$ denote  closed balls centered at the origin of radius $R>0$. Assume $\mu_{1}$ and $\nu_{1}$ are probability measures on $X$ and $Y$, respectively, and $\mu_{2}$ and $\nu_{2}$ are probability measures on $X^{\prime}$ and $Y^{\prime}$, respectively. Suppose $T_{1}: X\to Y$ and $T_{2}: X^{\prime}\to Y^{\prime}$ are measurable maps satisfying the push-forward condition between probability measures.

We now incorporate continuous paths to the bilinear minimax transport problem. Define the path spaces
\begin{align*}
\Omega_1 &= \left\{\gamma\in W^{1,2}([0,1];X):
\int_0^1|\dot\gamma(t)|^2\,dt\leq K\right\},\\
\Omega_2 &= \left\{\xi\in W^{1,2}([0,1];Y):
\int_0^1|\dot\xi(t)|^2\,dt\leq L\right\},
\end{align*}
where $K,L > 0$.

Let $\Omega:=\Omega_{1}\times \Omega_{2}$. We endow $\Omega_i$, $i=1,2$, with the metric
\begin{equation}\label{eq:path-metric}
\|\gamma-\xi\|_{\Omega_i}
:= \|\gamma-\xi\|_{L^\infty([0,1])}
+ \|\dot\gamma-\dot{\xi}\|_{L^2([0,1])}.
\end{equation}

We equip  $\Omega := \Omega_1 \times \Omega_2$ with the product metric
\[
d\big((\gamma_1,\xi_1),(\gamma_2,\xi_2)\big)
:= \sqrt{
\|\gamma_1-\gamma_2\|_{\Omega_1}^2
+
\|\xi_1-\xi_2\|_{\Omega_2}^2 }.
\]

\noindent We now introduce a cost function defined on a path space. Let $c:\Omega \to \mathbb{R}$ be a cost function defined on the joint path space. Costs that enjoy the following conditions will prove essential.


\begin{condition}\label{condition 1}
(Coercivity) Suppose $c: \Omega_1 \times \Omega_2 \to \mathbb{R}$ is bounded, $c(\cdot, \xi)$ is lower semicontinuous, $c(\gamma, \cdot)$ is upper semicontinuous, and given any real number $M > 0$ the set
\begin{gather*}
        \Omega_{M} := \left\{(\gamma,\xi) \in \Omega_1\times \Omega_2: \sqrt{|\gamma(0)|^2 + |\xi(0)|^2} \leq M,\ |c(\gamma,\xi)| \leq M\right\}.
    \end{gather*}
is compact in $\Omega$.
\end{condition}

\begin{condition}\label{condition 2}
 (Quadratic Interaction)\\Suppose the function $c: \Omega_1 \times \Omega_2 \to \mathbb{R}$ is of the form 
\[
c(\gamma, \xi) = \begin{cases}
 \int_0^1 \frac12 |\dot{\gamma}|^2 - \frac12|\dot{\xi}|^2 + \alpha |\gamma - \xi|^2 dt, & \text{if } \dot{\gamma}, \dot{\xi} \in L^{2}(dt) \\
+\infty, & \text{otherwise}\;,
\end{cases}
\]
where $\alpha > 0$.
\end{condition}

We will consider a Lagrangian $L(\gamma,\xi, \dot{\gamma}, \dot{\xi}, t)$ which induces a cost on paths, $c(\gamma,\xi)$ given by \eqref{Lagrangian cost}. For instance, when the Lagrangian is given by
\begin{align}\label{kinetic potential energy}
\begin{split}
L(\gamma(t), \xi(t), \dot{\gamma}(t), \dot{\xi}(t), t)&:=\tfrac{1}{2}|\dot{\gamma}|^{2}-\tfrac{1}{2}|\dot{\xi}|^{2}+\alpha |\gamma(t)-\xi(t)|^{2}
\end{split}
\end{align}
we recover the cost in Condition \ref{condition 2}.

 We will focus on Lagrangians of the form (\ref{kinetic potential energy})
provided $\gamma$ and $\xi$ are continuously differentiable. Everything that follows can be done for more general Lagrangians, for example for functions that have Fourier transforms of the finite, positive measures on $\mathbb{R}^{d}$ which
are exactly the cone of functions of
positive type, which is due to Bochner's theorem \cite{ReedSimon1972}. For example a class of functions that satisfy Bochner's condition of being of positive type are the following: the cost defined in Condition \ref{condition 2} with a Coulomb potential or a Gaussian distribution \cite[Appendix B]{Cabrera2022} instead of the quadratic interaction term, that is,
\begin{align*}
C(\gamma, \xi)&=\int_0^1 \frac12 |\dot{\gamma}|^2 - \frac12|\dot{\xi}|^2 + \alpha |\gamma - \xi|^{2-d} dt\quad (d\geq 3),\\
G(\gamma, \xi)&=\int_0^1 \frac12 |\dot{\gamma}|^2 - \frac12|\dot{\xi}|^2 + \alpha \exp\{\beta|\gamma - \xi|^{2}\} dt\quad(\beta>0).
\end{align*}
This gives a rich class of functions which are of positive type that our theory supports. The first author dealt with these class of background potentials in \cite{cabrera2022optimaltransportationprincipleinteracting}. But for concreteness, we will focus on the Lagrangian in (\ref{kinetic potential energy}). 

The quadratic term given in (\ref{kinetic potential energy}), can be thought of as a background potential $V,$ \[V(\gamma(t), \xi(t),t):=\alpha|\gamma(t)-\xi(t)|^{2}\] 
Then upon taking a gradient with respect to $\gamma$, we have
\begin{align}\label{eqn: separable cost}
   \nabla_{\gamma} V(\gamma, \xi, t)=2\alpha (\gamma-\xi), \nabla_{\xi}V(\gamma,\xi,t)=-2\alpha (\gamma-\xi).
\end{align}
Then the first and second derivatives with respect to $\gamma$ are bounded uniformly in $t$. Indeed, since Condition \ref{condition 1} ensures $\gamma$ and $\xi$ are bounded, namely $\|\gamma(t)\|_{L^{\infty}}\leq M_{1}$ and $\|\xi(t)\|_{L^{\infty}}\leq M_{2}$ for all $t\in [0, 1]$, by (\ref{eqn: separable cost}), we have, $\|\nabla_{\gamma}V(\gamma, \xi,t)\|_{L^{\infty}}\leq 4\alpha (M_{1}+M_{2})$. So $\|\nabla_{\gamma}V\|_{L^{\infty}}$ is uniformly bounded. 

For the second derivative, we look at the Hessian,
\[
D^{2}_{\gamma}V=2\alpha\mathbb{I}_{d}.
\]
As this is constant, it does not depend on $t, \gamma$ or $\xi$. Therefore, the second derivative is always bounded uniformly in $t$. 
\begin{proposition}\label{minimal paths plus potential}
Stationary points of $c(\gamma, \xi)$, with endpoints $x$ and $y$, $x^{\prime}$ and $y^{\prime}$  of $\gamma$ and $\xi$, respectively fixed, satisfy the equation
\[
\ddot{\gamma}(t)-\ddot{\xi}(t) = 0.
\]
provided they are of class $C^2$.
\end{proposition}

\begin{proof}
This is just the Euler-Lagrange equation applied to the cost functional, see Evans, \cite[Ch. 8]{evans2022partial} and \cite{cabrera2022optimaltransportationprincipleinteracting} followed by subtracting the resulting equations.
\end{proof}

\noindent Notice that when $V$ is quadratic, the paths are translates of each other. Namely, if $\ddot{\gamma}(t) - \ddot{\xi}(t) = 0$ for all $t \in [0,1]$, then
\[
\gamma(t) = \xi(t) + at + b, \quad a, b \in \mathbb{R}^d.
\]
That is, the paths differ by an affine function of $t$. If the interaction $V$ is sufficiently small, we can therefore expect that ``minimaximizing" paths $\gamma$ and $\xi$ have the same shape up to an affine $t$-dependent translation. Proposition 2.1 in \cite{cabrera2022optimaltransportationprincipleinteracting} quantifies this intuition.

To end the section, we put probabilities in the path space.
 As such, to include paths to the bilinear minimax (Kantorovich) transport problem \eqref{eqn: infsup stationary}; define the evaluation maps
\[
e_t : \Omega_i \to \mathbb{R}^d, \qquad e_t(\gamma) := \gamma(t),
\quad t \in [0,1], \; i=1,2.
\]
In particular, $e_0(\gamma)=\gamma(0)$ and $e_1(\gamma)=\gamma(1)$. This is so we can study the full problem on paths, \eqref{inf sup problem}.

First, let us recall that given a set $\Omega$ then $\mathcal{P}(\Omega)$ denotes the set of probability measures defined on $\Omega$. We also define
\begin{gather*}
    \mathcal{P}_2(\Omega) =: \left\{\mu \in \mathcal{P}(\Omega): \int_{\Omega} |x|^2d\mu(x) < +\infty \right\}.
\end{gather*}
Then the constraints on the bilinear transport plans that include paths, called \textit{bilinear path plans} are the following: Let $\mu_i,\nu_i \in \mathcal P_2(\mathbb{R}^d)$ for $i=1,2$; $\pi_1$ is a probability measure on paths $\gamma$ with $(e_0)_\# \pi_1 = \mu_1$ and $(e_1)_\# \pi_1 = \nu_1$ (i.e., $\pi_1$ describes paths from $\mu_1$ to $\nu_1$), $\pi_2$ is a probability measure on paths $\xi$ with $(e_0)_\# \pi_2 = \mu_2$ and $(e_1)_\# \pi_2 = \nu_2$ (i.e., $\pi_2$ describes paths from $\mu_2$ to $\nu_2$). We define the set of admissible path transport bilinear plans by 
\begin{align}\label{path plans}
    \Pi_{\text{path}}(\mu_{i},\nu_{i})&:=\left\{\pi_i\in \mathcal{P}(\Omega_{i}): (e_0)_\# \pi_i = \mu_i; (e_1)_\# \pi_i = \nu_i,\; i=1, 2\right\}.
\end{align}
Naturally, transport plans contained in this set will be referred to as \emph{admissible plans}, when it is clear from context. 

Then we seek to find a pair $(\pi_{1}, \pi_{2})$ satisfying \eqref{inf sup problem} over path space, $\Omega$. We will use $\sigma_{0}(\cdot, \cdot)$ given by (\ref{sigma functional}) to denote the bilinear transport functional. We will first show that there exist solutions to \eqref{inf sup problem} under the general hypothesis of Condition \ref{condition 1}.

\subsection{Analysis of the Cost Function}\label{section: analysis of cost} In this section we extend results from \cite{cabrera2022optimaltransportationprincipleinteracting} that will be applied to our setting. We generalize the Lipschitz bounds of the minimal and maximal curves $\gamma, \xi$ which minimaximize the trajectories of the cost $c(\gamma, \xi)$. Namely, 
\begin{proposition}\label{boundary value problem}
 Suppose $\gamma, \xi: [0, 1]\to \mathbb{R}^{d}$ are $L$-Lipschitz continuous, respectively and twice differentiable in $(0, 1)$. Let $c$ be as in Condition \ref{condition 2} with $\alpha<3$. Then for all $x, y, x', y'\in \mathbb{R}^{d}$ there is a unique pair of paths $\gamma_{x, y}(\cdot), \xi_{x^{\prime},y^{\prime}}(\cdot)$ which minimaximizes $c(\gamma, \xi)$ among all paths from $x$ to $y$ and $x^{\prime}$ to $y^{\prime}$. Namely, $\gamma_{x,y}$ minimizes $c(\cdot, \xi)$ while $\xi_{x',y'}$ maximizes $c(\gamma, \cdot)$. The pair of paths $\gamma_{x,y}, \xi_{x^{\prime}, y^{\prime}}$ are of class $C^2$ with respect to $t$. Furthermore the pairs of paths are Lipschitz continuous in $x_{i}, y_{i}$ and $x_{i}^{\prime}, y_{i}^{\prime}$, for $i=1,2$, respectively; and, we have the bounds
 \begin{align*}
     \|\gamma_{x_{1}, y_{1}}(t)-\gamma_{x_{2}, y_{2}}(t)\|_{\infty}\leq\frac{1}{1-\tfrac{\alpha}{3}}\left(|x_{1}-x_{2}|+|y_1-y_2|\right),
 \end{align*}
 \begin{align*}
     \|\xi_{x_{1}^{\prime}, y_{1}^{\prime}}(t)-\xi_{x_{2}^{\prime}, y_{2}^{\prime}}(t)\|_{\infty}\leq\frac{1}{1-\tfrac{\alpha}{3}}\left(|x_{1}^{\prime}-x_{2}^{\prime}|+|y_{1}^{\prime}-y_{2}^{\prime}|\right),
 \end{align*}
 
\end{proposition}

\begin{proof} Consider the metric space 
    \[
    \Omega_{x,y}:=\{\gamma:[0, 1]\to \mathbb{R}^{d}\;|\;\gamma\;\text{is continuous and }\;\gamma(0)=x,  \gamma(1)=y\},
    \]
    and, according to \cite[Proposition 2.4]{cabrera2022optimaltransportationprincipleinteracting}, define the mapping $\mathcal{F}:\Omega_{x,y}\to \Omega_{x,y}$ by 
    \begin{align*}
    \mathcal{F}(\gamma)[t]&=(1-t)\gamma(0)+t\gamma(1)+2\alpha t\int_{0}^{1}(1-s)\left(\gamma(s)-\xi(s)\right)ds-2\alpha\int_{0}^{t}(t-s)\left(\gamma(s)-\xi(s)\right)ds.
    \end{align*}
    We will show that given two end-points $x,y$, then for all pairs $\eta_{1}, \eta_{2}\in \Omega_{x,y}$, we have 
    \[
    \|\mathcal{F}(\gamma_{1})-\mathcal{F}(\gamma_{2})\|_{\infty}\leq L\|\gamma_{1}-\gamma_{2}\|_{\infty},
    \]
    for $L:=\frac{\alpha}{3}<1$ to be determined later. Indeed, this means $\mathcal{F}$ is a contraction mapping on each $\Omega_{x,y}$, so it will have a unique fixed point in $\Omega_{x,y}$ by the Contraction Mapping Theorem.

Let $\gamma_{x,y}$ denote the unique fixed point of $\mathcal{F}$, i.e. $\mathcal{F}(\gamma_{x,y})[t]=\gamma_{x,y}(t)$. Then, $\gamma_{x,y}$ solves the equation
\begin{align*}\gamma_{x,y}(t)&=
(1-t)x+ty+2\alpha t\int_{0}^{1}(1-s)(\gamma_{x,y}(s)-\xi(s))ds-2\alpha\int_{0}^{t}(t-s)(\gamma_{x,y}(s)-\xi(s))ds.
\end{align*}
Assuming $\gamma_{x,y} \in C^2([0,1])$, this is simply the integral formulation of the boundary value problem:
\begin{align*}
\begin{cases}
\ddot{\gamma}(t)= -2\alpha(\gamma(t)-\xi(t)), \\
\gamma(0) = x, \\
\gamma(1) = y.
\end{cases}
\end{align*}

Now, we establish the contraction mapping property of $\mathcal{F}$. In fact, we will show more, we will prove that for any two pair of curves, say $\gamma_{1}, \gamma_{2}\in \Omega_{x,y}$, we get the bound
\begin{align*}
\|\mathcal{F}(\gamma_{1})-\mathcal{F}(\gamma_{2})\|_{\infty}&\leq |\gamma_{1}(0)-\gamma_{2}(0)|+|\gamma_{1}(1)-\gamma_{2}(1)|+\frac{\alpha}{4}\|\gamma_{1}(t)-\gamma_{2}(t)\|_{\infty}.
\end{align*}

To show this assertion, we compute directly from the formula of $\mathcal{F}$ and the triangle inequality, along with the inequalities $0\leq 1-t\leq 1$ for all $0\leq t \leq 1$. Then
\begin{align*}
    \left|\mathcal{F}(\gamma_1)(t)
-\mathcal{F}(\gamma_2)(t)\right|&\leq|\gamma_{1}(0)-\gamma_{2}(0)|+|\gamma_{1}(1)-\gamma_{2}(1)|\\
&+2\alpha\int_{0}^{1}(1-s)\left|\gamma_{1}(s)-\gamma_{2}(s)\right|ds+2\alpha\int_{0}^{t}(t-s)\left|\gamma_{1}(s)-\gamma_{2}(s)\right|ds\\
&\leq |\gamma_{1}(0)-\gamma_{2}(0)|+|\gamma_{1}(1)-\gamma_{2}(1)|\\
&+2\alpha\int_{0}^{1}(1-s)\left|\gamma_{1}(s)-\gamma_{2}(s)\right|ds+2\alpha\int_{0}^{t}(1-s)\left|\gamma_{1}(s)-\gamma_{2}(s)\right|ds.\\
\end{align*}
Notice that the $\xi$'s canceled in the above inequality as $\mathcal{F}(\cdot)$ is merely a functional on $\gamma$'s.
Moreover, the last inequality comes from the integral bound $\int_{0}^{t}(t-s)|u(s)|ds\leq\int_{0}^{t}(1-s)|u(s)|ds$.
Changing the domain of integration, we rewrite the integral expressions of the latter inequality into one integral term
\[
\int_0^1 K(t,s)\left(\gamma_{1}(s)-\gamma_{2}(s)\right)\,ds,
\]
where
\[
K(t,s)
=
\begin{cases}
s(1-t), & 0\leq s\leq t,\\
t(1-s), & t\leq s\leq 1.
\end{cases}
\]
Therefore
\[
\int_0^1 K(t,s)\,ds
=
\int_0^t s(1-t)\,ds
+
\int_t^1 t(1-s)\,ds
=
\frac{t(1-t)}{2}.
\]
Incorporating the above integral calculation into the inequality above yields,
\begin{align*}
\left|\mathcal{F}(\gamma_1)(t)
-\mathcal{F}(\gamma_2)(t)\right|
&\leq |\gamma_{1}(0)-\gamma_{2}(0)|+|\gamma_{1}(1)-\gamma_{2}(1)|+\alpha t(1-t)\|\gamma_{1}(s)-\gamma_{2}(s)\|_{\infty}
\end{align*}
So since $t(1-t)\leq 1/4$, we obtain the following bound
\begin{align}\label{inq: contraction}
\begin{split}
\|\mathcal{F}(\gamma_1)
-\mathcal{F}(\gamma_2)\|_\infty
&\leq |\gamma_{1}(0)-\gamma_{2}|+|\gamma_{1}(1)-\gamma_{2}(1)|+\frac{\alpha}{4}\|\gamma_{1}(s)-\gamma_{2}(s)\|_\infty.
\end{split}
\end{align}
Furthermore, applying the fixed point $\gamma_{1}=\gamma_{x_{1},y_{1}}$ and $\gamma_{2}=\gamma_{x_{2}, y_{2}}$, to the left-hand side of the bound \eqref{inq: contraction} and simplifying, we get that
\begin{align*}
\|\gamma_{1}(t)-\gamma_{2}(t)\|_{\infty}&\leq\frac{1}{1-\frac{\alpha}{4}}\left(|\gamma_{1}(0)-\gamma_{2}(0)|+|\gamma_{1}(1)-\gamma_{2}(1)|\right)\\
&\leq \frac{1}{1-\frac{\alpha}{3}}\left(|\gamma_{1}(0)-\gamma_{2}(0)|+|\gamma_{1}(1)-\gamma_{2}(1)|\right)\\
&\leq \frac{1}{1-\frac{\alpha}{3}}\left(|x_{1}-x_{2}|+|y_{1}-y_{2}|\right).
\end{align*}
since $\alpha/4\leq \alpha/3$. Thusly we may take, and we do, $L=\frac{\alpha}{3}$. Therefore $\mathcal{F}$ is a contraction whenever $\frac{\alpha}{3}<1$, that is, whenever $\alpha<3$. Identical arguments hold for the maximal curve $\xi$ as the Euler-Lagrange equations applied to the functional in Condition \ref{condition 2} for geodesic $\xi$ satisfies $\ddot{\xi}(t)=-2\alpha(\xi(t)-\gamma(t))$, with $\xi(0)=x'$ and $\xi(1)=y'$.
\end{proof}

The standard theory of optimal transport suggests that for the minimax bilinear transport theory, we can consider an action functional. Indeed, for paths $\gamma:[0,1]\to X$ and $\xi:[0,1]\to Y$ having initial and final points, $\gamma(0)=x, \gamma(1)=y$ and $\xi(0)=x^{\prime}, \xi(1)=y^{\prime}$, the \textbf{end-point cost function} in this case is given by,
\begin{align}\label{qot end-point cost}
    c_{e}(x,y,x^{\prime}, y^{\prime}):=\inf_{\gamma}\sup_{\xi}\left\{c(\gamma, \xi):\gamma(0)=x, \gamma(1)=y;\;\xi(0)=x^{\prime}, \xi(1)=y^{\prime}\right\}.
\end{align}
If the pair $\gamma, \xi$ appears in the pair of optimal plans $(\pi_{1}^{*},\pi_{2}^{*})$ satisfying (\ref{inf sup problem}), then we could expect that
\[
c(\gamma,\xi)=c_{e}(\gamma(0),\gamma(1),\xi(0),\xi(1)).
\]
Henceforth, such paths will be called $c$-\textit{minimaximal}. Note that this is merely a heuristic. Section \ref{sec: effective cost and cyclical monotonicity} contains a rigorous account of this notion.

In order to discuss the \textit{twist} condition of (\ref{qot end-point cost}), to ensure that the gradient of the cost $c_{e}$ is one-to-one and apply the standard theory of optimal transport, we draw Lemma \ref{twist condition}  from \cite{cabrera2022optimaltransportationprincipleinteracting}. Indeed this lemma is concerned with the \textit{twist} condition on $c_{e}$. Namely, that $\nabla c_{e}$ is injective in $(x,x^{\prime})$. 

{\begin{lemma}\label{twist condition}
Let $V$ be the interaction from Condition \ref{condition 2} with $\alpha \neq 3$. For all $y, y^{\prime} \in \mathbb{R}^{d}$, and $x_{1}\neq x_{2}$ and $x_{1}^{\prime}\neq x_{2}^{\prime}$, we have 
\begin{align*}
\begin{cases}
    \nabla_{y}c_{\textbf{e}}(x_{1},y,x^{\prime},y^{\prime})\neq\nabla_{y}c_{\textbf{e}}(x_{2}, y, x^{\prime},y^{\prime}) &\forall (x',y') \in \mathbb{R}^{2d},\\
    \nabla_{y^{\prime}}c_{\textbf{e}}(x,y,x_{1}^{\prime},y^{\prime})\neq\nabla_{y^{\prime}}c_{\textbf{e}}(x, y, x_{2}^{\prime},y^{\prime}) &\forall (x,y) \in \mathbb{R}^{2d}.
\end{cases}
\end{align*}
\end{lemma}
\begin{proof} By the lemma in \cite[Lemma 2.5]{cabrera2022optimaltransportationprincipleinteracting}  
we have
\[
\nabla_{y}c_{e}(x,y,x', y')=y-x+2\alpha\int_{0}^{1}t(\gamma(t)-\xi(t))dt.
\] 
Let $(\gamma_{x,y}(t), \xi_{x',y'}(t))$ be the \textit{minimaximal} path which solve the Euler-Lagrange equations for $c_e$. By Proposition \ref{minimal paths plus potential} we have
\begin{gather*}
    d(t) = \gamma(t) - \xi(t) = (1-t)(x-x') + t(y-y').
\end{gather*}
In particular, this implies
\begin{gather*}
    2\alpha \int_0^1 td(t)dt = \frac{\alpha}{3}(x-x') + \frac{2\alpha}{3}(y-y').
\end{gather*}Then for all $y$, $x_{1}\neq x_{2}$, we acquire the following
\begin{align*}
    \nabla_{y}c_{e}(x_{1},y,x', y')-\nabla_{y}c_{e}(x_{2},y,x', y')&=x_{2}-x_{1}+ \frac{\alpha}{3}(x_1-x_2)\\
    &= \left|1 - \frac{\alpha}{3}\right|\left|x_2 - x_1\right| \\
    &> 0
\end{align*}
provided $\alpha \neq 3$. A symmetric argument holds for taking the gradient with respect to $y'$.
\end{proof}

\subsection{Existence to Bilinear Minimax Transport Problem}\label{existence to biliner minimax}
We are interested in the existence and uniqueness to \eqref{inf sup problem} (but not necessarily a NETP). Before we prove the theorem, let us first show that the function given by Condition \ref{condition 2} satisfies Condition \ref{condition 1}. To do this, we first show that it is lower and upper semicontinuous on $\Omega_1$ and $\Omega_2$ respectively.

\begin{proposition}\label{prop: lo/up continuity}
    Given the cost function $c(\gamma,\xi) = \int_0^1 \frac12|\dot{\gamma}(t)|^2 - \frac12|\dot{\xi}(t)|^2 + \alpha |\gamma(t) - \xi(t)|^2\ dt$, then $\gamma \mapsto c(\gamma, \xi)$ is lower semicontinuous for all $\xi \in \Omega_2$ and $\xi \mapsto c(\gamma, \xi)$ is upper semicontinuous for all $\gamma \in \Omega_1$.
\end{proposition}

\begin{proof} Firstly, notice that the quadratic interaction term, $\alpha|\gamma-\xi|^{2}$ is fully continuous in $\gamma$ and $\xi$. Fix $\xi\in \Omega_{2}$. Furthermore, the spatial derivatives of $\alpha|\gamma-\xi|^{2}$ were computed in \eqref{eqn: separable cost} and are bounded uniformly in $t$. Thus, assuming $\gamma_{n}$ is a sequence in $\Omega_{1}$, and $\gamma_{*}\in \Omega_{1}$ such that $\|\gamma_{n}-\gamma_{*}\|\to 0$ as $n\to \infty$, then
\[
\int_{0}^{1}\left|\alpha|\gamma_{n}-\xi|^{2}-\alpha|\gamma_{*}-\xi|^{2}\right|dt\leq 4\alpha|\gamma_{n}-\gamma_{*}|.
\]
Therefore, 
\[
\lim_{n\to \infty}\int_{0}^{1}\alpha|\gamma_{n}-\xi|^{2}dt= \int_0^1 \alpha|\gamma_{*}-\xi|^{2}dt.
\]
In \cite{cabrera2022optimaltransportationprincipleinteracting}, it is shown that the kinetic term is lower semicontinuous with respect to $\gamma$. Consequently, $c(\gamma, \cdot)$ is lower semicontinuous.

Similarly, fixing $\gamma$, the kinetic term in $\xi$ is upper semicontinuous with respect to $\xi$. Therefore, $c(\cdot, \xi)$ is upper semicontinuous with respect to $\xi$.

\end{proof}

The next proposition ensures the cost function defined in Condition \ref{condition 2} is coercive. Or that Condition \ref{condition 2} implies Condition \ref{condition 1}. Recall that ``$\rightharpoonup$" stands for weak convergence: A sequence $\{\mu_{n}\}\subset \mathcal{P}(X)$ \emph{weakly converges} to $\mu$, denoted $\mu_{n}\rightharpoonup\mu$, provided $\int_{X}f(x)d\mu_{n}(x)\to \int_{X}f(x)d\mu(x)$ for all $f\in C_{b}(X)$, of course for $C_{b}(X)$ denoting the space of continuous and bounded functions on $X$.

\begin{proposition}\label{prop:cond2impliescond1}
    The cost function $c(\gamma,\xi) = \int_0^1 \frac12|\dot{\gamma}(t)|^2 - \frac12|\dot{\xi}(t)|^2 + \alpha |\gamma(t) - \xi(t)|^2\ dt$ satisfies Condition \ref{condition 1}.
\end{proposition}

\begin{proof}
By Proposition \ref{prop: lo/up continuity}, it suffices to prove sequential compactness of $\Omega_M$. Since $\Omega$ is a metric space, this implies compactness.

Let $\{(\gamma_n,\xi_n)\}_n\subset\Omega_M$. By definition of $\Omega$, there exist $K,L>0$ such that
\begin{equation}\label{eq:W22bounds_patched}
\int_0^1|\dot\gamma_n|^2\,dt\leq K,
\qquad
\int_0^1|\dot\xi_n|^2\,dt\leq L,
\quad\forall n.
\end{equation}
In particular, $\|\dot\gamma_n\|_2\leq K^{1/2}$ and $\|\dot\xi_n\|_2\leq L^{1/2}$.

\medskip
We now establish the uniform equicontinuity and boundedness of the sequences.
Fix $0\le t_2<t_1\le 1$. By Cauchy--Schwarz,
\[
|\gamma_n(t_1)-\gamma_n(t_2)|
\leq \int_{t_2}^{t_1}|\dot\gamma_n(t)|\,dt
\leq \|\dot\gamma_n\|_2\,|t_1-t_2|^{1/2}
\leq K^{1/2}|t_1-t_2|^{1/2},
\]
and similarly
\[
|\xi_n(t_1)-\xi_n(t_2)|
\leq \|\dot\xi_n\|_2\,|t_1-t_2|^{1/2}
\leq L^{1/2}|t_1-t_2|^{1/2}.
\]
Thus $\{\gamma_n\}$ and $\{\xi_n\}$ are equicontinuous.

Moreover, since $\sqrt{|\gamma_n(0)|^2+|\xi_n(0)|^2}\leq M$, for all $n\geq 0$, for some constant $M>0$, we have the following bounds $|\gamma_n(0)|,|\xi_n(0)|\leq M$; and hence for any $t\in[0,1]$,

\begin{align}\label{inq: gamma dot}
\begin{split}
&|\gamma_n(t)|\leq |\gamma_n(0)|+\int_0^1|\dot\gamma_n|\ dt
\leq M+\|\dot\gamma_n\|_2
\leq M+K^{1/2},\\
&|\xi_n(t)|\leq |\xi_n(0)|+\int_0^1|\dot\xi_n|\ dt
\leq M+\|\dot\xi_n\|_2\leq M+L^{1/2}
\end{split}
\end{align}
 It follows that
\begin{gather}
    \|(\gamma_n,\xi_n)\|_\Omega \leq C
\end{gather}
where $C$ comes from applying the bounds in Condition \ref{condition 1} and \eqref{inq: gamma dot}.
Therefore $\{(\gamma_n,\xi_n)\}$ is uniformly bounded and equicontinuous in $\Omega$.

Applying the Arzel\`a--Ascoli theorem, there exists a subsequence (not relabeled) and $(\gamma,\xi)\in\Omega$ such that
\begin{equation}\label{eq:uniform_conv_patched}
\gamma_n\to\gamma,\qquad \xi_n\to\xi
\qquad \text{uniformly on }[0,1].
\end{equation}

\medskip
Next we show strong $L^2$ convergence of $\dot\gamma_n$ and $\dot\xi_n$.
From \eqref{eq:W22bounds_patched}, the components of the sequences $\{\dot\gamma_n\}$ and $\{\dot\xi_n\}$ are bounded in $W^{1,2}([0,1])$. According to \cite[Theorem 9.16]{brezis2011functional}
by the compactness of the embedding $W^{1,2}([0,1])\hookrightarrow L^2([0,1])$, there exist subsequences  $\gamma_{n_{k}}$ and $\xi_{n_{k}}$ such that
\begin{equation}\label{eq:strong_derivs_patched}
\dot\gamma_{n_{k}}\to \dot\gamma \quad\text{strongly in }L^2([0,1]),
\qquad
\dot\xi_{n_{k}}\to \dot\xi \quad\text{strongly in }L^2([0,1]).
\end{equation}
Hence, $\{(\dot{\gamma}_{n_k}, \dot{\xi}_{n_k})\}_{k}$ is uniformly integrable in $L^1([0,1])$. 

However, even after obtaining strong convergence, it is not clear that the limit
is actually the derivative of some curve $\gamma$. What we ultimately
obtain is the following: if $\dot{\gamma}^{\,i}_{n_k}$ denotes the $i$-th
component of $\dot{\gamma}_{n_k}$, then there exists a subsequence (not relabeled)
and a function $f_i \in L^2([0,1])$ such that,
\[
\dot{\gamma}^{\,i}_{n_k} \rightharpoonup f_i
\quad \text{in } L^2([0,1]).
\]

Next we show that for some $f=(f_{1},\dots, f_{d})$, $f$ is the (time) derivative of a curve $\gamma$. Let $\dot{\gamma}_{n_k}^i$ and $\dot{\xi}_{n_k}^i$ denote the $i$th components of $\dot{\gamma}_{n_k}$ and $\dot{\xi}_{n_k}$, respectively. Since $\dot{\gamma}_{n_k}^i ,\dot{\xi}_{n_k}^i \in L^2([0,1])$ are uniformly bounded, then by the Banach-Alaoglu theorem we have that up to a subsequence, there exist $f_i,g_i \in L^2[0,1]$ such that $\dot{\gamma}_{n_k}^i \rightharpoonup f_i ,\dot{\xi}_{n_k}^i \rightharpoonup g_i$. This implies that for any test function $\phi \in C_c^\infty([0,1]; \mathbb{R}^{d})$, we have
    \begin{gather*}
        \int_0^1 \dot{\gamma}_{n_k}^i(t)\phi(t)dt \to \int_0^1 f_i(t)\phi(t)dt.
    \end{gather*}
Similarly for $\{\dot{\xi}_{n_k}^i\}$. Note that if we integrate by parts on the left-hand side we obtain
    \begin{gather*}
        \int_0^1 \dot{\gamma}_{n_k}^i(t)\phi(t)dt  = - \int_0^1 \gamma_{n_k}^i(t)\dot{\phi}(t)dt,\quad \forall \phi \in C_c^\infty([0,1]; \mathbb{R}^{d}).
    \end{gather*}
From uniform convergence of $\gamma_{n_k} \to \gamma_0$ and applying integration by parts again we have that
    \begin{gather*}
        - \int_0^1 \gamma_{n_k}^i(t)\dot{\phi}(t)dt \to - \int_0^1 \gamma_0^i(t)\dot{\phi}(t)dt = \int_0^1 \dot{\gamma}_0^i(t)\phi(t)dt,\quad \forall \phi \in C_c^\infty([0,1]; \mathbb{R}^{d})
    \end{gather*}
    It follows that $f_i = \dot{\gamma}_0^i$ in distribution.

    Next, since the path $\gamma_0$ is absolutely continuous, then by the fundamental theorem of calculus we have on the one hand
    \begin{gather*}
        \gamma_0(s) = \gamma_0(0) + \int_0^s \dot{\gamma}_0(t)\ dt.
    \end{gather*}
    From the convergence, we have
    \begin{gather*}
        \int_0^1 \dot{\gamma}_{n_k}(t)g(t)dt \to \int_0^s f(t)g(t)dt\quad \forall g \in [L^2([0,1])]^* \cong L^2([0,1])
    \end{gather*}
    Choosing $g(t) = \chi_{[0,t]}$ yields
    \begin{gather*}
        \int_0^s \dot{\gamma}_{n_k}(t)dt \to \int_0^s f(t)dt.
    \end{gather*}
    Thus, by weak convergence we have
    \begin{gather*}
        \gamma_{n_k}(s) - \gamma_{n_k}(0) = \int_0^s \dot{\gamma}_{n_k}(t) \to \int_0^s f(t).
    \end{gather*}
    On the other hand, uniform convergences gives that the left-hand side converges and
    \begin{gather*}
        \gamma_0(s) - \gamma_0(0) = \int_0^s f(t)dt.
    \end{gather*}
    Therefore,
    \begin{gather*}
        \gamma_0(0) + \int_0^s f(t)dt = \gamma_0(0) + \int_0^s \dot{\gamma}_0(t)dt.
    \end{gather*}
    By the Lebesgue differentiation theorem, we conclude $\dot{\gamma}_0 = f$ a.e. We repeat the same argument for $\dot{\xi}$.

Armed with the above argument and combining it with \eqref{eq:uniform_conv_patched} with
$\gamma_n(s)-\gamma_n(0)=\int_0^s\dot\gamma_n$ and 
$\xi_n(s)-\xi_n(0)=\int_0^s\dot\xi_n$, we
see that $\gamma,\xi$ are absolutely continuous and their derivatives are $\dot\gamma,\dot\xi$ a.e.

\medskip
Finally we establish closedness of the constraint $|c|\leq M$.
By \eqref{eq:strong_derivs_patched},
\[
\int_0^1 \frac12|\dot\gamma_n|^2\,dt \to \int_0^1 \frac12|\dot\gamma|^2\,dt,
\qquad
\int_0^1 \frac12|\dot\xi_n|^2\,dt \to \int_0^1 \frac12|\dot\xi|^2\,dt.
\]
Also, by uniform convergence \eqref{eq:uniform_conv_patched}, we have $\gamma_n-\xi_n\to\gamma-\xi$ uniformly, hence in $L^2$, so
\[
\int_0^1 \alpha|\gamma_n-\xi_n|^2\,dt \to \int_0^1 \alpha|\gamma-\xi|^2\,dt.
\]
Therefore,
\[
c(\gamma_n,\xi_n)\to c(\gamma,\xi).
\]
Since $|c(\gamma_n,\xi_n)|\le M$ for all $n$, taking limits yields $|c(\gamma,\xi)|\le M$.
The endpoint constraint passes to the limit by \eqref{eq:uniform_conv_patched}. Hence $(\gamma,\xi)\in\Omega_M$.

Thus every sequence in $\Omega_M$ has a uniformly convergent subsequence whose limit lies in $\Omega_M$. Therefore $\Omega_M$ is sequentially compact in $\Omega$.
\end{proof}

The following proposition ensures that for any $\mu ,\nu \in \mathcal{P}(X)$, then $\Pi(\mu,\nu) \subset \mathcal{P}(X \times X)$ is weak-* closed.
\begin{proposition}\label{prop:limit}
    Let $\mu,\nu \in \mathcal{P}(X)$. Suppose $\{\pi_k\} \subset \Pi(\mu,\nu)$ such that $\pi_k \rightharpoonup \pi$. Then $\pi \in \Pi(\mu,\nu)$.
\end{proposition}
\begin{proof}
    The proof can be found in \cite{ambrosio2013user}. For the convenience of the reader, we supply the proof. For any $\eta \in \Pi(\mu,\nu)$, let $e_t$ be the evaluation map such that $(e_0)_{\#}\eta = \mu$ and $(e_1)_{\#}\eta = \nu$. Then for any test function $f \in C_b(X)$ we have by the change of variables
    \begin{align*}
        \int_{X} f(x)\ d((e_0)_{\#}\pi(x)) = \int_{\Omega_{1}} f(\gamma(0))\ d\pi(\gamma)
        &= \lim_{k \to\infty} \int_{\Omega_{1}} f(\gamma(0))\ d\pi_k(\gamma)\\
        &= \lim_{k\to \infty} \int_{X} f(x)\ d((e_0)_{\#}(\pi_k)(x))\\
        &= \int_{X} f(x)\ d\mu(x).
    \end{align*}
    Since this holds for every test function we must have $(e_0)_\#\pi = \mu$. Using the same proof, we can also show that $(e_1)_\#\pi = \nu$.
\end{proof}

The next result will be crucial for determining that the minimax value of \eqref{sigma functional} coincides with the maximin value of \eqref{sigma functional}. In particular, it is essential to prove Lemma \ref{lemma: uplo continuity}, which is essential to prove the reverse inequality of Proposition \ref{prop: weak duality}, and which both rely on the following proposition. 
\begin{lemma}\label{lemma:limitproduct}
    Suppose $\mathcal{X},\mathcal{Y}$ are polish spaces and $\mu_1,\nu_1\in \mathcal{P}(\mathcal{X})$ and $\mu_2,\nu_2 \in \mathcal{P}(\mathcal{Y})$. Let $\{\pi_{1,i}\}_{i\geq 0}\in \Pi(\mu_1,\nu_1)$ and $\{\pi_{2,i}\}_{i\geq 0}\in \Pi(\mu_2,\nu_2)$ be sequences. If $\pi_{1,i}\otimes \pi_{2,i} \rightharpoonup \rho$, then $\pi_{1,i} \rightharpoonup \pi_1$, $\pi_{2,i} \rightharpoonup \pi_2$ and $\rho = \pi_1 \otimes \pi_2$. 
\end{lemma}
\begin{proof}
Let $\rho_i = \pi_{1,i}\otimes \pi_{2,i}$ and define $\pi_{1} = (\text{proj}_1)_\#\rho$, $\pi_{2} = (\text{proj}_2)_\#\rho$. Clearly, we have $\pi_{j,i}:=\left(\text{proj}_{j}\right)_{\#}\rho_{i}$ for $j=1,2$. To show that $\pi_{1,i} \rightharpoonup \pi_{1}$, we note that for any $f \in C_b(\mathcal{X})$
\begin{align*}
    \int f(x)\ d\pi_{1,i}(x,y) &= \int f(x)\ d(\text{proj}_1)\#\rho_i(x,y) \\
    &= \int f\circ (\text{proj}_1^{-1}(x,y))\ d\rho_i(x,y) \\
    &\to \int f\circ(\text{proj}_1^{-1}(x,y))\ d\rho(x,y)\\
    &= \int f(x)\ d(\text{proj}_1\#(\rho))\\
    &= \int f(x)\ d\pi_1(x,y)
\end{align*}
where the convergence follows as $f \in C_b(\mathcal{X})$ implies $f \in C_b(\mathcal{X}\times \mathcal{Y})$. A similar argument holds for $\pi_{2,i} \rightharpoonup \pi_2$.

Since $\pi_{1,i}\rightharpoonup \pi_{1}\in \mathcal{P}(\mathcal{X})$ and $\pi_{2,i}\rightharpoonup \pi_{2}\in \mathcal{P}(\mathcal{Y})$, take any $ F\in C_{b}(\mathcal{X}\times \mathcal{Y})$. We shall show
\[
\int Fd(\pi_{1,i}\otimes \pi_{2,i})\to \int F d(\pi_{1}\otimes \pi_{2}).
\]
As $\pi_{1,i}\rightharpoonup \pi_{1}$, then the collection $\{\pi_{1,i}\}\cup\{\pi_{1}\}$ is tight. For every $\varepsilon>0$ there exists a $K_{1}\subset \mathcal{X}$ such that $\sup_{i}\pi_{1,i}\left[K_{1}^{c}\right]\leq \varepsilon$  and $\pi_{1,i}\left[K_{1}^{c}\right]\leq \varepsilon$. Similarly choose $K_{2}\subset \mathcal{Y}$ such that $\sup_{i}\pi_{2,i}\left[K_{2}^{c}\right]\leq \varepsilon$  and $\pi_{2,i}\left[K_{2}^{c}\right]\leq \varepsilon$.

Then for $K:=K_{1}\times K_{2}$, we get for all $i$
\[
(\pi_{1,i}\otimes \pi_{2,i})\left[K^{c}\right]\leq \pi_{1,i}\left[K_{1}^{c}\right]+\pi_{2,i}\left[K_{2}^{c}\right]\leq 2\varepsilon.
\]
Similarly, $(\pi_{1}\otimes \pi_{2})\left[K^{c}\right]\leq 2\varepsilon.$

Now, we aim to use separable functions to approximate $F$ uniformly on $K$. On compact sets $K:=K_{1}\times K_{2}$, the algebra
\[
\mathcal{A}:=\left\{\sum_{n=1}^{N}f_{n}(x)g_{n}(y): f_{n}\in C(K_{1}),\; g_{n}\in C(K_{2})\right\}
\]
is uniformly dense in $C(K)$ by the Stone-Weierstrass theorem. Hence, choose $R(x,y)=\sum_{n=1}^{N}f_{n}(x)g_{n}(y)$ such that 
\begin{align}\label{inq: SW}
\sup_{(x,y)\in K} \left|F(x,y)-R(x,y)\right|\leq\varepsilon.
\end{align}
Furthermore, we can extend each $f_{n}, g_{n}$ to continuous  bounded functions on $\mathcal{X}, \mathcal{Y}$.

Next, for each product $f(x)g(y)$, we have
\begin{align*}
    \iint f(x)g(y)d((\pi_{1,i}\otimes\pi_{2,i})(x,y))&=\left(\int f(x)d\pi_{1,i}(x,y)\right)\left(\int g(y) d\pi_{2,i}(x,y)\right)\\
    &\longrightarrow \left(\int f (x)d\pi_{1}(x,y)\right)\left(\int g(y) d\pi_{2}(x,y)\right)\\
    &=\iint f(x)g(y)d((\pi_{1}\otimes \pi_{2})(x,y)),
\end{align*}
since $\pi_{1,i}\rightharpoonup \pi_{1}$ and $\pi_{2,i}\rightharpoonup \pi_{2}$. Therefore,
\[
\iint R(x,y)d((\pi_{1,i}\otimes \pi_{2,i})(x,y))\to \iint R(x,y)d((\pi_{1}\otimes \pi_{2})(x,y)).
\]

Then for any probability measure $\varphi \in P(\mathcal{X}\times \mathcal{Y})$, we have $\int F d\varphi=\int_{K}Fd\varphi+\int_{K^{c}}Fd\varphi$, then we have 

\begin{align*}
    \left|\int Fd\varphi -\int F\chi_{K}d\varphi\right|&\leq \int |F|\chi_{K^{c}}d\varphi\leq \|F\|_\infty\int \chi_{K^{{c}}}d\varphi=\|F\|_{\infty}\varphi[K^{c}].
\end{align*}

Then applying the above to $\pi_{1,i}\otimes \pi_{2,i}$ and $\pi_1\otimes \pi_2$ in place of $\varphi$ we acquire the two estimates
\begin{align}\label{inq: F1}
    \left|\int Fd(\pi_{1,i}\otimes\pi_{2,i})-\int F\chi_{K}d(\pi_{1,i}\otimes\pi_{2,i})\right|\leq 2\|F\|_{\infty}\varepsilon,
\end{align}

\begin{align}\label{inq: F2}
    \left|\int Fd(\pi_{1}\otimes\pi_{2})-\int F\chi_{K}d(\pi_{1}\otimes\pi_{2})\right|\leq 2\|F\|_{\infty}\varepsilon.
\end{align}
Now, on $K$, 
\begin{align}\label{inq: FR}
    \left|\int F\chi_{K}d(\pi_{1,i}\otimes\pi_{2,i})-\int R\chi_{K}d(\pi_{1,i}\otimes\pi_{2,i})\right|\leq \varepsilon.
\end{align}

Finally, we have by the triangle inequality, with $\varphi_{i}:=\pi_{1,i}\otimes \pi_{2,i}$ and $\varphi = \pi_{1}\otimes \pi_{2}$,

\begin{align*}
\left|\int F\,d\varphi_i - \int F\,d\varphi\right|
&\leq 
\left|\int F\,d\varphi_i - \int F\chi_{K}\,d\varphi_i\right|+
\left|\int F\chi_{K}\,d\varphi_i - \int R\chi_{K}\,d\varphi_i\right|\\
&+
\left|\int R\chi_{K}\,d\varphi_i - \int R\chi_{K}\,d\varphi\right|
+
\left|\int R\chi_{K}\,d\varphi - \int F\chi_{K}\,d\varphi\right|\\
&+
\left|\int F\chi_{K}\,d\varphi - \int F\,d\varphi\right|.
\end{align*}
The first two terms on the right-hand side of the above inequality are bounded above by \eqref{inq: F1} and \eqref{inq: FR}, while the third term on the right-hand side of the above inequality has the convergence of separable functions in the integrand, namely,
\[
\left|\int R\chi_{K}d\varphi_{i}-\int R\chi_{K}d\varphi\right|\to 0,\quad\text{as}\; i\to \infty;
\]
and the fourth and fifth terms in the above inequality stem from the inequalities \eqref{inq: SW} and \eqref{inq: F2}, respectively.

Consequently, putting these inequalities together yields

\[
\limsup_{i\to \infty}\left|\int Fd(\pi_{1,i}\otimes \pi_{2,i})-\int Fd(\pi_{1}\otimes\pi_{2})\right|\leq (4\|F\|_{\infty}+2)\varepsilon.
\]
Since $\varepsilon>0$ was arbitrary, the limsup goes to $0$, and so $(\pi_{1,i}\otimes \pi_{2,i})\rightharpoonup (\pi_{1}\otimes \pi_{2})$. By assumption, since $\pi_{1,i}\otimes \pi_{2,i} \rightharpoonup \rho$, then by uniqueness of limits we have $\rho = \pi_{1}\otimes \pi_{2}$.
\end{proof}

\begin{proof}
    [\textit{Proof of Theorem \ref{thm:minimaxExistence}}] We begin by fixing $\pi_1 \in \Pi_1(\mu_1,\nu_1)$, as defined on \eqref{path plans}. Note that we may rewrite the optimization of the cost function as
    \begin{align*}
        \inf_{\pi_1\in \Pi_1(\mu_1,\nu_1)}\sup_{\pi_{2}\in \Pi_2(\mu_2,\nu_2)} \int_{\Omega_2}c(\gamma,\xi)\ d\pi_2& = \inf_{\pi_1\in \Pi_1(\mu_1,\nu_1)}\left[-\inf_{\pi_2\in \Pi_2(\mu_2,\nu_2)} \left(-\int_{\Omega_2}c(\gamma,\xi)\ d\pi_2\right)\right]\\
        &= -\sup_{\pi_1\in \Pi_1(\mu_1,\nu_1)}\inf_{\pi_2\in\Pi_2(\mu_2,\nu_2)}\int_{\Omega_2} p(\gamma,\xi)\ d\pi_2,
    \end{align*}
    where $p(\gamma, \xi):= -c(\gamma, \xi)$.  Note that Condition \ref{condition 1} implies there exists a constant $K > 0$ such that the map satisfies
    \begin{gather}\label{ineq:lowerbound}
        p(\gamma, \xi) \geq -K.
    \end{gather}
    
For fixed $\pi_{1}\in \mathcal{P}(\Omega_{1})$, define functional $\hat{\sigma}_{0}:\mathcal{P}(\Omega_{2})\to \mathbb{R}$ by 
\[\widehat{\sigma}_{0}(\pi_{2}):=\widehat{\sigma}_0(\pi_2)[\pi_{1}]: = \int_{\Omega_1}\int_{\Omega_{2}} p(\gamma,\xi)+K\ d\pi_2(\xi)\;d\pi_{1}(\gamma) .\] 

Thus, we start by exhibiting existence of a minimizer to this problem. It follows that for any $\pi_2 \in \mathcal{P}(\Omega_2)$ then
    \begin{gather*}
        \widehat{\sigma}_0(\pi_2) \geq 0,\quad \forall \pi_1 \in \mathcal{P}(\Omega_1).
    \end{gather*}
 If $\widehat{\sigma}_{0}(\pi_{2})=+\infty$ for all admissible $\pi_2\in\mathcal{P}\left(\Omega_{2}\right)$, then in this case any admissible measure trivially minimizes $\widehat{\sigma}_{0}$. Hence, we may assume without loss of generality that $\widehat{\sigma}_{0}(\pi_2)<+\infty$ for at least some admissible $\pi_2 $ for all $\pi_1 \in \mathcal{P}(\Omega_1)$. Then this implies $\inf_{\pi_2 \in \Pi_2(\mu_2,\nu_2)} \hat{\sigma}_0(\pi_2)$ is finite. Let $\{\tau_k\}_k$ be a minimizing sequence such that $\sigma_0(\pi_1,\tau_k) \to \inf_{\pi_2 \in \Pi_2(\mu_2,\nu_2)} \hat{\sigma}_0(\pi_2)$. Observe that there exists $C > 0$ such that
    \begin{gather*}
        \hat{\sigma}_0(\tau_k) \leq C,\quad \forall k \geq 1.
    \end{gather*}

    From here on end, we apply the same reasoning that was done on the proof of \textbf{Theorem 1.4} in \cite{cabrera2022optimaltransportationprincipleinteracting}, but to our setting as follows. Recall Condition \eqref{condition 1} applied to $p$. That is, given any real number $\ell>0$, we have that the set $\widehat{\Omega}_{\ell}$ under fixed $\pi_1$ is given by
    \begin{gather*}
       \widehat{\Omega}_{\ell}:= \left\{\omega \in \Omega_2: \exists \gamma \in spt(\pi_1), \sqrt{|\gamma(0)|^{2}+|\omega(0)|^{2}} \leq \ell,\ |p(\gamma, \omega)| \leq \ell\right\}.
    \end{gather*}
    Observe that this set is compact. Indeed, since $spt(\pi_1)$ is compact, then $spt(\pi_1)\times \Omega_2$ is a closed subset of the compact set $\Omega_1\times \Omega_2$, hence compact. Next, we have that $\widehat{\Omega}_\ell = \text{proj}_2(spt(\pi_1)\times \Omega_2)$, and since the  continuous image of a compact set is compact, we conclude. 
    
    Suppose that $\ell>0$ is chosen large enough such that $\textit{spt}\;\mu_{1}\subset B_{\ell}(0)$ and $\textit{spt}\;\mu_{2}\subset B_{\ell}(0)$, then  
    \[\textit{spt}\;(\mu_{1}\otimes\mu_{2})\subset \textit{spt}\;\mu_{1}\times\textit{spt}\;\mu_{2}\subset B_{\ell}(0)\times B_{\ell}(0)
    \]
    Define the set $X_{\ell} := \{\omega \in \Omega_2: \exists \gamma \in spt(\pi_1)\ s.t.\ |p(\gamma, \omega)| \leq \ell\}.$
    As $(e_{0},e_{0})_{\sharp}(\pi_{1}\otimes\pi_{2})=\mu_{1}\otimes\mu_{2}$, then $\textit{spt}\;(\pi_{1}\otimes\pi_{2})\subset (e_{0},e_{0})^{-1}\left(\textit{spt}\;(\mu_{1}\otimes\mu_{2})\right)$. So since $\pi_{1}\otimes\pi_{2}$ is supported in the set $ \Omega_{1}\times\hat{\Omega}_{\ell}$, then for every admissible $\pi_2$ we have
    \begin{align*}
       \pi_1 \otimes \pi_2(\Omega_{1}\times X_{\ell}) &= \pi_1 \otimes\pi_2\left(\Omega_{1}\times\widehat{\Omega}_{\ell}\right)\\
       \pi_1(\Omega_{1}) \cdot\pi_2(X_{\ell}) &= \pi_1 (\Omega_{1})\cdot\pi_2\left(\widehat{\Omega}_{\ell}\right)\\
\pi_{2}\left(X_{\ell}\right)&=\pi_{2}\left(\widehat{\Omega}_{\ell}\right).
    \end{align*}
    Consequently,
    \begin{gather*}
        \pi_2(X_{\ell}^c)=\pi_{2}\left(\widehat{\Omega}_{\ell}^{c}\right).
    \end{gather*}
    Since $|p(\gamma, \xi)| > \ell$ on $X_{\ell}^c$ we have, with constant $K$ as in \eqref{ineq:lowerbound},
\begin{align*}
    \pi_{2}\left(X_{\ell}^{c}\right)\leq\frac{1}{\ell}\int_{X_{\ell}^{c}}|p(\gamma, \xi)|d\pi_{2}(\xi)\leq \frac{1}{\ell}\int_{\Omega_{2}}|p(\gamma, \xi)-K+K|\;d\pi_{2}(\xi)\leq\frac{1}{\ell}\int_{\Omega_{2}} p(\gamma, \xi) + 2K\;d\pi_{2}(\xi).
\end{align*}
Integrating both sides with respect to $\pi_1$ gives
\begin{gather*}
    \pi_2(X_\ell^c) \leq \frac{\widehat{\sigma}_0(\pi_2) + K}{\ell}.
\end{gather*}
    
Applying this to each $\tau_{k}$, we acquire
    \begin{align}\label{inq: sequence inq}
        \tau_{k}\left(\widehat{\Omega}_{\ell}^{c}\right)\leq \frac{\widehat{\sigma}_{0}(\tau_{k})+K}{\ell}\leq \frac{\widetilde{C}}{\ell},\;\forall \;k.
    \end{align}
    Using this inequality in (\ref{inq: sequence inq}), for each $\ell>0$, given $\varepsilon:=\frac{\widetilde{C}}{\ell}>0$, we have thus
    \[
\tau_{k}\left(\widehat{\Omega}_{\ell}^{c}\right)\leq\frac{\widetilde{C}}{\ell}=\varepsilon,\; \forall \;k.
    \]
    From Condition \ref{condition 1} we know each $\widehat{\Omega}_{\ell}$ is compact and hence the sequence $\{\tau_k\}_k$ is tight. By Prokhorov's Theorem, up to a subsequence, there exists a Borel probability measure $\tau$ such that $\tau_k \rightharpoonup \tau$ as $k\to \infty$. By Proposition \ref{prop:limit} it follows that $\tau \in \Pi_2(\mu_2,\nu_2)$. 
    
    That the cost function in Condition \ref{condition 2} is lower semi-continuous follows by \cite[Proposition 3.3]{cabrera2022optimaltransportationprincipleinteracting}. Therefore, $p(\gamma, \xi)+K$ \;is lower semi-continuous with respect to $\xi$. From which it follows that it can be written as the limit of non-decreasing bounded continuous functions. Moreover, since $p_{n}(\gamma, \xi) + K \geq 0$ then by the Monotone Convergence Theorem we have
    \begin{align*}
        \hat{\sigma}_0(\tau) &= \iint_{\Omega_1\times \Omega_2}(p(\gamma, \xi) + K)\ d\tau(\xi)d\pi_1(\gamma) = \lim_{n\to \infty}\iint_{\Omega_1\times\Omega_2} (p_n(\gamma,\xi) + K)\ d\tau(\xi)d\pi_1(\gamma) \\
        &= \lim_{n\to \infty}\lim_{k \to \infty} \iint_{\Omega_1\times\Omega_2} (p_n(\gamma, \xi) + K)\ d\tau_k(\xi)d\pi_{1}(\gamma) \leq \liminf_{k\to \infty} \iint_{\Omega_1\times\Omega_2} (p(\gamma,\xi) + K)\ d\tau_k(\xi)d\pi_{1}(\gamma).
    \end{align*}
    Since $\lim_{k\to \infty} \iint_{\Omega_{1}\times\Omega_2} (p(\gamma, \xi) + K)\ d\tau_k(\xi) d\pi_{1}(\gamma)= \inf_{\pi_2 \in \Pi_2(\mu_2,\nu_2)} \hat{\sigma}_0(\pi_2)$, then 
    \[\hat{\sigma}_0(\tau) \leq \inf_{\pi_2 \in \Pi_2(\mu_2,\nu_2)} \hat{\sigma}_0(\pi_2).\]
     Subtracting out the constant $K$ gives
    \begin{gather*}
        \iint_{\Omega_1\times \Omega_2} p(\gamma,\xi)\ d\tau(\xi)d\pi_1(\gamma) \leq \inf_{\pi_2 \in \Pi_2(\mu_2,\nu_2)} \iint_{\Omega_1\times \Omega_2} p(\gamma,\xi)\ d\pi_2(\xi)d\pi_1(\gamma).
    \end{gather*}
    This shows that for each $\pi_1 \in \mathcal{P}(\Omega_1)$ there exists $\tau^{\pi_1} := \tau(\pi_1)$ such that
    \[
     \iint_{\Omega_1\times \Omega_2} p(\gamma,\xi)\ d\tau^{\pi_{1}}(\xi)d\pi_1(\gamma) = \inf_{\pi_2 \in \Pi_2(\mu_2,\nu_2)} \iint_{\Omega_1\times \Omega_2} p(\gamma,\xi)\ d\pi_2(\xi)d\pi_1(\gamma).
    \]
    Since $p(\gamma, \xi)=-c(\gamma,\xi)$ we thus have,
    \begin{gather*}
        \iint_{\Omega_1\times \Omega_2} c(\gamma,\xi)\ d\tau^{\pi_1}(\xi)d\pi_1(\gamma)=\sup_{\pi_2\in \Pi_2(\mu_2,\nu_2)} \iint_{\Omega_1\times \Omega_2} c(\gamma,\xi)\ d\pi_1d\pi_2 . 
    \end{gather*}
    
    Similarly, let $\check{\sigma}_{0}:\mathcal{P}(\Omega_{1})\to \mathbb{R}$ be defined by
    \[
    \check{\sigma}_{0}(\pi_{1}):=\iint_{\Omega_{1}\times\Omega_{2}}c(\gamma, \xi)d\tau^{\pi_{1}}(\xi)d\pi_{1}(\gamma).
    \]
    Let $\{\eta_j\}_{j=1}^\infty$ be a minimizing sequence of $\check{\sigma}_0$ such that $\check{\sigma}_{0}(\eta_{j})\to \inf_{\pi_{1}}\check{\sigma}_{0}(\pi_{1})$. Since $c$ is (uniformly) bounded, then there exists a constant $C > 0$ such that
    \begin{gather*}
        \iint_{\Omega_1\times \Omega_2}c(\gamma,\xi)d\tau^{\eta_j}(\xi)d\eta_j(\gamma) > -C > -\infty.
    \end{gather*}
    Thus, the infimum is finite. We may repeat the same argument as above. That is, we define the set $\Omega_{\ell}$ as in Condition \ref{condition 1}. Consider the following set
    \begin{align}\label{Y ell}
        Y_\ell = \{(\gamma,\xi) \in \Omega_1\times \Omega_2: |c(\gamma,\xi)| \leq \ell\}
    \end{align}
   As before, choose $\ell>0$ large enough so that $\textit{spt}(\mu_{1})$ and $\textit{spt}(\mu_{2})$ are entirely contained in $B_{\ell}(0)$. Then 
    \[
    \textit{spt}(\mu_{1})\otimes \textit{spt}(\mu_{2})\subset B_{\ell}(0)\times B_{\ell}(0). 
    \]
    Therefore, since $(e_{0}, e_{0})_{\sharp}(\pi_{1}\otimes \tau^{\pi_{1}})=\mu_{1}\otimes \mu_{2}$ iff $(e_{0}, e_{0})_{\sharp}\pi_{1}\otimes \pi_{2}=\mu_{1}\otimes \mu_{2}$, then $\textit{spt}(\pi_{1}\otimes\tau^{\pi_{1}})\subset (e_{0}, e_{0})^{-1}(\textit{spt}(\mu_{1}\otimes 
     \mu_{2}))$. As $\pi_{1}\otimes\tau^{\pi_{1}}$ is supported on $\Omega_{\ell}$, we have
    \begin{gather*}
        \pi_1\otimes \tau^{\pi_1}(\Omega_\ell) =\pi_1\otimes \tau^{\pi_1}(Y_\ell)\Longrightarrow \pi_1\otimes \tau^{\pi_1}(\Omega_\ell^{c}) =\pi_1\otimes \tau^{\pi_1}(Y_\ell^{c}) 
    \end{gather*}
    Applying Markov's inequality yields
    \begin{gather*}
        \pi_{1} \otimes \tau^{\pi_{1}}(Y_\ell^c) \leq \frac{\iint_{\Omega_1\times \Omega_2} c(\gamma,\xi)d\pi_1(\gamma)d\tau^{\pi_1}(\xi)}{\ell}\leq \frac{\iint c(\gamma, \xi)+C d\pi_{\gamma}d\tau^{\pi_{1}}(\xi)+C}{\ell}=:\frac{\widetilde{C}}{\ell}.
    \end{gather*}
    Applying this to our minimizing sequence yields
    \begin{gather*}
        \eta_j \otimes \tau^{\eta_j}(\Omega_\ell^c) \leq \frac{\tilde{C}}{\ell} := \varepsilon
    \end{gather*}
    which may be made sufficiently small for $\ell$ large enough. This implies the sequence of measures is tight and hence by Prokhorov's there exists a subsequence and $\rho$ such that $\eta_j\otimes \tau^{\eta_j} \rightharpoonup\rho$. By Proposition \ref{prop:limit} we have $\rho \in \Pi_{\text{path}}(\mu_1,\nu_1)$ and according to Lemma \ref{lemma:limitproduct}, $\rho$ is contained in $\Pi_{\text{path}}(\mu_1,\nu_1)\otimes \Pi_{\text{path}}(\mu_2,\nu_2)$. Defining $\pi_1 = (\text{proj}_1)_\#\rho$ and $\pi_2 = (\text{proj}_2)_\#\rho$, then $\eta_j \otimes \tau^{\eta_j} \rightharpoonup \pi_1 \otimes \pi_2$. Using the same argument, we can show that $\rho$ realizes the infimum, thus completing the proof.
\end{proof}

\begin{corollary}
    Under the same assumptions as Theorem \ref{thm:minimaxExistence}, there then exists a maximin transport plan such that
    \begin{gather*}
        \sup_{\pi_2\in \Pi_2(\mu_2,\nu_2)}\inf_{\pi_1\in \Pi_1(\mu_1,\nu_1)} \iint_{\Omega_1\times \Omega_2} c(\gamma,\xi)d\pi_1(\gamma)d\pi_2(\xi).
    \end{gather*}
\end{corollary}
\begin{proof}
    Since
    \begin{align*}
        \sup_{\pi_2\in \Pi_2(\mu_2,\nu_2)}\inf_{\pi_1\in \Pi_1(\mu_1,\nu_1)} &\iint_{\Omega_1\times \Omega_2} c(\gamma,\xi)d\pi_1(\gamma)d\pi_2(\xi) \\
        &= -\inf_{\pi_1 \in \Pi_1(\mu_1,\nu_1)}\sup_{\pi_2\in \Pi_2(\mu_2,\nu_2)} \iint -c(\gamma,\xi)d\pi_1(\gamma)d\pi_2(\xi),
    \end{align*}
    then applying Theorem \ref{thm:minimaxExistence} to the cost function $\tilde{c}(\gamma,\xi) = -c(\gamma, \xi)$ yields the desired result.
\end{proof}

\subsection{Minimax equals Maximin}\label{minimax section} 
In this section we begin to set the stage. In fact, we will show weak duality (Proposition \ref{prop: weak duality}) always holds. Moreover, we establish that equality actually holds (strong duality) under certain conditions (Theorem \ref{thm:minimax=maximin}). It is an interesting question to ask when this problem is exactly equal to the maximin. This implies the existence of a saddle-point equilibrium. This is an extremely subtle question. Classic theorems such as the Kakutani Fixed Point Theorem \cite{kakutani1941generalization} or Glicksberg's Existence Theorem \cite{glicksberg1952further} which deduce the existence of saddle points require the ambient spaces to be compact, closed and/or convex subsets of Euclidean spaces or Banach spaces. Therefore, the minimax theorems cannot be applied directly to this setting. We will now turn to a classical result from minimax theory that will be essential to our setting; namely to show $\inf_{\pi_{1}}\sup_{\pi_{2}}\iint cd\pi_{2}d\pi_{2}=\sup_{\pi_{2}}\inf_{\pi_{1}}\iint cd\pi_{1}d\pi_{2}$.

The following construction is essential in the sequel. For each $\pi_{1}\in \mathcal{P}(\Omega_{1})$, let 
\begin{align}\label{big phi}
    \Phi(\pi_{1})=\argmax\left(\pi_{2}\mapsto \iint cd\pi_{1}d\pi_{2}\right)
\end{align}
be the set of $\pi_{2}\in \Pi_{\texttt{path}}(\mu_{2}, \nu_{2})$ which maximize $\pi_{2}\mapsto \iint cd\pi_{1}d\pi_{2}$. Similarly, for each fixed $\pi_{2}\in\mathcal{P}(\Omega_{2})$, let 
\begin{align}\label{big psi}
    \Psi(\pi_{2})=\argmin\left(\pi_{1}\mapsto \iint cd\pi_{1}d\pi_{2}\right)
\end{align}
be the set of $\pi_{1}\in \Pi_{\text{path}}(\mu_{1}, \nu_{1})$ which minimize $\pi_{1}\mapsto \iint cd\pi_{1}d\pi_{2}$, and where $c$ is as in Condition \ref{condition 2}.

Recall that for $Y\subset \Omega_{i}$, for $i=1,2$, and the evaluation map, $e_{t}:\Omega_{i}\to X$, then we have the set of transport path plans, defined equivalently as in (\ref{path plans}). We first show that the sets $\Omega_1,\Omega_2$ are convex and closed in the path space topology specified by the metric \eqref{eq:path-metric}.

\begin{lemma}\label{lemma: compact/convex sets}
    The sets $\Omega_1$ and $\Omega_2$ are compact and convex in the path space topology.
\end{lemma}

\begin{proof}
\noindent That $\Omega_{1}$ is convex follows from the following argument. Suppose $\gamma_{1},\gamma_{2}\in \Omega_{1}$ and let $s\in [0, 1]$. Consider 
\[
\gamma_{s}(t)=s\gamma_{1}(t)+(1-s)\gamma_{2}(t).
\]
Since $\gamma_{1}$ and $\gamma_{2}$ are Lipschitz, then so is $\gamma_{s}$ for all $s$. Next, we have $\gamma_{s}(0)=x$ while $\gamma_{s}(1)=y$. From the derivative bound and the triangle inequality, we have
\[
\|\dot{\gamma}_{s}\|\leq s\|\dot{\gamma}_{1}\|+(1-s)\|\dot{\gamma}_{2}\|\leq K^{1/2}.
\]
Therefore, $\Omega_{1}$ is convex. Arguments for $\Omega_{2}$ are identical.

We will show compactness of $\Omega_{1}$ through the \textit{Arzelà--Ascoli Theorem}. 
Then the claim is that $\Omega_{1}$ is compact in the path space topology. In fact, by the same calculation in Proposition \ref{prop:cond2impliescond1} we have that $\Omega_1$ is sequentially compact in the path space topology, \eqref{eq:path-metric}.
\end{proof}

Now, let us look at probability measures on $\Omega_{1}$ and $\Omega_{2}$.
The functional $\sigma_{0}:\mathcal{P}(\Omega_{1})\times \mathcal{P}(\Omega_{2})\to \mathbb{R}$ defined by (\ref{sigma functional}), namely, 
\begin{align*}
    \sigma_{0}(\pi_{1}, \pi_{2}):=\iint_{\Omega_{1}\times \Omega_{2}} c(\gamma, \xi)d\pi_{1}(\gamma)d\pi_{2}(\xi),
\end{align*}
is \textit{affine} in $\pi_{1}$ and $\pi_{2}$, since the integral is clearly linear in $\pi_{1}, \pi_{2}$. That is, each of $\pi_{1}\mapsto \sigma_{0}(\pi_{1},\pi_{2})$  and $\pi_{2}\mapsto \sigma_{0}(\pi_{1}\pi_{2})$ is affine. 

\begin{lemma}\label{lemma: uplo continuity}
    Suppose $c$ satisfies Condition \ref{condition 2} and let $(\pi_{1, k})_{k\in \mathbb{N}}$ be a sequence of probability measures on $\Omega_{1}$ converging weakly to some $\pi_{1}\in\mathcal{P}(\Omega_{1})$, in such a way that $c\in L^{1}(\pi_{1, k}), c\in L^{1}(\pi_{1})$. Similarly let $(\pi_{2, k})_{k\in \mathbb{N}}$ be a sequence of probability measures on $\Omega_{2}$ converging weakly to some $\pi_{2}\in\mathcal{P}(\Omega_{2})$, in such a way that $c\in L^{1}(\pi_{2, k}), c\in L^{1}(\pi_{2})$. Then the maps, for fixed $\pi_1$, $\pi_2$ respectively,
    \begin{gather*}
        \pi_1 \mapsto \iint c(\gamma,\xi) d\pi_1(\gamma) d\pi_{2}(\xi),\quad \pi_{2}\mapsto \iint c(\gamma, \xi)d\pi_{1}(\gamma)d\pi_{2}(\xi)
    \end{gather*}
    are weakly lower and weakly upper semicontinuous, respectively. In particular,
    \[
    (\pi_{1}, \pi_{2})\mapsto\iint c(\gamma,\xi) d\pi_1 (\gamma)d\pi_{2}(\xi)
    \]
    is jointly continuous.
\end{lemma}
\begin{proof}
    We may assume without loss of generality that $c \geq 0$. Condition \ref{condition 1} says that $c(\gamma, \cdot)$ is upper semicontinuous and $c(\cdot, \xi)$ is lower semicontinuous, and bounded. Therefore, there exists a nonincreasing sequence of functions $\{c_{n}\}$ which are continuous and bounded such that $c$ can be written as $c(\gamma,\xi)=\sup_{n}c_{n}(\gamma,\xi)$. Fix $\pi_{2}$. According to the Monotone Convergence Theorem and \cite[Lemma 4.2]{cabrera2022optimaltransportationprincipleinteracting}, we have
\[
\iint cd\pi_{1}d\pi_{2}=\lim_{n\to\infty}\iint c_{n}d\pi_{1}d\pi_{2}=\lim_{n\to \infty}\lim_{k\to\infty}\iint c_{n}d\pi_{1,k}d\pi_{2}\leq \liminf_{k\to\infty}\iint cd\pi_{1,k}d\pi_{2}.
\]
The claim now follows.

For the weakly upper semicontinuous case, fix $\pi_{1}\in \mathcal{P}(\Omega_{1})$, then identical arguments lead to the desired conclusion.
Now for the latter statement, we have
\[
\liminf_{k\to\infty}\sigma_{0}(\pi_{1,k}, \pi_{2})\geq \sigma_{0}(\pi_{1}, \pi_{2})\geq \limsup_{k\to\infty}\sigma_{0}(\pi_{1}, \pi_{2,k}).
\]

The goal is to find a version of these inequalities to arrive at the joint continuity of $\sigma_{0}$,
\begin{align}\label{joint cont.}
\lim_{k\to\infty}\sigma_{0}(\pi_{1, k}, \pi_{2,k})=\sigma_{0}(\pi_{1}, \pi_{2}).
\end{align}

So since the preceding semicontinuity statements are only separate in the two
variables, they imply that, for fixed $\pi_1\in\mathcal P(\Omega_1)$,
the map $\pi_2\mapsto \sigma_0(\pi_1,\pi_2)$ is upper semicontinuous,
and for fixed $\pi_2\in\mathcal P(\Omega_2)$, the map
$\pi_1\mapsto \sigma_0(\pi_1,\pi_2)$ is lower semicontinuous. However,
these separate semicontinuity properties do not by themselves imply
joint continuity of $\sigma_0$. To prove
\eqref{joint cont.},
we will prove the required estimates for the diagonal sequence
$(\pi_{1,k},\pi_{2,k})$, namely
\begin{align}\label{joint estimates}
\begin{split}
\limsup_{k\to\infty}\sigma_0(\pi_{1,k},\pi_{2,k})
\leq \sigma_0(\pi_1,\pi_2)\\
\liminf_{k\to\infty}\sigma_0(\pi_{1,k},\pi_{2,k})
\geq \sigma_0(\pi_1,\pi_2).
\end{split}
\end{align}
But by Lemma \ref{lemma:limitproduct} we thus have $\pi_{1,k}\otimes\pi_{2,k}\rightharpoonup\pi_{1}\otimes\pi_{2}$. As $c\in C_{b}(\Omega_{1}\times\Omega_{2})$, then 
\[
\iint_{\Omega_{1}\times\Omega_{2}} c(\gamma, \xi)d(\pi_{1,k}\otimes\pi_{2,k})\to\iint_{\Omega_{1}\times\Omega_{2}}c(\gamma, \xi)d(\pi_{1}\otimes\pi_{2}),
\]
but this is equivalent to $\sigma_{0}(\pi_{1,k},\pi_{2,k})\to\sigma_{0}(\pi_{1}, \pi_{2})$. Therefore, we get the joint estimates \eqref{joint estimates}. Now, we have
\[
\sigma_{0}(\pi_{1},\pi_{2})\leq \liminf_{k\to \infty}\sigma_{0}(\pi_{1,k},\pi_{2,k})\leq \limsup_{k\to \infty}\sigma_{0}(\pi_{1,k},\pi_{2,k})\leq \sigma_{0}(\pi_{1},\pi_{2})
\]
and the result follows at once.
\end{proof}

Now to establish the equality between the  minimaximum and maximinimum. We already have the following.
\begin{proposition} \label{prop: weak duality}
    \begin{align}\label{geq: minimax}
     \inf_{\pi_{1}\in \Pi_{\text{path}}(\mu_{1},\nu_{1})}\sup_{\pi_{2}\in \Pi_{\text{path}}(\mu_{2}, \nu_{2})}\sigma_{0}(\pi_{1}, \pi_{2})\geq \sup_{\pi_{2}\in \Pi_{\text{path}}(\mu_{2}, \nu_{2})}\inf_{\pi_{1}\in \Pi_{\text{path}}(\mu_{1}, \nu_{1})}\sigma_{0}(\pi_{1}, \pi_{2}).
\end{align}
\end{proposition}

\begin{proof}
For any fixed $\pi_{1}\in \Pi_{\text{path}}(\mu_{1}, \nu_{1})$ and $\overline{\pi}_{2}\in \Pi_{\text{path}}(\mu_{2}, \nu_{2})$, we have 
\[
\sup_{\pi_{2}\in\Pi_{\text{path}}(\mu_{2}, \nu_{2})}\sigma_{0}(\pi_{1},\pi_{2})\geq\sigma_{0}(\pi_{1},\overline{\pi}_{2}).
\]
Thus, minimizing both sides of the above inequality with respect to $\pi_{1}\in \Pi_{\text{path}}(\mu_{1}, \nu_{1})$ yields,
\[
\inf_{\pi_{1}\in\Pi_{\text{path}}(\mu_{1}, \nu_{1})}\sup_{\pi_{2}\in\Pi_{\text{path}}(\mu_{2}, \nu_{2})}\sigma_{0}(\pi_{1},\pi_{2})\geq\inf_{\pi_{1}\in\Pi_{\text{path}}(\mu_{1}, \nu_{1})}\sigma_{0}(\pi_{1},\overline{\pi}_{2}).
\]
Now since this holds for any fixed $\overline{\pi}_{2}\in \Pi_{\text{path}}(\Omega_2)$, we can maximize the right hand side of the latter inequality over all $\overline{\pi}_{2}\in \Pi_{\text{path}}(\mu_{2}, \nu_{2})$ to get
\begin{align*}
\inf_{\pi_{1}\in\Pi_{\text{path}}(\mu_{1}, \nu_{1})}\sup_{\pi_{2}\in\Pi_{\text{path}}(\mu_{2}, \nu_{2})}\sigma_{0}(\pi_{1},\pi_{2})&\geq\sup_{\overline{\pi}_{2}\in\Pi_{\text{path}}(\mu_{2}, \nu_{2})}\inf_{\pi_{1}\in\Pi_{\text{path}}(\mu_{1}, \nu_{1})}\sigma_{0}(\pi_{1},\overline{\pi}_{2})\\
&=\sup_{\pi_{2}\in\Pi_{\text{path}}(\mu_2,\nu_2)}\inf_{\pi_{1}\in\Pi_{\text{path}}(\mu_1,\nu_1)}\sigma_{0}(\pi_{1},\pi_{2}).
\end{align*} 
\end{proof}

We will show the equality between minimax and maximin of the bilinear functional (\ref{sigma functional}). Namely, Theorem \ref{thm:minimax=maximin}.
\begin{proof}[Proof of Theorem \ref{thm:minimax=maximin}]
\noindent 
In order to prove the reverse inequality of (\ref{geq: minimax}) we will need the following.
The non-emptiness of $\Phi(\pi_{1})$ defined in \eqref{big phi} follows from upper semicontinuity of the affine mapping $\pi_{2}\mapsto \sigma_{0}(\pi_{1},\pi_{2})$ (Lemma \ref{lemma: uplo continuity}) and relative compactness of $\Pi_{\text{path}}(\mu_{2}, \nu_{2})$ in $\Omega_{2}$.  That $\Pi_{\text{path}}(\mu_{2}, \nu_{2})$ is relatively compact follows from identical arguments from Theorem \ref{thm:minimaxExistence}.

The closedness (and hence compactness) of $\Phi(\pi_{1})$ follows from continuity as well. Convexity follows from the affine property. Indeed, if $\pi_{2}^{0}$ and $\pi_{2}^{1}$ belong to $\Phi(\pi_{1})$, then $\sigma_{0}(\pi_{1},\pi_{2}^{0})=\sigma_{0}(\pi_{1},\pi_{2}^{1})=\mathcal{M}:=\sup_{\pi_{2}\in\Pi_{\text{path}}(\mu_{2}, \nu_{2})}\sigma_{0}(\pi_{1},\pi_{2}).$ Then for $t\in [0, 1]$, we have
\begin{align*}
    \sigma_{0}(\pi_{1}, t\pi_{2}^{1}+(1-t)\pi_{2}^{0})&=t\sigma_{0}(\pi_{1},\pi_{2}^{1})+(1-t)\sigma_{0}(\pi_{1},\pi_{2}^{0})=t\mathcal{M}+(1-t)\mathcal{M}=\mathcal{M},
\end{align*}
implying $t\pi_{2}^{1}+(1-t)\pi_{2}^{0}\in \Phi(\pi_{1})$. 

Lastly, we show that the graph of $\Phi(\pi_{1})$ is closed. Suppose that $\pi_{1,k}\rightharpoonup \pi_{1}$ under the weak topology, and $\pi_{2,k}\in \Phi(\pi_{1,k})$ such that $ \pi_{2,k}\rightharpoonup \pi_{2}$. We have for any $\overline{\pi}_{2}$,
\[
\sigma_{0}(\pi_{1,k},\overline{\pi}_{2})\leq \sigma_{0}(\pi_{1,k},\pi_{2,k}).
\]
Taking limits and using upper/lower semicontinuity; thus continuity of $\sigma_{0}$, by Lemma \ref{lemma: uplo continuity}, yields
\[
\sigma_{0}(\pi_{1},\overline{\pi}_{2})\leq \sigma_{0}(\pi_{1}, \pi_{2}),
\]
which, since $\overline{\pi}_{2}$ was arbitrary, means that $\pi_{2}\in \Phi(\pi_{1})$. Identical arguments applied to $\Psi(\pi_{2})$ show the graph of $\Psi(\pi_{2})$ is closed.

Armed with this knowledge and the \textit{Fan-Glicksberg Fixed Point Theorem} \cite{fan1952fixed}, \cite{glicksberg1952further}, we prove the reverse inequality of (\ref{geq: minimax}). Since our path spaces, $\Pi_{\text{path}}(\mu_{1}, \nu_{1})$ and $\Pi_{\text{path}}(\mu_{2}, \nu_{2})$ are  infinite dimensional, by an application of \textit{Fan-Glicksberg Fixed Point Theorem}, there exists $\overline{\pi}_{1}\in \Pi_{\text{path}}(\mu_{1}, \nu_{1}), \overline{\pi}_{2}\in \Pi_{\text{path}}(\mu_{2}, \nu_{2})$ such that $\overline{\pi}_{2}\in \Phi(\overline{\pi}_{1})$ and $\overline{\pi}_{1}\in \Psi(\overline{\pi}_{2})$. Then
\[
\sigma_{0}(\overline{\pi}_{1},\overline{\pi}_{2})=\sup_{\pi_{2}\in \Pi_{\text{path}}(\mu_{2},\nu_{2})}\sigma_{0}(\overline{\pi}_{1},\pi_{2})\geq \inf_{\pi_{1}\in \Pi_{\text{path}}(\mu_{1},\nu_{1})}\sup_{\pi_{2}\in \Pi_{\text{path}}(\mu_{2},\nu_{2})}\sigma_{0}(\pi_{1},\pi_{2})
\]
and
\[
\sigma_{0}(\overline{\pi}_{1},\overline{\pi}_{2})=\inf_{\pi_{1}\in \Pi_{\text{path}}(\mu_{1},\nu_{1})}\sigma_{0}(\pi_{1},\overline{\pi}_{2})\leq \sup_{\pi_{2}\in\Pi_{\text{path}}(\mu_{2}, \nu_{2})}\inf_{\pi_{1}\in \Pi_{\text{path}}(\mu_{1}, \nu_{1})}\sigma_{0}(\pi_{1},\pi_{2}).
\]
Putting together both of these inequalities gives
\[\inf_{\pi_{1}\in \Pi_{\text{path}}(\mu_{1}, \nu_{1})}\sup_{\pi_{2}\in \Pi_{\text{path}}(\mu_{2}, \nu_{2})}\sigma_{0}(\pi_{1},\pi_{2})\leq \sigma_{0}(\overline{\pi}_{1},\overline{\pi}_{2})\leq
\sup_{\pi_{2}\in\Pi_{\text{path}}(\mu_{2},\nu_{2})}\inf_{\pi_{1}\in \Pi_{\text{path}}(\mu_{1},\nu_{1})}\sigma_{0}(\pi_{1},\pi_{2}),
\]
and thus
\[
\inf_{\pi_{1}\in \Pi_{\text{path}}(\mu_{1}, \nu_{1})}\sup_{\pi_{2}\in \Pi_{\text{path}}(\mu_{2}, \nu_{2})}\sigma_{0}(\pi_{1},\pi_{2})\leq 
\sup_{\pi_{2}\in\Pi_{\text{path}}(\mu_{2}, \nu_{2})}\inf_{\pi_{1}\in \Pi_{\text{path}}(\mu_{1}, \nu_{1})}\sigma_{0}(\pi_{1},\pi_{2}).
\]
This inequality together with Proposition \ref{prop: weak duality} yields the desired equality.
\end{proof}

\section{Characterization of Nash-Monge-Kantorovich Transport Maps}\label{NMK maps}
In this section we study the following problem 
\begin{gather*}
    \inf_{\pi_1}\sup_{\pi_2} \sigma_{0}(\pi_{1}, \pi_{2});
\end{gather*}
where $\sigma_{0}$ is defined in \eqref{sigma functional}. Here we specialize the cost function $c$ to be the one in Condition \ref{condition 2}, that is
\begin{gather*}
    c_e(x, y,x',y') = \inf_{\substack{\gamma: \gamma(0) = x\\\gamma(1) = y}}\sup_{\substack{\xi: \xi(0) = x'\\\xi(1) = y'}} c(\gamma, \xi)
\end{gather*}
We will see that the value of $\alpha$ determines the existence of NMK transport maps in subtle ways. This is reflected from the cost function analysis in Section \ref{section: analysis of cost} where the choice of $\alpha$ yields existence of Lipschitz paths and/or the twist condition. However, this particular cost also yields several consequences that follow the progression in proving Brenier's theorem, i.e. that the transport map is given by the gradient of a convex potential.

\subsection{Quadratic Interaction and when Minimax = Maximin}
Recall that we are interested in defining the endpoint functional:
\begin{gather*}
    \inf_{\substack{\gamma(0) = x\\\gamma(1) = y}}\sup_{\substack{\xi(0) = x'\\\xi(1) = y'}}\int_0^1 c_{\alpha}(\gamma,\xi)\ dt
\end{gather*}
where $c_{\alpha}(\gamma,\xi)$ is the Lagrangian defined in \eqref{kinetic potential energy}. Here, the dependence on $\alpha$ is crucial as this determines whether the minimax value is finite, whether it equals the maximin value, and whether it can be obtained using variational calculus. 

In order to use variational methods (e.g. Euler-Lagrange equations) to characterize this value, we need the \emph{strong duality equality}
\begin{gather}\label{eq:minimax}
    \inf\sup c(\gamma,\xi) = \sup\inf c(\gamma, \xi),
\end{gather}
which implies that there exists a \textit{saddle-point} (or \textit{Nash}) equilibrium on the level of curves $\gamma, \xi$.

Since we cannot apply the minimax theorem directly, let us determine whether \eqref{eq: minimax} holds for this particular problem.

Fix an admissible path $\gamma$ such that $\gamma(0) = x$ and $\gamma(1) = y$. Then consider
\begin{gather*}
    \sup_{\substack{ \xi(0) = x'\\\xi(1) = y'}}\int_0^1 \frac12|\dot{\gamma}|^2 - \frac12|\dot{\xi}|^2 + \alpha|\gamma - \xi|^2\ dt.
\end{gather*}
For any curve $\xi$ we may write $\xi(t) = \ell(t) + u(t)$ where $\ell(t) = x' + (y'-x')t$, $u(t)$ is a function which satisfies Dirichlet conditions $u(0) = u(1) = 0$, and $z \in \mathbb{R}^{d}$. Since $\gamma$ is fixed, then we may consider the quadratic approximation on $\xi$:
\begin{gather*}
    \int_0^1 \frac12|\dot{\gamma}|^2 - \frac12|\dot{\xi}|^2 + \alpha|\gamma - \xi|^2\ dt \approx \int_0^1-\frac12|\dot{\xi}|^2 + \alpha |\xi|^2\ dt
\end{gather*}
where we neglect the lower order terms. Now, consider the boundary-value problem
\begin{gather*}
\begin{cases}
    -u'' = \lambda u\\
    u(0) = u(1) = 0
\end{cases}.
\end{gather*}
The eigenfunctions can be expressed as $u_n(t) = \sin(n\pi t)$ with corresponding eigenvalues $\lambda_n = n^2\pi^{2}$, $n=1,2,\ldots$. Then for any $u \in L^2(0,1)$, we write using superposition:
\begin{gather*}
    u(t) = \sum_{i=1}^\infty u_n\sin(n\pi t).
\end{gather*}
Now dropping the lower order terms from the following integral, we obtain
\begin{align*}
    \int_0^1|\xi|^2\ dt &= \int_0^1|\ell(t)|^2 + 2\ell(t)u(t) + |u(t)|^2\ dt\\
    &\approx \sum_{n=1}^\infty \int_0^1 u_n^2 \sin^2(n\pi t)\ dt \\
    &= \sum_{i=1}^\infty\frac{u_n^2}{2}. 
\end{align*}
Moreover, looking at the kinetic term and incorporating the second derivative of the Fourier series above, we acquire the following integral computation,
\begin{align*}
    \int_0^1|\dot{\xi}|^2dt  &= \xi\dot{\xi}\Big|_0^1 - \int_0^1 \xi\ddot{\xi}\ dt= - \int_0^1 [\ell(t) + u(t)]\ddot{u}(t)\ dt\\
   & = -\sum_{n=1}^\infty n^2\pi^{2}u_n\int_0^1\ell(t)\sin(n\pi t)\ dt + \sum_{n=1}^\infty n^2\pi^{2}u_n^2\int_0^1 \sin^2(n\pi t)\ dt\\
    &\approx \sum_{n=1}^\infty \frac{n^2\pi^{2}u_n^2}{2}.
\end{align*}
Thus, the quadratic approximation on $\xi$ is given in the following expression,
\begin{gather*}
    \int_0^1-\frac12|\dot{\xi}|^2 + \alpha |\xi|^2\ dt \approx -\frac12\sum_{n=1}^\infty \frac{n^2\pi^{2}u_n^2}{2} + \alpha \sum_{n=1}^\infty \frac{u_n^2}{2} = \frac12\sum_{i=1}^\infty \left(\alpha - \frac{n^2\pi^{2}}{2} \right)u_n^2.
\end{gather*}
Then each eigenfunction contributes a factor of $\alpha - \frac{n^2\pi^{2}}{2}$. This gives two cases:
\begin{itemize}
    \item[1.] If $0 \leq \alpha < \frac{\pi^{2}}{2}$, then $\alpha - \frac{n^2\pi^{2}}{2} < 0$ for all $n$ and hence the infsup value is finite.

    \item[2.] If $\alpha \geq \frac{\pi^{2}}{2}$, then there exists $k$ such that $\alpha - \frac{k^2\pi^{2}}{2} > 0$. Consequently, scaling the $u_k(t)$ by a factor $c > 0$ gives that $\xi(t) = \ell(t) + cu_k(t)$ is admissible and yields
    \begin{gather*}
        \sup_{\substack{\xi(0) = x'\\
        \xi(1) = y'}}\int_0^1 \frac12|\dot{\gamma}|^2 - \frac12|\dot{\xi}|^2 + \alpha|\gamma - \xi|^2\ dt = +\infty.
    \end{gather*}
\end{itemize}

For the interested reader we include an alternative way of seeing this, which can be found by using Poincar\'e's inequality \cite{brezis2011functional}. That is, fixing $\gamma$, we expand the higher order term and consider when the quadratic in $\xi$ is bounded above. That is,
\begin{align*}
    \int_0^1 \frac12|\dot{\gamma}|^2 - \frac12|\dot{\xi}|^2 + \alpha|\gamma - \xi|^2\ dt &= \int_0^1 -\frac12|\dot{\xi}|^2 + \alpha|\xi|^2 + \frac12|\dot{\gamma}|^2 - 2\alpha \gamma \cdot \xi + \alpha |\gamma|^2\ dt.
\end{align*}
Since the linear function in $\xi$ is convex, the terms in $\gamma$ are constant with respect to $\xi$, and any admissible path $\xi$ can be decomposed as $\xi = \ell + u$ for $\ell$ a linear path from $x'$ to $y'$ and $u \in H^1_0([0,1])$, then it is sufficient to consider the functional
\begin{gather*}
    J[\xi] = \int_0^1 \alpha |\xi|^2 - \frac12|\dot{\xi}|^2\ dt,\quad \xi \in H_0^1([0,1]).
\end{gather*}
The second variation is given by
\begin{gather*}
    \delta^2J[h] = \int_0^1 2\alpha |h|^2 - |\dot{h}|^2\ dt
\end{gather*}
and a sufficient condition for a maximizer to exist is
\begin{gather*}
    \delta^2J[h] < 0\quad \forall h \in H_0^1([0,1]).
\end{gather*}
This yields the condition
\begin{gather*}
    \alpha\int_0^1 |h|^2\ dt < \frac{1}{2}\int_0^1 |\dot{h}|^2\ dt.
\end{gather*}
From Poincar\'e's inequality we have
\begin{gather*}
    \alpha \int_0^1 |h|^2\ dt \leq \frac{\alpha}{\pi^{2}}\int_0^1 |\dot{h}|^2\ dt. 
\end{gather*}
Then comparing the two inequalities, we require $\frac{\alpha}{\pi^{2}} < \frac{1}{2}$ which gives $\alpha < \frac{\pi^{2}}{2}$.

Conversely, we show that the ``supinf" is always finite. We first exhibit a bound from below. Let $\xi_0(t) = (1-t)x' + ty'$ and consider the functional
\begin{gather*}
    \int_0^1 \frac12|\dot{\gamma}|^2 - \frac12|\dot{\xi}_0|^2 + \alpha |\gamma - \xi_0|^2\ dt.
\end{gather*}
As this functional is convex in $\gamma$ and lower-semicontinuous on an affine translate of $H_0^1([0,1])$, then there exists a minimizer $\gamma^*$ and the value of this functional is bounded. In particular
\begin{gather*}
    \sup_{\xi}\inf_\gamma \int_0^1 c_{\alpha}(\gamma, \xi)\ dt \geq \int_0^1 c_{\alpha}(\gamma^*,\xi_0 )\ dt > -\infty.
\end{gather*}

Next, to obtain an upper bound let us define for any given admissible $\xi$ the path $\gamma(t) = \xi(t) + (1-t)a + tb$ where $a = x'-x$ and $b = y' - y$. Then $\gamma$ is also an admissible path whenever $\xi$ is one. Plugging these into the integral give
\begin{align*}
    \int_0^1 \frac12|\dot{\gamma}|^2 - \frac12|\dot{\xi}|^2 + \alpha |\gamma - \xi|^2\ dt &= \int_0^1 \frac12|\dot{\xi} + b-a|^2 - \frac12|\dot{\xi}|^2 + \alpha |(1-t)a+tb|^2\ dt\\
    &= \int_0^1 \dot{\xi}\cdot (b-a) + \frac12|b-a|^2 + \alpha(1-t)|a|^2 + 2\alpha (1-t)ta\cdot b + \alpha t|b|^2\ dt\\
    &= \xi\cdot (b-a)\Big|_0^1 + \frac12|b-a|^2 + \frac12\alpha |a|^2 + \frac13\alpha a\cdot b + \frac12\alpha |b|^2 \\
    &= (y'-x')\cdot (b-a)+ \frac12|b-a|^2 + \frac12\alpha |a|^2 + \frac13\alpha a\cdot b + \frac12\alpha |b|^2\\
    &= C(x,y,x',y')
\end{align*}
where $C:=C(x,y,x',y')$ is some constant that only depends on the points $x,y,x'$ and $y'$. Thus,
\begin{gather*}
    \inf_{\gamma}\int_0^1 c_{\alpha}(\gamma,\xi)\ dt \leq C < \infty;
\end{gather*}
which implies after taking a supremum that
\begin{gather*}
    \sup_{\xi}\inf_{\gamma}\int_0^1 c_{\alpha}(\gamma, \xi)\ dt \leq C < \infty.
\end{gather*}
Combining the previous inequality shows that the supinf value is always finite.

Finally, letting $\ell$ and $\ell'$ denote the line segments connecting $x$ to $y$ and $x'$ to $y'$ respectively, the functional $\mathcal{F}: H_0^1([0,1])\times H_0^1([0,1]) \to \mathbb{R}$ defined by
\begin{gather*}
    \mathcal{F}[u,v] = \int_0^1 \frac12|\ell + u|^2 - \frac12|\ell' + v|^2 + \alpha |\ell + u - v - \ell'|^2\ dt
\end{gather*}
is continuous in each variable. Moreover, it is convex in $u$ and concave in $v$ provided $\alpha < \frac{\pi^{2}}{2}$. By the Sion minimax theorem \cite{sion1958general}-\cite{komiya1988elementary}, we conclude that, on the level of curves, we get strong duality.
\begin{lemma}
    If $\alpha < \frac{\pi^{2}}{2}$, then
    \begin{gather*}
        \inf_{\gamma}\sup_{\xi} \int_0^1 c(\gamma, \xi,\alpha)\ dt = \sup_{\xi}\inf_{\gamma}\int_0^1 c(\gamma,\xi,\alpha)\ dt.
    \end{gather*}
\end{lemma}
This result suggests that upon defining an ``end-point" cost function, the order of infsup or supinf that we may consider is immaterial. 

\subsection{Stationary Problem Associated to Quadratic Interaction}\label{sec: stationary}
This section is devoted to explicitly calculating the end-point cost function. Consider the end-point cost function defined by
\begin{align}\label{eqn: end-point cost function}
    c_e(x,y,x',y') = \inf_{\substack{\gamma(0) = x\\
    \gamma(1) = y}}\sup_{\substack{\xi(0) = x'\\
    \xi(1) = y'}}\int_0^1 \frac12 |\dot{\gamma}(t)|^2 - \frac12|\dot{\xi}(t)|^2 + \alpha |\gamma(t)-\xi(t)|^2\ dt
\end{align}
where $\alpha < \frac{\pi^{2}}{2}$; which is another way of explicitly writing \eqref{qot end-point cost} incorporating Condition \eqref{condition 2}. The end-point cost function defines a cost function in $\mathbb{R}^d\times \mathbb{R}^d$ from which we proceed to minimaximize over couplings in $\mathbb{R}^d\times \mathbb{R}^d$.

\begin{lemma}\label{thm:explicit cost}
    Let $c(\gamma,\xi) = \int_0^1 \frac12|\dot{\gamma}|^2 - \frac12|\dot{\xi}|^2 + \alpha |\gamma - \xi|^2\ dt$ and define the infsup (and respectively the supinf) cost functions. If $\alpha < \frac{\pi^{2}}{2}$, then
    \begin{gather*}
        \inf_{\substack{\gamma(0) = x\\\gamma(1) = y}}\sup_{\substack{\xi(0) = x'\\\xi(1) = y'}} c(\gamma,\xi) = \sup_{\substack{\xi(0) = x'\\\xi(1) = y'}}\inf_{\substack{\gamma(0) = x\\\gamma(1) = y}} c(\gamma,\xi) \\
        =  \frac12|x-y|^2 - \frac12|x'-y'|^2 + \frac\alpha3\left[(x-x')^2 + (x-x')(y-y') + (y-y')^2 \right].
    \end{gather*}
\end{lemma}

\begin{proof} 
Note that at the saddle point, the Euler-Lagrange equations must be satisfied, that is
\begin{gather*}
    \begin{cases}
        \frac{\partial}{\partial t}\frac{\partial L}{\partial \dot\gamma} - \frac{\partial L}{\partial \gamma} = 0\\
        \frac{\partial}{\partial t}\frac{\partial L}{\partial \dot\xi} - \frac{\partial L}{\partial \xi} = 0
    \end{cases}.
\end{gather*}
Here $L$ is the Lagrangian given by \eqref{kinetic potential energy}.
Therefore, upon substitution of $L$ the above coupled equation is given as follows
\begin{gather*}
\begin{cases}
    \ddot{\gamma} - 2\alpha(\gamma - \xi) = 0\\
    -\ddot{\xi} - 2\alpha(\xi - \gamma) = 0 
\end{cases} \Longrightarrow\quad  \begin{cases}
    \ddot{\gamma} - 2\alpha(\gamma - \xi) = 0\\
    \ddot{\xi} - 2\alpha(\gamma - \xi) = 0 
\end{cases}.
\end{gather*}
Subtracting both equations and setting $d = \gamma - \xi$ we have
\begin{gather*}
    \begin{cases}
        \ddot{d} = 0\\
        d(0) = x - x'\\
        d(1) = y - y'
    \end{cases}.
\end{gather*}
The solution is therefore given by
\begin{gather*}
    d(t) = (x-x')(1-t) + (y-y')t.
\end{gather*}
In particular, substituting this into the coupled equation above and applying the boundary conditions, we get
\begin{gather*}
    \ddot{\gamma} = 2\alpha d = 2\alpha(x-x')(1-t) + 2\alpha(y-y')t.
\end{gather*}
 Integrating twice we can solve for $\gamma$,
\begin{gather*}
    \gamma(t) = C_1 + C_2t + \alpha(x - x')t^2 +  \frac{\alpha}{3}(y - y' + x' - x)t^3.
\end{gather*}
Applying the boundary condition $\gamma(0) = x$ yields
\begin{gather*}
    C_1 = x.
\end{gather*}
Next, applying the other boundary condition $\gamma(1) = y$ yields
\begin{gather*}
    y = x + C_2 + \alpha(x-x') + \frac{\alpha}{3}(y - y'+x' - x).
\end{gather*}
Thus, we can solve for $C_{2}$,
\begin{gather*}
    C_2 = y - x - \alpha\left[\frac13(y - y') + \frac23(x-x') \right].
\end{gather*}
It follows that $\gamma$ is given by
\begin{gather*}
    \gamma(t) = x + \left[y - x - \alpha \left(\frac23(x-x') + \frac13(y-y') \right) \right]t + \alpha(x-x')t^2 + \frac{\alpha}{3}[(y-y') - (x - x')]t^3.
\end{gather*}
Since $\xi = \gamma - d$ we have
\begin{gather*}
    \xi(t) = x' + \left[y' - x' - \alpha\left(\frac23(x-x') + \frac13(y - y') \right) \right]t + \alpha(x-x')t^2 + \frac{\alpha}{3}[(y-y') - (x-x')]t^3.
\end{gather*}
We must now substitute these into the end-point point cost function \eqref{qot end-point cost} to get the infsup/supinf value. To this end, we define new variables:
\begin{gather*}
\begin{cases}
    D = \frac12(\gamma - \xi)\\
    S = \frac12(\gamma + \xi)
\end{cases} \Longleftrightarrow\quad
\begin{cases}
    \gamma = S + D\\
    \xi = S -D
\end{cases}
\end{gather*}
Then observe that the difference in kinetic energies of $\gamma$ and $\xi$ gives
\begin{gather*}
    \frac12|\dot{\gamma}|^2 - \frac12|\dot\xi|^2 = \frac12|\dot{S} + \dot{D}|^2 - \frac12|\dot{S} - \dot{D}|^2 = 2\dot{S}\dot{D},
\end{gather*}
while
\begin{gather*}
    \alpha|\gamma - \xi|^2 = 4\alpha |D|^2.
\end{gather*}
Then the saddle-point value is given by
\begin{gather*}
    c_{e}(x,y,x',y') = \int_0^1 2\dot{S}(t)\dot{D}(t) + 4\alpha |D(t)|^2\ dt.
\end{gather*}
We integrate by parts the first expression to obtain
\begin{align*}
    \int_0^1 2\dot{S}\dot{D}\ dt &= 2S\dot{D}\Big|_{0}^1 - 2\int S\ddot{D}\ dt = 2S\dot{D}\Big|_{0}^1 = 2\left(\frac{y +y'}{2} - \frac{x + x'}{2} \right)\left(\frac{y - y' - x + x'}{2} \right)\\
    &= \frac12\left((y - x) + (y' - x') \right)\left((y-x) - (y' - x') \right) = \frac12|x-y|^2 - \frac12|x'-y'|^2.
\end{align*}
where, according to Proposition \ref{minimal paths plus potential}, $\ddot{D}=0$.

For the second expression, we have
\begin{align*}
    \int_0^1 4\alpha |D(t)|^2\ dt &= \alpha \int_0^1\left|(x-x')(1-t) + (y-y')t \right|^2\ dt = \alpha \int_0^1 \left|(x-x') + (y-y'-x + x')t \right|^2\\
    &= \alpha \int_0^1 |x-x'|^2 + 2(x-x')\cdot(y-y' - x + x')t + |y-y'-x+x'|^2t^2\ dt \\
    &= \alpha|x-x'|^2 + \alpha(x-x')\cdot(y-y'-x+x') + \frac\alpha3|y-y'-x + x'|^2\\
    &= \alpha(x-x')\cdot (y - y') + \frac\alpha3(|y-y'|^2 - 2(y-y')\cdot(x-x') + |x - x'|^2)\\
    &= \frac\alpha3\left[|x-x'|^2 + (x-x')\cdot(y-y') + |y-y'|^2 \right].
\end{align*}
Finally, the end-point cost function is given as
\begin{align}\label{eqn:new cost}
    c_e(x,y,x',y') = \frac12|x-y|^2 - \frac12|x'-y'|^2 + \frac\alpha3\left[|x-x'|^2 + (x-x')\cdot(y-y') + |y-y'|^2 \right].
\end{align}
\end{proof}
A plot of several trajectories which realize the inf-sup/sup-inf are given in Figure \ref{fig:Trajectories}. Note that when $\alpha = 0$, the cost function becomes fully separable in the infimum and supremum arguments and thus the trajectories becomes straight line geodesics as expected from the classic theory. As $\alpha$ increases, the trajectories begin to curve either toward (or away) with respect to the other trajectory which represent the incentives of each agent to either pursue or evade the other one.

In the next section we apply Lemma \ref{thm:explicit cost} in order to obtain a Nash-Monge-Kantorovich map which corresponds to the stationary cost function. 

\begin{figure}
     \centering
     \begin{subfigure}[b]{0.49\textwidth}
         \centering
         \includegraphics[width=\textwidth]{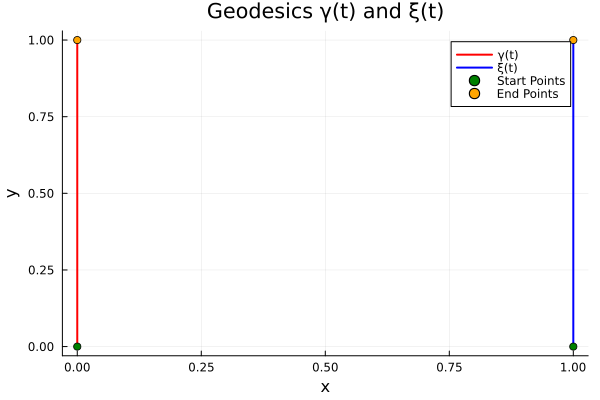}
         \caption{$\alpha = 0$}
         \label{fig:alpha0straight}
     \end{subfigure}
     \hfill
     \begin{subfigure}[b]{0.49\textwidth}
         \centering
         \includegraphics[width=\textwidth]{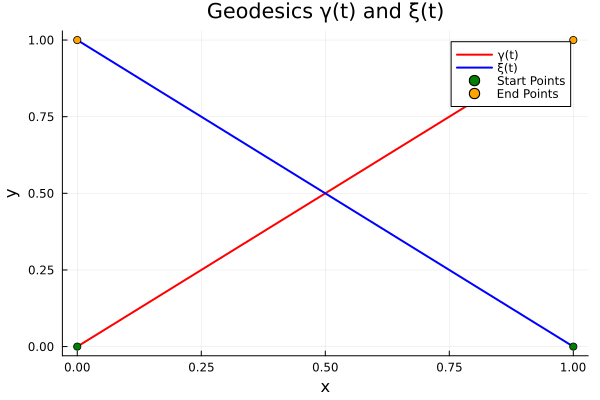}
         \caption{$\alpha = 0$}
         \label{fig:alpha0cross}
     \end{subfigure}
     \hfill
     \begin{subfigure}[b]{0.49\textwidth}
         \centering
         \includegraphics[width=\textwidth]{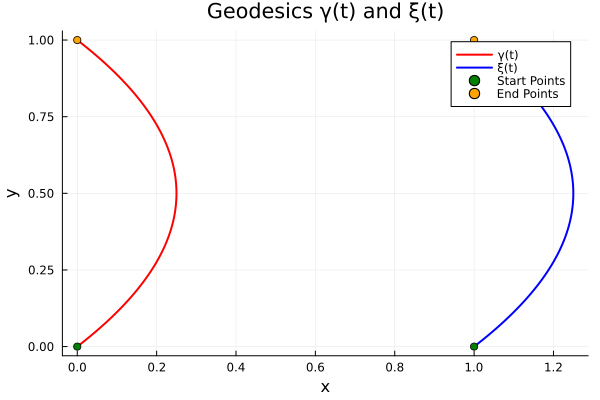}
         \caption{$\alpha = 1$}
         \label{fig:alpha1straight}
     \end{subfigure}
     \begin{subfigure}[b]{0.49\textwidth}
         \centering
         \includegraphics[width=\textwidth]{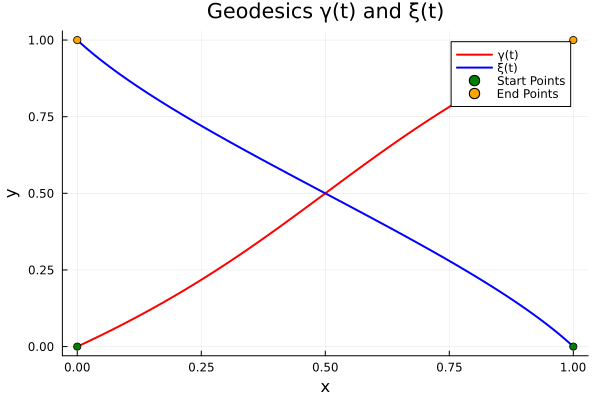}
         \caption{$\alpha = 1$}
         \label{fig:alpha1cross}
     \end{subfigure}
     \hfill
     \begin{subfigure}[b]{0.49\textwidth}
         \centering
         \includegraphics[width=\textwidth]{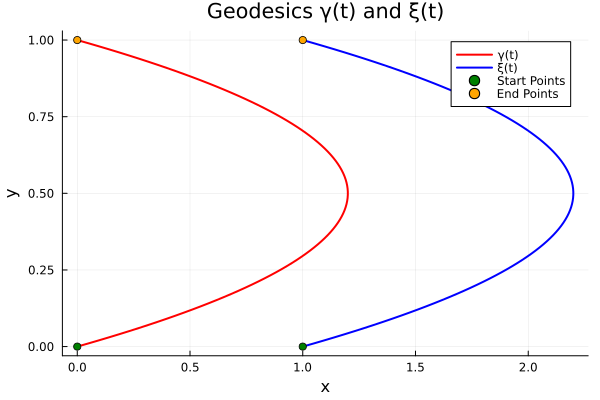}
         \caption{$\alpha = 4.8$}
         \label{fig:alpha4.8straight}
     \end{subfigure}
     \hfill
     \begin{subfigure}[b]{0.49\textwidth}
         \centering
         \includegraphics[width=\textwidth]{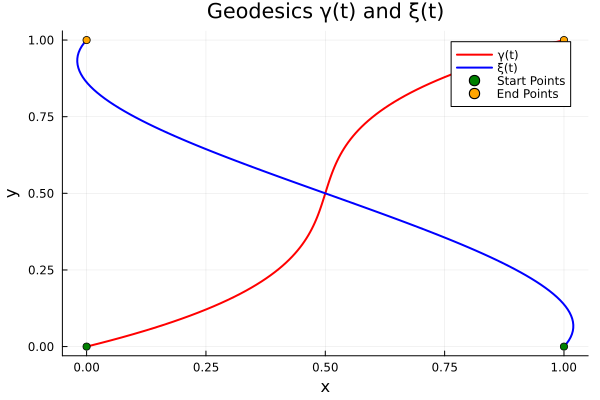}
         \caption{$\alpha = 4.8$}
         \label{fig:alpha4.8cross}
     \end{subfigure}
        \caption{Plot of representative saddle paths which solve the coupled Euler-Lagrange equations where initial and terminal points for each agent lie parallel (Top Row) or across (Bottom Row). \textcolor{red}{Pursuers} (left curves)  and \textcolor{blue}{Evaders} (right curves) are depicted for various interaction strengths $\alpha$.}
        \label{fig:Trajectories}
\end{figure}

\subsection{Nash-Monge-Kantorovich Transport Map}\label{section: Nash-Monge-Kantorovich Transport Map}
We will now study the stationary minimax problem (i.e. in Euclidean spaces) but we will later relate the solution to this problem to the solution to the original dynamical bilinear transport problem in Section \ref{section: dual}. By definition we have that the NETP induces a solution to
\begin{align*}
    \inf_{\pi_1\in \Pi_{\text{path}}(\mu_1,\nu_1)}&\sup_{\pi_2\in \Pi_{\text{path}}(\mu_2, \nu_2)} \int_{(\mathbb{R}^d \times \mathbb{R}^d)^2} c_e(x,y,x',y')\ d\pi_1d\pi_2\\
    &= \sup_{\pi_2\in \Pi_{\text{path}}(\mu_2, \nu_2)} \inf_{\pi_1\in \Pi_{\text{path}}(\mu_1,\nu_1)}\int_{(\mathbb{R}^d \times \mathbb{R}^d)^2} c_e(x,y,x',y')\ d\pi_1d\pi_2
\end{align*}
where $c_e$ is defined as in \eqref{qot end-point cost}. By Lemma \ref{thm:explicit cost} in the case of the cost function satisfying Condition \ref{condition 2} we can show that the transport plans are characterized by the existence of transport maps. Observe that transport plans in this setting exist due to Glicksberg's Theorem \cite{glicksberg1952further}.

\begin{proof}[Proof of Theorem \ref{thm:NashBrenier} for stationary case]
In this proof we compute the maps $T_{1}$ and $T_{2}$ for the stationary case with cost $c_{e}(x,y,x',y')$. In Section \ref{sec: bilinear transport maps} we take a rigorous approach to incorporate paths. This proof illustrates what the maps $T_{1}, T_{2}$ actually are in terms of convex/ concave potentials. The proof of Section \ref{sec: bilinear transport maps} demonstrates where these maps come from, namely, the minimal and maximal paths $\gamma, \xi$, respectively with associated NETP solutions $(\pi_{1}^{*}, \pi_{2}^{*})$.

Define the following ``effective" cost functions: for each fixed $\pi_2$,
    \begin{gather*}
        c_1(x,y;\pi_2) = \frac12|x-y|^2 - \int_{\mathbb{R}^d\times \mathbb{R}^d}\frac12|x'-y'|^2 -\frac\alpha3\left[|x-x'|^2 + (x-x')\cdot(y-y') + |y-y'|^2 \right]d\pi_2(x',y').
    \end{gather*}
    Similarly, for each fixed $\pi_1$
    \begin{gather*}
        c_2(x',y';\pi_1) = -\frac12|x'-y'|^2 + \int_{\mathbb{R}^d\times \mathbb{R}^d}\frac12|x-y|^2 +\frac\alpha3\left[|x-x'|^2 + (x-x')\cdot(y-y') + |y-y'|^2 \right]d\pi_1(x,y).
    \end{gather*}
    At equilibrium, the saddle-point $(\pi_1,\pi_2)$ satisfies the system
    \begin{gather*}
        \begin{cases}
            \inf_{\pi_1} \int_{\mathbb{R}^d\times \mathbb{R}^d} c_1(x,y;\pi_2)d\pi_1(x,y)\\
            \sup_{\pi_2} \int_{\mathbb{R}^d\times \mathbb{R}^d} c_2(x',y';\pi_1)d\pi_2(x',y')
        \end{cases}
    \end{gather*}
    Let us write $c_1(x,y;\pi_2)$ explicitly:
    \begin{align*}
        c_1(x,y;\pi_2) &= \frac12|x|^2 + \frac12|y|^2 - x\cdot y + \int_{\mathbb{R}^{d}\times\mathbb{R}^{d}} \Big(x'\cdot y' - \frac12|x'|^2 - \frac12|y'|^2 + \frac{\alpha}{3}|x|^2 + \frac{\alpha}{3}|x'|^2 - \frac{2\alpha}{3}x\cdot x'  \\ 
        &+ \frac{\alpha}{3} x\cdot y - \frac{\alpha}{3}x\cdot y' - \frac{\alpha}{3}x'\cdot y + \frac{\alpha}{3} x'\cdot y' + \frac{\alpha}{3}|y|^2 + \frac{\alpha}{3}|y'|^2 - \frac{2\alpha}{3}y\cdot y'\ \Big)d\pi_2(x',y') \\
        &= \frac{3+2\alpha}{6}(|x|^2 + |y|^2) + \frac{\alpha - 3}{3}x\cdot y - \frac{\alpha}{3}(2m_{x'} + m_{y'})\cdot x - \frac{\alpha}{3}(m_{x'} + 2m_{y'})\cdot y + K
    \end{align*}
    where $K:=\int_{\mathbb{R}^{d}\times\mathbb{R}^{d}}x'\cdot y'd\pi_{2}(x', y')$, which is a constant independent of $x$ and $y$, and furthermore
    \begin{gather*}
        m_{x'} := \int_{\mathbb{R}^d\times \mathbb{R}^d} x'\ d\pi_2(x',y') = \int_{\mathbb{R}^d} x' d\mu_2(x');\quad
        m_{y'} := \int_{\mathbb{R}^d\times \mathbb{R}^d} y'\ d\pi_2(x',y') = \int_{\mathbb{R}^d} y' d\nu_2(y').
    \end{gather*}
    Set $a = \frac{3+2\alpha}{6}$, $b = \frac{\alpha - 3}{3}$, $\ell_{x} = - \frac{\alpha}{3}(2m_{x'} + m_{y'})$, and $\ell_{y} = - \frac{\alpha}{3}(m_{x'} + 2m_{y'})$, and where the latter $\ell_{x}$, $\ell_{y}$ are the coefficients of $x$ and $y$, respectively. Then
    \begin{align*}
        c_1(x,y;\pi_2) &= a|x|^2 + a|y|^2 + (bx+\ell_y)\cdot y + \ell_x \cdot x + K\\
        &= a\left|y + \frac{bx + \ell_y}{2a} \right|^2 - \left(\frac{bx + \ell_y}{2a} \right)^2 + a|x|^2 + \ell_x\cdot x + K.
    \end{align*}
    
    Thus, we have that the minimization problem is equivalent to
    \begin{gather*}
        \inf_{\pi_1 \in \Pi(\mu_1,\nu_1)} a\int_{\mathbb{R}^d \times \mathbb{R}^d} \left|y + \frac{bx + \ell_y}{2a} \right|^2d\pi_1(x,y) = \inf_{\tilde{\pi}_1 \in \Pi(\tilde{\mu}_1,\nu_1)} a\int_{\mathbb{R}^d \times \mathbb{R}^d} \left|\tilde{x} - y \right|^2d\tilde{\pi}_1(\tilde{x},y) 
    \end{gather*}
    where $\tilde{x} = -\frac{b}{2a}x - \frac{\ell_y}{2a}$ and $\tilde{\mu}_1 = \tilde{x}_\#(\mu_1)$. 

    Next, fix $\pi_1\in\Pi(\mu_1,\nu_1)$ and recall the effective cost
    \begin{gather*}
        c_2(x',y';\pi_1)
        =-\frac12|x'-y'|^2+\int_{\mathbb{R}^d\times\mathbb{R}^d}
        \frac12|x-y|^2+\frac{\alpha}{3}\Big(|x-x'|^2+(x-x')\cdot(y-y')+|y-y'|^2\Big)\,d\pi_1(x,y).
    \end{gather*}
    For the supremum problem, we define
    \begin{gather*}
        \tilde c_2(x',y';\pi_1):=-c_2(x',y';\pi_1).
    \end{gather*}
    Then $\sup_{\pi_2}\int c_2\,d\pi_2$ is equivalent to $\inf_{\pi_2}\int \tilde c_2\,d\pi_2$, and the optimizer $\pi_2^\ast$ is unchanged. Writing $\tilde c_2$ explicitly and expanding the quadratic terms and collecting those that depend on $(x',y')$, we obtain
    \begin{gather*}
        \tilde c_2(x',y';\pi_1) =\frac12|x'|^2+\frac12|y'|^2-x'\cdot y' -\int_{\mathbb{R}^d\times\mathbb{R}^d}
        \frac12|x|^2+\frac12|y|^2-x\cdot y
        +\frac{\alpha}{3}|x|^2+\frac{\alpha}{3}|x'|^2-\frac{2\alpha}{3}x\cdot x'
        \\+\frac{\alpha}{3}x\cdot y-\frac{\alpha}{3}x\cdot y'
        -\frac{\alpha}{3}x'\cdot y+\frac{\alpha}{3}x'\cdot y' +\frac{\alpha}{3}|y|^2+\frac{\alpha}{3}|y'|^2-\frac{2\alpha}{3}y\cdot y'
        \,d\pi_1(x,y).
    \end{gather*}
    All terms involving only $(x,y)$ integrate to a constant independent of $(x',y')$. Using the first moments
    \begin{gather*}
        m_x:=\int_{\mathbb{R}^d\times\mathbb{R}^d}x\,d\pi_1(x,y)=\int_{\mathbb{R}^d}x\,d\mu_1(x),
        \qquad
        m_y:=\int_{\mathbb{R}^d\times\mathbb{R}^d}y\,d\pi_1(x,y)=\int_{\mathbb{R}^d}y\,d\nu_1(y),
    \end{gather*}
    we simplify the $(x',y')$-dependent part of $\tilde c_2$ to
    \begin{gather*}
        \tilde c_2(x',y';\pi_1)
        =\frac{3+2\alpha}{6}\big(|x'|^2+|y'|^2\big)+\frac{\alpha-3}{3}\,x'\cdot y'
        +\frac{\alpha}{3}(2m_x+m_y)\cdot x'
        +\frac{\alpha}{3}(m_x+2m_y)\cdot y'
        +K',
    \end{gather*}
    where $K'$ is a constant independent of $x'$ and $y'$. Set
    \begin{gather*}
        a:=\frac{3+2\alpha}{6},
        \qquad
        b:=\frac{\alpha-3}{3},
        \qquad
        \ell_{x'}:=\frac{\alpha}{3}(2m_x+m_y),
        \qquad
        \ell_{y'}:=\frac{\alpha}{3}(m_x+2m_y).
    \end{gather*}
    Then
    \begin{gather*}
        \tilde c_2(x',y';\pi_1)
        =a|x'|^2+a|y'|^2+(bx'+\ell_{y'})\cdot y'+\ell_{x'}\cdot x' +K'.
    \end{gather*}
    Completing the square in $y'$ gives
    \begin{gather*}
        \tilde c_2(x',y';\pi_1)
        =a\left|y'+\frac{bx'+\ell_{y'}}{2a}\right|^2
        -\left|\frac{bx'+\ell_{y'}}{2a}\right|^2
        +a|x'|^2+\ell_{x'}\cdot x' +K'.
    \end{gather*}
    Hence the minimization problem
    \begin{gather*}
        \inf_{\pi_2\in\Pi(\mu_2,\nu_2)}\int_{\mathbb{R}^d\times\mathbb{R}^d}\tilde c_2(x',y';\pi_1)\,d\pi_2(x',y')
    \end{gather*}
    is equivalent (up to $x'$-dependent terms and constants) to
    \begin{gather*}
        \inf_{\pi_2\in\Pi(\mu_2,\nu_2)}
        a\int_{\mathbb{R}^d\times\mathbb{R}^d}
        \left|y'+\frac{bx'+\ell_{y'}}{2a}\right|^2\,d\pi_2(x',y').
    \end{gather*}
    Define the affine change of variables
    \begin{gather*}
        \tilde x'
        :=-\frac{b}{2a}x'-\frac{\ell_{y'}}{2a},
        \qquad
        \tilde\mu_2:=\tilde x'_{\#}\mu_2.
    \end{gather*}
    Then the above minimization problem is equivalent to a standard linear transport  problem:
    \begin{gather*}
        \inf_{\tilde\pi_2\in\Pi(\tilde\mu_2,\nu_2)}
        a\int_{\mathbb{R}^d\times\mathbb{R}^d}
        \left|\tilde x'-y'\right|^2\,d\tilde\pi_2(\tilde x',y').
    \end{gather*}
    Observe first that $a > 0$ when $\alpha > 0$ and that this new cost is the standard quadratic cost from optimal transport theory. As $\tilde{\mu}_1$ is absolutely continuous since it is the push-forward by an affine mapping, then by Brenier's Theorem \cite{brenier1991polar} there exists a Monge map given by $\tilde{T}_1(\tilde{x}) = \nabla p_1(\tilde{x})$ where $p_1$ is convex. In particular, pulling back gives
    \begin{align*}
        T_1(x) &= \tilde{T}_1(\tilde{x}(x)) = \nabla p_1\left(\frac{(3-\alpha)/3}{2(3+2\alpha)/6}x + \frac{\alpha(m_{x'} + 2m_{y'})/3)}{2(3+2\alpha)/6}\right)\\
        &= \nabla p_1\circ \left(\frac{3-\alpha}{3+2\alpha}x + \frac{\alpha(m_{x'} + 2m_{y'})}{3+2\alpha}\right).
    \end{align*}
    Define $u_1 = \frac{3+2\alpha}{3-\alpha}p_1(\tilde{x}(x))$. Since $p_1$ is convex and $\tilde{x}$ is affine, then if $\alpha < 3$ it follows that $u_1$ is convex. Moreover
    \begin{gather*}
        \nabla u_1(x) = \frac{3-\alpha}{3+2\alpha}\cdot \frac{3+2\alpha}{3-\alpha}\nabla p_1(\tilde{x}(x)) = \nabla p_1(\tilde{x}(x))= T_1(x).
    \end{gather*}
    Therefore, the optimal map $T_1$ is the gradient of a convex function.

    Similarly by Brenier's Theorem, the optimizer $\tilde\pi_2^\ast$ is induced by a map
    \begin{gather*}
        y'=\tilde{T}_2(x')=\nabla p_2(\tilde x')\quad\text{for some convex }p_2,
    \end{gather*}
    and $\pi_2^\ast$ is the pushforward of $\tilde\pi_2^\ast$ under $(x',y')\mapsto(\tilde x'(x'),y')$. In particular, pulling back yields
    \begin{gather*}
        T_2(x')
        =\tilde{T}_2(\tilde{x}'(x'))
        =\nabla p_2\left(
            \frac{(3-\alpha)/3}{2(3+2\alpha)/6}x'
            -\frac{\alpha(m_x+2m_y)/3}{2(3+2\alpha)/6}
        \right)\\
        =\nabla p_2\circ
        \left(
            \frac{3-\alpha}{3+2\alpha}x'
            -\frac{\alpha(m_x+2m_y)}{3+2\alpha}
        \right).
    \end{gather*}
    Define
    \[
        u_2(x')
        :=\frac{3+2\alpha}{3-\alpha}\,
        p_2(\tilde{x}'(x')).
    \]
    Since $p_2$ is convex and $\tilde{x}'$ is affine, it follows that $u_2$ is convex whenever $\alpha<3$. Moreover,
    \begin{gather*}
        \nabla u_2(x')
        =\frac{3-\alpha}{3+2\alpha}\cdot
        \frac{3+2\alpha}{3-\alpha}\,
        \nabla p_2(\tilde{x}'(x'))
        =\nabla p_2(\tilde{x}'(x'))
        =T_2(x').
    \end{gather*}
    Therefore, for $\alpha<3$, the optimal map $T_2$ is also the gradient of a convex function. A similar argument shows that for $\alpha > 3$, the optimal maps are given by gradients of concave functions.
\end{proof}

\begin{remark}
    
When $\alpha = 3$, then observe that the effective cost functions are of the form
\begin{gather*}
    a|x|^2 + a|y|^2 + \ell_y\cdot y + \ell_x \cdot x + K \quad \text{and}\quad a'|x'|^2 + a'|y'|^2 + \ell_{y'}\cdot y' + \ell_{x'} \cdot x' + K'.
\end{gather*}
In this case, the bilinear transport plan is not unique as every coupling is admissible and yields the same value.
\end{remark}

Finally, we conclude this section with the following corollary which follows from the standard derivations as in \cite{villani2008optimal}:
\begin{corollary}
    Under the same conditions as Theorem \ref{thm:NashBrenier}, when $\alpha \neq 3$ the optimal maps $p_1$, $p_2$ satisfy the following coupled Monge-Amp\`ere system:
    \begin{gather*}
    \begin{cases}
        \left|\operatorname{det}D^2p_1\left(\frac{3-\alpha}{3+2\alpha}x + \frac{\alpha(m_{x'} + 2m_{y'})}{3+2\alpha} \right)\right| = \left(\frac{3+2\alpha}{3-\alpha}\right)^d\frac{\mu_1(x)}{\nu_1\left(\frac{3-\alpha}{3+2\alpha}x + \frac{\alpha(m_{x'} + 2m_{y'})}{3+2\alpha} \right)}\\
        \left|\operatorname{det}D^2p_2\left(\frac{3-\alpha}{3+2\alpha}x' - \frac{\alpha(m_{x} + 2m_{y})}{3+2\alpha} \right)\right| = \left(\frac{3+2\alpha}{3-\alpha}\right)^d\frac{\mu_2(x')}{\nu_2\left(\frac{3-\alpha}{3+2\alpha}x' - \frac{\alpha(m_{x} + 2m_{y})}{3+2\alpha} \right)}\\
        m_y = \int_{\mathbb{R}^d} \nabla_xp_1\left(\frac{3-\alpha}{3+2\alpha}x + \frac{\alpha(m_{x'} + 2m_{y'})}{3+2\alpha} \right) \ d\mu_1(x)\\
        m_{y'} = \int_{\mathbb{R}^d} \nabla_{x'}p_2\left(\frac{3-\alpha}{3+2\alpha}x' - \frac{\alpha(m_{x} + 2m_{y})}{3+2\alpha} \right) \ d\mu_2(x')\\
        p_1,p_2\ convex
    \end{cases}
    \end{gather*}
    where the coupling occurs through the moments $m_{x'}$ and $m_{y'}$.
\end{corollary}
From this, we can now apply Caffarelli's regularity theory \cite{Caffarelli1991, Caffarelli1992a, Caffarelli1992b, Caffarelli2000} to deduce the smoothness of the transport maps. In particular, if there exist $\alpha_1,\alpha_2>1$ such that $\mu_1,\nu_1 \in C^{\alpha_1 - 1}(\Omega_1)$ and $\mu_2,\nu_2 \in C^{\alpha_2 - 1}(\Omega_2)$ are absolutely continuous with respect to the Lebesgue measure and there exist constants $\lambda_1,\lambda_2 > 0$ such that
\begin{gather*}
    \lambda_1^{-1} \leq \mu_1,\nu_1\leq \lambda_1,\quad \forall x \in \Omega_1^{\mathrm{o}}
\end{gather*}
and
\begin{gather*}
    \lambda_2^{-1} \leq \mu_2,\nu_2\leq \lambda_2,\quad \forall x \in \Omega_2^{\mathrm{o}}
\end{gather*}
then we have the interior regularity estimates such that $p_1 \in C^{1+\alpha_1}(\Omega_1^{\mathrm{o}})$, $p_2 \in C^{1+\alpha_2}(\Omega_2^\mathrm{o})$. Under suitable assumptions on the boundaries of the supports, we may obtain similar global regularity of the potentials. Corresponding Calder\'on-Zygmund type estimates also hold \cite{dephilippis2013} for appropriate right-hand sides. A more comprehensive survey of regularity results on Monge-Amp\`ere equations can be found in \cite{figalli2017monge}.

\section{Dual problem to Minimax Bilinear Transport Problem}\label{section: dual} We aim to develop a dual problem to our minimax, or maximin, bilinear transport problem, (\ref{inf sup problem}), in particular when the cost is given as in Condition \ref{condition 2}. One aspect which we will discuss in this section that differentiates the dual formulation here from the classical optimal transport problem is that in fact we will have coupled dual problems. Moreover, we can arrive at the same conclusion of the existence of NMK (Definition \ref{def: NMK}) transport maps through this formulation. In addition, we recover classical results from optimal transport in the setting of our bilinear transport problem. In other words, we acquire a coupled solution to the dual problem, the characterization of minimax (or maximin) solutions through cyclical monotonicity of Knott and Smith \cite{villani2021topics}, and a construction to a coupled classical Kantorovich solution with respect to a certain pair of marginals with respect to a certain coupled \emph{effective cost}. 

In the subsequent sections, we closely parallel the standard theory of optimal transport, as well as the developments in the geometry of paths and congestion, within the framework of bilinear transport theory.

\subsection{Effective Cost and Cyclical Monotonicity} \label{sec: effective cost and cyclical monotonicity} In this section, we consider what will be called an effective cost function. Let $c$ be defined as in Condition \ref{condition 2} and recall the end-point cost function \eqref{eqn: end-point cost function} from Sections \ref{sec: stationary} and \ref{section: Nash-Monge-Kantorovich Transport Map}. 
Furthermore, recall the action functional \eqref{Lagrangian cost} with its corresponding Lagrangian \eqref{kinetic potential energy}. The pair of \emph{effective cost functions}
on $\gamma$ and $\xi$ for fixed $\pi_{1}$ and $\pi_{2}$ respectively, are defined as follows. If $\pi_{2}^{*}$ is a maximizer of $\eqref{inf sup problem}$ and if $\pi_{1}^{*}$ is a minimizer of \eqref{inf sup problem}, then the \emph{effective costs} of $\pi_{2}^{*}$ and $\pi_{1}^{*}$, respectively, are
\begin{align}\label{eqn: eff path}
\begin{split}
    c_{\pi_{2}^{*}}(\gamma)&:=\int_{\Omega_{2}}c(\gamma, \xi)d\pi_{2}^{*}(\xi),\quad\text{for fixed}\quad\gamma\in\Omega_{1},\;\forall\xi\\
    c_{\pi_{1}^{*}}(\xi)&:=\int_{\Omega_{1}}c(\gamma, \xi)d\pi_{1}^{*}(\gamma),\quad\text{for fixed}\quad\xi\in\Omega_{2},\;\forall\gamma.
    \end{split}
\end{align}
Furthermore, their \textit{end-point cost} functions will be defined by
\begin{align}\label{eqn: effective costs}
\begin{split}
    c_{e, \pi_{2}^{*}}(x,y)&:=\inf\left\{c_{\pi_{2}^{*}}(\gamma):\;\gamma(0)=x, \gamma(1)=y\right\}\\
    c_{e, \pi_{1}^{*}}(x',y')&:=\sup\left\{c_{\pi_{1}^{*}}(\xi):\;\xi(0)=x', \xi(1)=y'\right\}.
    \end{split}
\end{align}
The \emph{effective cost function} is associated to the Lagrangian \eqref{Lagrangian cost}-\eqref{kinetic potential energy}, $c(\gamma, \xi)$ (see Lemma \ref{lemma:saddlepoint-endpoint}).
Then its corresponding collective \emph{end-point cost function}, $c_{e}(x,y,x',y')$, is given by \eqref{eqn: end-point cost function} as before.

The next result allows us to construct a pair of cost functions that will be used to formulate a dual problem. See Theorem \ref{thm:coupled duality}.

\begin{lemma}\label{lemma:saddlepoint-endpoint}
Let $c_{e}$ be the end-point cost function given by \eqref{qot end-point cost}. For a fixed pair of NETP's $(\pi_{1}^{*}, \pi_{2}^{*})$ to \eqref{inf sup problem}, define the following Euclidean measures through the coupled evaluation map, $(e_{0}, e_{1})_{\sharp}\pi_{1}^{*}:=\check{\pi}_{1}$ and $(e_{0}, e_{1})_{\sharp}\pi_{2}^{*}:=\check{\pi}_{2}$.  
We have the following equalities
\begin{align}\label{euc integral}
    c_{e, \pi_{2}^{*}}(x, y)=\int_{\mathbb{R}^{d}\times\mathbb{R}^{d}}c_{e}(x,y,x',y')d\check{\pi}_{2}(x',y'),
\end{align}
and
\[
\iint_{(\mathbb{R}^{d}\times\mathbb{R}^{d})^{2}}c_{e}(x,y,x',y')d\check{\pi}_{1}(x,y)d\check{\pi}_{2}(x',y')=\int_{\mathbb{R}^{d}\times\mathbb{R}^{d}}\inf_{\gamma}c_{\pi_{2}^{*}}(\gamma)d\check{\pi}_{1}(x,y).
\]
Similar results hold for $c_{e, \pi_{1}^{*}}(x', y')$ and
\[
\iint_{(\mathbb{R}^{d}\times\mathbb{R}^{d})^{2}}c_{e}(x,y,x',y')d\check{\pi}_{1}(x,y)d\check{\pi}_{2}(x',y')=\int_{\mathbb{R}^{d}\times\mathbb{R}^{d}}\sup_{\xi}c_{\pi_{1}^{*}}(\xi)d\check{\pi}_{2}(x',y').
\]
\end{lemma}
\begin{proof}
Indeed, to prove \eqref{euc integral}, by Lemma \ref{thm:explicit cost} there exists a unique $\gamma^{*}$ that minimizes the first equation of \eqref{eqn: effective costs}, and a unique $\xi^{*}$ that maximizes the second equation of \eqref{eqn: effective costs}. According to Lemma \ref{lem:support_saddle}, which will be proved later, for any  pairs $(\gamma, \xi) \subset \text{spt}(\pi_{1}^{*})\otimes\text{spt}(\pi_{2}^{*})$,  then $\gamma^{*}=\gamma$ and $\xi^{*}=\xi$. 

We will first show that 
\[
\inf_{\gamma}\int_{\Omega_{2}}c(\gamma, \xi^{*})d\pi^{*}_{2}(\xi^{*})=\int_{\Omega_{2}}\inf_{\gamma}c(\gamma, \xi^{*})d\pi_{2}^{*}(\xi^{*}).
\]
To that end, let $m = \inf_\gamma \int c(\gamma,\xi) d\pi_2(\xi)$. By monotonicity of the integral we have the following inequality
\begin{gather*}
    m = \inf_\gamma \int_{\Omega_2} c(\gamma, \xi^*) d\pi^{*}_2(\xi^*) \geq \int_{\Omega_2} \inf_\gamma c(\gamma,\xi^*)d\pi_2^*(\xi^*) =: \hat{m}
\end{gather*}
So since $\gamma^* = \operatorname{argmin} c(\gamma, \xi^*)$ is admissible, choose $\gamma_n = \gamma^*$ for all $n$.
Then
\begin{gather*}
    \hat{m} = \int \inf_\gamma c(\gamma, \xi^{*}) d\pi^{*}_2(\xi^*) = \int c(\gamma^*, \xi^{*}) d\pi^{*}_2(\xi^*) = \int c(\gamma_n, \xi^{*})\ d\pi^{*}_2(\xi^*),\quad \forall n \in \mathbb{N}.
\end{gather*}
Taking limits and applying lower semicontinuity implies that since $\gamma_n \to \gamma^*$ trivially we have
\begin{gather*}
    \hat{m} = \int_{\Omega_2} c(\gamma^*,\xi^*) d\pi_2^*(\xi^*) \geq \inf_\gamma \int_{\Omega_2} c(\gamma, \xi^*) d\pi^*_2(\xi^*) = m,
\end{gather*}
achieving the desired equality.

Now to complete the proof, compute the right hand side of \eqref{euc integral}, while we apply what we just proved, and apply \eqref{qot end-point cost} to get
\begin{align*}
    \int_{\mathbb{R}^{d}\times\mathbb{R}^{d}}c_{e}(x,y,x',y')d\check{\pi}_{2}(x',y')&=\int_{\mathbb{R}^{d}\times\mathbb{R}^{d}}\left(\inf_{\gamma}\sup_{\xi}c(\gamma, \xi)\right)d(e_{0}, e_{1})_{\sharp}\pi_{2}^{*}(\xi) \\
    &=\int_{\Omega_{2}} \inf_\gamma c(\gamma, \xi^*) d\pi_2^*(\xi^*) \\
    &=\inf_{\gamma}\int_{\Omega_{2}}c(\gamma, \xi^{*})d\pi_{2}^{*}(\xi^{*}).
\end{align*}
Then we have
\begin{align*}
    \iint_{(\mathbb{R}^{d}\times\mathbb{R}^{d})^{2}}c_{e}(x,y,x',y')d\check{\pi}_{1}(x,y)d\check{\pi}_{2}(x',y')&=\int_{\mathbb{R}^{d}\times\mathbb{R}^{d}}\int_{\mathbb{R}^{d}\times\mathbb{R}^{d}}\left(\inf_{\gamma}\sup_{\xi}c(\gamma,\xi)\right)d{\check{\pi}}_{2}(x',y')d\check{\pi}_{1}(x,y)\\
    &=\int_{\mathbb{R}^{d}\times\mathbb{R}^{d}}\inf_{\gamma}\left(\int_{\Omega_{2}}c(\gamma, \xi^{*})d\pi_{2}^{*}(\xi^{*})\right)d\check{\pi}_{1}(x,y)\\
    &=\int_{\mathbb{R}^{d}\times\mathbb{R}^{d}}\inf_{\gamma}c_{\pi_{2}^{*}}\;d\check{\pi}_{1}(x,y).
\end{align*}
The first equation follows form the definition of the saddle point cost function \eqref{Lagrangian cost}-\eqref{kinetic potential energy}. The second equation follows from the push forward, and the equality $m=\hat{m}$. 
Identical arguments give the other case with respect to $c_{e, \pi_{1}^{*}}$. 

\end{proof}

We now recall a classical construction known as \textit{cyclical monotonicity} and adapt a standard notion of \emph{cyclical anti-monotonicity} of which Rockafeller laid the groundwork for these concepts in \cite{rockafellar1970convex}. Namely, the latter concept stems from the following: A mapping $S$ is \emph{anti-monotone} if $-S$ is monotone.  In particular, the following definition will be applied to effective costs \eqref{eqn: effective costs}.
\begin{definition}\label{def. cyclical monotone}
Fix $(\pi_{1}, \pi_{2})\in \Pi_{\text{path}}(\mu_{1}, \nu_{1})\otimes\Pi_{\text{path}}(\mu_{2}, \nu_{2})$, not necessarily a NETP. A finite set $\{(x_i,y_i)\}_{i=1}^n$ is called \emph{$c_{e, \pi_{2}^{*}}$-cyclically monotone}  if the following inequality
\[
\sum_{i=1}^n c_{e, \pi_{2}^{*}}(x_i, y_i)
\;\le\;
\sum_{i=1}^n c_{e, \pi_{2}^{*}}\bigl(x_i,y_{i+1}\bigr)
\]
holds, and a finite set $\{(x_{i}^{\prime}, y_{i}^{\prime})\}^{n}_{i=1}$ is called $c_{e, \pi_{1}^{*}}$-\emph{cyclically anti-monotone} if the following holds
\[
\sum_{i=1}^n c_{e, \pi_{1}^{*}}(x_{i}^{\prime}, y_{i}^{\prime})
\;\geq\;
\sum_{i=1}^n c_{e, \pi_{1}^{*}}\bigl(x'_i,y'_{i+1}\bigr)
\]

Here $y_{n+1}=y_{1}$ and $y'_{n+1}=y'_{1}$.
\end{definition}

In the classical theory of optimal transport, when the cost is quadratic $c_{Q}=|x-y|^{2}$, a necessary and sufficient condition for an optimal solution to exist is that of $c_{Q}$-cyclical monotonicity \cite[Ch 1.2]{ambrosio2013user}. The first author in \cite{cabrera2022optimaltransportationprincipleinteracting} extends this result to include paths and congestion. In the current manuscript we apply the latter to coupled Kantorovich potentials.

\subsection{Potentials} The effective costs \eqref{eqn: effective costs} will prove essential to define coupled potentials with corresponding $c_{e, \pi_{i}^{*}}$-transform ($i=1,2$).
We recall from standard transport that the Monge relaxation problem (which we call the primal problem) has an associated dual problem known as the Kantorovich dual problem. The dual problem typically considers an alternative formulation involving supremums over potential functions $\varphi,\psi$. In the case of bilinear minimax transport, the corresponding ``dual" problem will involve Kantorovich potentials $\varphi_1,\psi_1,\varphi_2,\psi_2$. 

To each pair of \textit{Monge-Kantorovich bilinear potentials}, $\varphi_{1}, \psi_{1}$ and $\varphi_{2}, \psi_{2}$ we can, and we do, apply their respective $c_{e, \pi_{2}^{*}}$-transform and $c_{e, \pi_{1}^{*}}$-transform. Indeed, any function $\psi_{1}:X\to\mathbb{R}\cup\{\infty\}$ induces its $c_{e, \pi_{2}^{*}}$-transform $\psi_{1}^{c_{e, \pi_{2}^{*}}}: X\to \mathbb{R}\cup\{\infty\}$  which is defined by
\begin{align}\label{eqn:c-transform x}
    \psi_{1}^{c_{e, \pi_{2}^{*}}}(x)=\inf_{y\in X}\{c_{e, \pi_{2}^{*}}(x, y)-\psi_{1}(y)\}.
\end{align}

Similarly, any function $\varphi_{1}:X\to\mathbb{R}\cup\{\pm\infty\}$ induces its $c_{e, \pi_{2}^{*}}$-transform $\varphi_{1}^{c_{e, \pi_{2}^{*}}}:X\to\mathbb{R}\cup\{\infty\}$ defined by
\begin{align}\label{eqn:c-transform y}
    \varphi_{1}^{c_{e, \pi_{2}^{*}}}(y)=\inf_{x\in X}\{c_{e, \pi_{2}^{*}}(x, y)-\varphi_{1}(x)\}.
\end{align}

Analogously, given any function $\psi_{2}: X\to \mathbb{R}\cup\{\pm\infty\}$, the $c_{e, \pi_{1}^{*}}$-transform of $\psi_{2}$ is $\psi_{2,c_{e, \pi_{1}^{*}}}:X\to\mathbb{R}\cup\{\infty\}$ defined by 
\begin{align}\label{eqn:c-transform 1}
    \psi_{2,c_{e, \pi_{1}^{*}}}(x')=\sup_{y\in X}\{-c_{e, \pi_{1}^{*}}(x^{\prime}, y^{\prime})-\psi_{2}(y')\}.
\end{align}
In the same way, one can define the $c_{e, \pi_{1}^{*}}$-transform functions of $\varphi_{2}$ on $X$.

\begin{definition}\label{c-convex} 
    A function $\varphi_{1}:X\to \mathbb{R}\cup\{-\infty\}$ is $c_{e, \pi_{2}^{*}}$-\textit{concave} if there exists $\psi_{1}:X\to\mathbb{R}\cup\{-\infty\}$ such that $\varphi_{1}=\psi_{1}^{c_{e, \pi_{2}^{*}}}$.

    By symmetry, $\varphi_{2}:X\to \mathbb{R}\cup\{+\infty\}$ is $c_{e, \pi_{1}^{*}}$-\textit{convex}, if there exists $\psi_{2}:X\to\mathbb{R}\cup\{+\infty\}$ such that $\varphi_{2}=\psi_{2,c_{e, \pi_{1}^{*}}}$.
\end{definition}
We now apply the standard theory of optimal transport from  Ambrosio's and Gigli's guide in \cite{ambrosio2013user} associated to \textit{superdifferentiability} to the pair $(\varphi_1,\psi_1)$. Let us write and study such properties. For $\varphi_{1}\in L^{1}(d\mu_{1})$ and $\psi_{1}\in L^{1}(d\nu_{1})$, we have the inequality $\varphi_{1}(x)+\psi_{1}(y)\leq c_{e, \pi_{2}^{*}}(x, y)$. Then the ``concavity" \textit{transforms} from Definition \ref{c-convex} that we will now compute, namely \eqref{eqn:c-transform y} and \eqref{eqn:c-transform x} maximize the linear functional,
\[
\mathcal{J}(\varphi_{1},\psi_{1}):=\int_{\mathbb{R}^{d}}\varphi_{1}(x)d\mu_{1}(x)+\int_{\mathbb{R}^{d}}\psi_{1}(y)d\nu_{1}(y).
\] 
Given a pair $(\varphi_{1}, \psi_{1}) \in C_{b}(X\times Y)$, for all $x$ and $y$ we have the standard inequality
\begin{align*}
    \psi_{1}(y)\leq c_{e, \pi_{2}^{*}}(x, y)-\varphi_{1}(x).
    \end{align*}
  Taking the infimum of the above inequality with respect to $x$ gives,
  \begin{align}\label{ineq: inf psi aux}
          \psi_{1}(y)\leq \inf_{x\in X}[c_{e, \pi_{2}^{*}}(x, y)-\varphi_{1}(x)]:=\varphi_{1}^{c_{e, \pi_{2}^{*}}}(y)
      \end{align}
  Similarly, for all $x$ and $y$
    \begin{align*}
          \varphi_{1}(x)\leq c_{e, \pi_{2}^{*}}(x, y)-\psi_{1}(y),
      \end{align*}
  taking the infimum on the right of the above inequality with respect to $y$, we get
  \begin{align}\label{ineq: inf phi aux}
          \varphi_{1}(x)\leq \inf_{y \in X}[c_{e, \pi_{2}^{*}}(x, y)-\psi_{1}(y)]:=\psi_{1}^{c_{e, \pi_{2}^{*}}}(x).
      \end{align}

Then we have        
  \begin{align*}
          \mathcal{J}(\varphi_{1},\varphi_{1}^{c_{e, \pi_{2}^{*}}})\geq \mathcal{J}(\varphi_{1}, \psi_{1}).
      \end{align*}
 From (\ref{ineq: inf psi aux}) and (\ref{ineq: inf phi aux}) we have
  \begin{align*}
      (\varphi_{1}^{c_{e, \pi_{2}^{*}}})^{c_{e, \pi_{2}^{*}}}(x)=:\inf_{y\in X}[c_{e, \pi_{2}^{*}}(x, y)-\varphi_{1}^{c_{e, \pi_{2}^{*}}}(y)]\geq \varphi_{1}(x),
  \end{align*}
  thus
  \begin{align*}
          \mathcal{J}((\varphi_{1}^{c_{e, \pi_{2}^{*}}})^{c_{e, \pi_{2}^{*}}}, \varphi_{1}^{c_{e, \pi_{2}^{*}}})\geq \mathcal{J}(\varphi_{1},\varphi_{1}^{c_{e, \pi_{2}^{*}}})\geq \mathcal{J}(\varphi_{1}, \psi_{1}).
      \end{align*}     
Therefore, the pair $\left((\varphi_{1}^{c_{e, \pi_{2}^{*}}})^{c_{e, \pi_{2}^{*}}}, \varphi_{1}^{c_{e, \pi_{2}^{*}}}\right)$ maximizes $\mathcal{J}(\varphi, \psi)$.     
This shows that in looking for a maximizing pair for the dual problem \eqref{eqn:new dual pi_2} we can focus on pairs $\varphi_{1},\psi_{1}$ that satisfy Definition \ref{c-convex}. Identical computations also hold for the pair $(\varphi_{2}, \psi_{2})\in C_{b}(X'\times Y')$ with cost $c_{e, \pi_{1}^{*}}(x^{\prime}, y^{\prime})$ defined in \eqref{eqn:cost pi_1} satisfying the constraint $-\varphi_{2}(x')-\psi_{2}(y')\geq c_{e, \pi_{1}^{*}}(x^{\prime}, y^{\prime})$ for its corresponding dual problem \eqref{eqn:new dual pi_1}. 
 
 The \textit{$c_{e, \pi_{2}^{*}}$-superdifferential} set defined for a $c_{e, \pi_{2}^{*}}$-concave function $\varphi_{1}$ is:
 \begin{align*}
          \partial^{c_{e, \pi_{2}^{*}}}\varphi_{1}:=\left\{(x,y) \in X\times X:\; \varphi_{1}(x)+\varphi_{1}^{c_{e, \pi_{2}^{*}}}(y)=c_{e, \pi_{2}^{*}}(x, y)\right\}.
      \end{align*}

 Identical definitions of $c_{e, \pi_{1}^{*}}$-convex functions $\varphi_{2}$ corresponding to the other dual problem \eqref{eqn:new dual pi_1} follow. Indeed,     
the \emph{$c_{e, \pi_{1}^{*}}$-subdifferential} $\partial_{c_{e, \pi_{1}^{*}}}$ of a $c_{e, \pi_{1}^{*}}$-convex function
$\varphi_{2} : X \to \mathbb{R} \cup \{+\infty\}$ is defined analogously by
\[
\partial_{c_{e, \pi_{1}^{*}}}\varphi_{2}
:= \left\{ (x',y') \in X \times X \;:\;
\varphi_{2}(x') + \varphi_{2}^{c_{e, \pi_{1}^{*}}}(y') = -c_{e, \pi_{1}^{*}}(x^{\prime}, y^{\prime}) \right\}.
\]
An equivalent characterization of the $c_{e, \pi_{2}^{*}}$-superdifferential (and similarly for the $c_{e, \pi_{1}^{*}}$-subdifferential) is the following:
\[
y \in \partial^{c_{e, \pi_{2}^{*}}} \varphi_{1}(x)
\quad \Longleftrightarrow \quad
\begin{cases}
\varphi_{1}(x) = c_{e, \pi_{2}^{*}}(x, y) - \varphi_{1}^{c_{e, \pi_{2}^{*}}}(y), &(x,y) \in \operatorname{spt}(\mu_1)\times \operatorname{spt}(\nu_1)\\
\varphi_{1}(x) \le c_{e, \pi_{2}^{*}}(x, y) - \varphi_{1}^{c_{e, \pi_{2}^{*}}}(y),
 &\forall x,y\in X\times X.
\end{cases}
\]
 We will see that these potentials are connected to the constraints of the pair of dual problems of Theorem \ref{thm:coupled duality}.

\subsection{Coupled Duality}
We furnish a dual problem corresponding to the ``primal problem" (\ref{inf sup problem}). For concreteness we shall use the cost (\ref{eqn:new cost}). The coupling of the dual problem is obtained through the following. 

Fix $\pi_{2}\in \Omega_{2}$. Thanks to Lemma \ref{lemma:saddlepoint-endpoint}, the cost function corresponding to this $\pi_{2}$ is given by
\begin{align}\label{eqn:cost pi_2}
    c_{e, \pi_{2}^{*}}(x, y)=\int_{\mathbb{R}^{d}\times\mathbb{R}^{d}} c_{e}(x,y,x^{\prime}, y^{\prime})d\pi_{2}(x^{\prime}, y^{\prime}),
\end{align}
where $c_{e}$ is given by (\ref{eqn:new cost}). Expanding (\ref{eqn:cost pi_2}), we acquire
\begin{align*}
    c_{e, \pi_{2}^{*}}(x, y)=\tfrac{1}{2}|x-y|^{2}+\mathcal{I}_{\pi_{2}}[x^{\prime},y^{\prime}] (x, y; \alpha).
\end{align*}
Here, for $\alpha>0$, $\mathcal{I}_{\pi_{2}}$ is defined by the integral
\begin{align}\label{integrtal cost 1}
\mathcal{I}_{\pi_{2}}[x^{\prime}, y^{\prime}](x,y; \alpha)=\int_{\mathbb{R}^{d}\times\mathbb{R}^{d}}\tfrac{\alpha}{3}\left(|x-x^{\prime}|^{2}+(x-x^{\prime})(y-y^{\prime})+|y-y^{\prime}|^{2}\right)-\tfrac{1}{2}|x^{\prime}-y^{\prime}|^{2}d\pi_{2}(x^{\prime},y^{\prime}).
\end{align}

From the standard theory of optimal transport, the \textit{dual problem} corresponding to this cost is thus,
\begin{align}\label{eqn:new dual pi_2}
    \sup_{\varphi_{1}, \psi_{1}\in L^{1}(d\mu_{1}\times d\nu_{1})}\int_{\mathbb{R}^{d}}\varphi_{1}(x)d\mu_{1}(x)+\int_{\mathbb{R}^{d}}\psi_{1}(y)d\nu_{1}(y)
\end{align}
subject to the constraint,
\begin{align*}
    \varphi_{1}(x)+\psi_{1}(y)\leq c_{e, \pi_{2}^{*}}(x, y).
\end{align*}
This is reminiscent of the classical dual problem corresponding to the standard primal problem (\ref{Kantorovich problem}). The goal of this section is to establish \textit{strong duality} (which is different from the \textit{minimax strong duality} in Section \ref{minimax section}). Observe that the cost function $c_{e, \pi_{2}^{*}}$ satisfies the conditions for a corresponding cost function in standard optimal transport, that is, it satisfies the regularity conditions (see \cite{villani2008optimal}). Thus we have

\begin{align*}
    \inf_{\pi_{1}}\int_{\mathbb{R}^{d}\times\mathbb{R}^{d}}c_{e, \pi_{2}^{*}}(x, y)d\pi_{1}(x,y)=\sup_{\varphi_{1}, \psi_{1}}\int_{\mathbb{R}^{d}}\varphi_{1}(x)d\mu_{1}(x)+\int_{\mathbb{R}^{d}}\psi_{1}(y)d\nu_{1}(y).
\end{align*}
This follows at once by an application of the standard optimal transport theory, \cite{villani2008optimal}.

In fact, according to the classical theory of optimal transport, we ``solve" the \textit{coupled dual} problem. Indeed, for a coupled pair $(\varphi_{1}, \psi_{1}), (\varphi_{2}, \psi_{2})$, we already saw that the pair $(\varphi_{1}, \psi_{1})$ solves strong duality for the above dual problem (\ref{eqn:new dual pi_2}). While the other pair $(\varphi_{2}, \psi_{2})$, solves another strong duality. Namely,
\begin{align}\label{eqn:new dual pi_1}
    \sup_{\pi_{2}}\int_{\mathbb{R}^{d}\times\mathbb{R}^{d}}c_{e, \pi_{1}^{*}}(x^{\prime}, y^{\prime})d\pi_{2}(x',y')=\inf_{\varphi_{2}, \psi_{2}\in L^{1}(d\mu_{2}\times d\nu_{2})}\left(\int_{\mathbb{R}^{d}}-\varphi_{2}(x')d\mu_{2}(x')+\int_{\mathbb{R}^{d}}-\psi_{2}(y')d\nu_{2}(y')\right).
\end{align}
The latter is minimized subject to the constraint,
\[
c_{e, \pi_{1}^{*}}(x^{\prime}, y^{\prime})\leq -\varphi_{2}(x')-\psi_{2}(y').
\]
We prove this strong duality.

Fixing $\pi_{1}\in \Omega_{1}$, the cost function corresponding to $\pi_{1}$ is
\begin{align}\label{eqn:cost pi_1}
    c_{e, \pi_{1}^{*}}(x^{\prime}, y^{\prime}):=-\tfrac{1}{2}|x^{\prime}-y^{\prime}|^{2}+\mathcal{I}_{\pi_{1}}[x, y](x^{\prime}, y^{\prime}; \alpha),
\end{align}
where
\begin{align}\label{integral cost 2}
\mathcal{I}_{\pi_{1}}[x,y](x^{\prime},y^{\prime};\alpha):=\int_{\mathbb{R}^{d}\times\mathbb{R}^{d}}\tfrac{\alpha}{3}\left(|x-x^{\prime}|^{2}+(x-x^{\prime})(y-y^{\prime})+|y-y^{\prime}|^{2}\right)+\tfrac{1}{2}|x-y|^{2}d\pi_{1}(x,y).
\end{align}
Its dual problem is
\begin{align}
    \inf_{\varphi_{2}, \psi_{2}\in L^{1}(d\mu_{2}\times d\nu_{2})}\int_{\mathbb{R}^{d}}\varphi_{2}(x^{\prime})d\mu_{2}(x^{\prime})+\int_{\mathbb{R}^{d}}\psi_{2}(y^{\prime})d\nu_{2}(y^{\prime})
\end{align}
subject to the constraint
\[
-\varphi_{2}(x^{\prime})-\psi_{2}(y^{\prime})\geq c_{e, \pi_{1}^{*}}(x^{\prime}, y^{\prime})
\]

Corresponding to this dual, according to the standard optimal transport theory \cite{villani2008optimal}, we have strong duality
\begin{align*}
\sup_{\pi_{2}}\int_{\mathbb{R}^{d}\times\mathbb{R}^{d}}c_{e, \pi_{1}^{*}}(x^{\prime}, y^{\prime})d\pi_{2}(x',y')&=-\inf_{\pi_{2}}\int_{\mathbb{R}^{d}\times\mathbb{R}^{d}}\widetilde{c}_{e, \pi_{1}^{*}}(x', y')d\pi_{2}(x',y')\\
(\text{OT})&=-\sup_{\varphi_{2}, \psi_{2}\in L^{1}(d\mu_{2}\times d\nu_{2})}\left(\int_{\mathbb{R}^{d}}\varphi_{2}(x')d\mu_{2}(x')+\int_{\mathbb{R}^{d}}\psi_{2}(y')d\nu_{2}(y')\right)\\
&=\inf_{\varphi_{2}, \psi_{2}\in L^{1}(d\mu_{2}\times d\nu_{2})}\left(\int_{\mathbb{R}^{d}}-\varphi_{2}(x')d\mu_{2}(x')+\int_{\mathbb{R}^{d}}-\psi_{2}(y')d\nu_{2}(y')\right).
\end{align*}
Where, $\widetilde{c}_{e, \pi_{1}^{*}}(x', y')=-c_{e, \pi_{1}^{*}}(x^{\prime}, y^{\prime}).$ 

In summation, we have proved the \textit{Monge-Kantorovich-bilinear coupled duality problem}:
\begin{theorem}\label{thm:coupled duality}
    Let $\mu_{1},\nu_{1}$ and $\mu_{2}, \nu_{2}$ be compactly supported probability measures in $\mathbb{R}^{d}$ that absolutely continuous with respect to Lebesgue, and let $c:\left(\mathbb{R}^{d}\times\mathbb{R}^{d}\right)^{2}\to \mathbb{R}$ be a lower semicontinuous cost function on $c(\cdot, \cdot, x', y')$ and upper semicontinuous cost function on $c(x,y,\cdot,\cdot)$. Then whenever $(\pi_{1}, \pi_{2})\in \Pi(\mu_{1},\nu_{1},\mu_{2}, \nu_{2})$ and $(\varphi_{1}, \psi_{1}), (\varphi_{2}, \psi_{2})$ belong to $L^{1}\left(d\mu_{1}\times d\nu_{1}\right), L^{1}\left(d\mu_{2}\times d\nu_{2}\right),$ respectively, such that
    \begin{align*}
        \varphi_{1}(x)+\psi_{1}(y)\leq c_{e, \pi_{2}^{*}}(x, y)\qquad\text{and}\quad-\varphi_{2}(x')-\psi_{2}(y')\geq c_{e, \pi_{1}^{*}}(x^{\prime}, y^{\prime}),
    \end{align*}
    then the pair of equalities hold
    \begin{align*}
\inf_{\pi_{1}}\int_{\mathbb{R}^{d}\times\mathbb{R}^{d}}c_{e, \pi_{2}^{*}}(x, y)d\pi_{1}(x,y)&=\sup_{\varphi_{1}, \psi_{1}}\int_{\mathbb{R}^{d}}\varphi_{1}(x)d\mu_{1}(x)+\int_{\mathbb{R}^{d}}\psi_{1}(y)d\nu_{1}(y),\\
\sup_{\pi_{2}}\int_{\mathbb{R}^{d}\times\mathbb{R}^{d}}c_{e, \pi_{1}^{*}}(x^{\prime}, y^{\prime})d\pi_{2}(x',y')
&=\inf_{\varphi_{2}, \psi_{2}\in L^{1}(d\mu_{2}\times d\nu_{2})}\left(\int_{\mathbb{R}^{d}}-\varphi_{2}(x')d\mu_{2}(x')+\int_{\mathbb{R}^{d}}-\psi_{2}(y')d\nu_{2}(y')\right).
\end{align*}
Here the cost functionals $c_{e, \pi_{2}^{*}}(x, y)$ and $c_{e, \pi_{1}^{*}}(x^{\prime}, y^{\prime})$ are defined in \eqref{eqn:cost pi_2}-\eqref{integrtal cost 1} and  \eqref{eqn:cost pi_1}-\eqref{integral cost 2}.
\end{theorem}

\begin{corollary}\label{corollary: zero sum dual}
    Let $\pi_{1}^{*}$ and $\pi_{2}^{*}$ be solutions to the primal problem of \eqref{eqn:inf sup euclidean} in inf and sup, respectively. That is
    \[
    \int c_{e, \pi_{2}^{*}}(x, y)d\pi_{1}^{*}=\inf_{\pi_{1}}\int c_{e, \pi_{2}^{*}}(x, y)d\pi_{1};\qquad \int c_{e, \pi_{1}^{*}}(x^{\prime}, y^{\prime})d\pi_{2}^{*}=\sup_{\pi_{2}}\int c_{e, \pi_{1}^{*}}(x^{\prime}, y^{\prime})d\pi_{2}.
    \]
    If there exists a quadruple $(\varphi_1^*,\psi_1^*,\varphi_2^*,\psi_2^*)$ of $c$-concave functions which satisfy the coupled dual problem in Theorem \ref{thm:coupled duality}, then the following identity holds
    \begin{gather*}
        \sup_{\varphi_1,\psi_1} \int \varphi_1(x)\ d\mu_1(x) + \int \psi_1(y)\ d\nu_1(y) = \inf_{\varphi_2,\psi_2} \int -\varphi_2(x')\ d\mu_2(x') + \int -\psi_2(y')\ d\nu_2(y')
    \end{gather*}
    or
    \begin{gather*}
        \int \varphi_1^*(x)\ d\mu_1(x) + \int \psi_1^*(y)\ d\nu_1(y) + \int \varphi_2^*(x')\ d\mu_2(x') + \int \psi_2^*(y')\ d\nu_2(y')=0.
    \end{gather*}
\end{corollary}
\begin{proof} According to Theorem \ref{thm:coupled duality}, we have 
\begin{align}\label{eqns:coupled system}
\begin{split}
\inf_{\pi_{1}}\int c_{e, \pi_{2}^{*}}(x, y)d\pi_{1}(x,y)&=\sup_{\varphi_{1},\psi_{1}}\int \varphi_{1}(x)d\mu_{1}(x)+\int\psi_{1}(y)d\nu_{1}(y),\\
\sup_{\pi_{2}}\int c_{e, \pi_{1}^{*}}(x^{\prime}, y^{\prime})d\pi_{2}(x',y')&=\inf_{\varphi_{2},\psi_{2}}\int -\varphi_{2}(x')d\mu_{2}(x')+\int -\psi_{2}(y)d\nu_{2}(y').
\end{split}
\end{align}
There exists a NETP $(\pi_{1}^{*}, \pi_{2}^{*})$  such that the left-hand sides of \eqref{eqns:coupled system} are achieved jointly, which follows from Theorem \ref{thm:minimax=maximin}. Observe that we have
\begin{gather*}
    \int c_{e, \pi_{2}^{*}}(x, y)d\pi_1^*(x,y) = \int c_{e, \pi_{1}^{*}}(x', y')d\pi_2^*(x',y') = \iint c(x,y,x',y')d\pi_1^*(x,y)d\pi_2^*(x',y').
\end{gather*}
Therefore, there exists $(\varphi_1^*,\psi_1^*,\varphi_2^*,\psi_2^*)$ achieving their sup and inf, respectively, such that according to \eqref{eqns:coupled system} we get
\[
\int \varphi_1^* (x)d\mu_{1}(x)+\int \psi_1^*(y)d\nu_{1}(y)+\int\varphi_2^*(x')d\mu_{2}(x')+\int\psi_2^*)(y')d\nu_{2}(y')=0.
\]

\end{proof}
\begin{remark}
    Corollary \ref{corollary: zero sum dual} is a statement on the zero-sum game structure of the problem. To understand what this means, we refer back to the classic interpretation of the dual problem. That is, at equilibrium the agent who plays $\pi_1$ has associated Kantorovich potentials $\varphi_1,\psi_1$ which represent a profit. However at the same time, the agent who plays $\pi_2$ has corresponding potentials $\varphi_2,\psi_2$ which in this context represents a cost. Thus, one player seeks to maximize their profit while the other is minimizing their costs. The identity encodes the fact that at equilibrium, the potentials must sum up to zero on their support -- otherwise the system is not in equilibrium and one agent has incentive to further profit (or lesser cost).
\end{remark}

As a consequence of Theorem \ref{thm:coupled duality}, we obtain the following result heuristically. However in the proof of Theorem \ref{thm:NashBrenier}, in Section \ref{sec: bilinear transport maps}, we give a rigorous treatment of this consequence. In particular, we show where the mappings $T_{1}$ and $T_{2}$ come from.
\begin{corollary}\label{cor: maps T}
    For cost function $c$ given as in Condition \ref{condition 2} with $\alpha<3$, and assuming the hypotheses from Theorem \ref{thm:coupled duality}, with $m_{x'}:=\int_{\mathbb{R}^{d}}x'd\mu_{2}(x')$ and $m_{y'}:=\int_{\mathbb{R}^{d}}y'd\nu_{2}(y')$. Then the maps $y=T_{1}(x)$ and $y'=T_{2}(x')$ are given as  gradients of some $c_{e, \pi_{2}^{*}},c_{e, \pi_{1}^{*}}$-concave functions
    \begin{align*}
    T_{1}(x)&=\nabla_{x}\left(\frac{3}{\alpha-3}\varphi_{1}(x)-\frac{3+2\alpha}{2(\alpha-3)}|x|^{2}\right)+\frac{2\alpha}{\alpha-3}m_{x'}+\frac{\alpha}{\alpha-3}m_{y'},\\
    T_{2}(x')&=\nabla_{x'}\left(\frac{3}{\alpha+3}\varphi_{2}(x')-\frac{3-2\alpha}{2(\alpha+3)}|x'|^{2}\right)-\frac{2\alpha}{\alpha+3}m_{x}+\frac{2}{\alpha+3}m_{y}.
    \end{align*}
\end{corollary}
\begin{proof}
    We compute one map, $T_{1}$; the second one is done similarly. Fix $\pi_{2}\in\mathcal{P}(\Omega_{2})$, and suppose $\pi_{1}$ is the minimum of  $\inf_{\widetilde{\pi_{1}}}\int c_{e, \pi_{2}^{*}}(x, y)d\widetilde{\pi}_{1}(x,y)$. Since the support of $\pi_{1}$ is contained in the $c_{e, \pi_{2}^{*}}$-superdifferential of $\varphi_{1}^{c_{e, \pi_{2}^{*}}}$, we have $y\in \partial^{c_{e, \pi_{2}^{*}}}\varphi_{1}(x)$. This implies that the function $\overline{x}\mapsto c_{e, \pi_{2}^{*}}(\overline{x}, y)=\varphi_{1}(\overline{x})$ is superdifferentiable at $x$. Moreover, it is clear that the cost $c_{e, \pi_{2}^{*}}(x, y)$ given through \eqref{eqn:new cost} is differentiable everywhere, hence superdifferentiable everywhere. Thus, $\varphi_{1}^{c_{e, \pi_{2}^{*}}}$ is supperdifferentiable at $x$. Consequently, the combination of these items imply that $c_{e, \pi_{2}^{*}}$ is both upper and lower differentiable
at $\overline{x}=x$. Therefore it is differentiable at $x$. Since $x$ was an arbitrary point where $\varphi_{1}^{c_{e, \pi_{2}^{*}}}$ is differentiable,
this proves that, on the support of $\pi_{1}$,  
    \begin{align*}
    \nabla_{x}\varphi_{1}(x)&=\nabla_{x}c_{e, \pi_{2}^{*}}(x, y)\\
    &=\nabla_{x}\left(\tfrac{1}{2}|x-y|^{2}+\mathcal{I}_{\pi_{2}}[x',y'](x,y;\alpha)\right)
    \end{align*}
   Since the functional $\mathcal{I}_{\pi_{2}}$ is given by \eqref{integrtal cost 1}, the gradient of the right hand side of the above is thus
    \begin{align*}
        \frac{3+2\alpha}{3}x+\frac{\alpha-3}{3}y-\frac{2\alpha}{3}m_{x'}-\frac{\alpha}{3}m_{y'}.
    \end{align*}
    Solving for $y$ gives the formula 
    \[
    y=\left[\nabla_{x}\varphi_{1}(x)-\frac{3+2\alpha}{3}x+\frac{2\alpha}{3}m_{x'}+\frac{\alpha}{3}m_{y'}\right]\frac{3}{\alpha-3}:=T_{1}(x).
    \]
    In particular, $y$ is a function of $x$, and we get the formula in the corollary.
\end{proof}

\subsection{Minimaximal Bilinear plans} \label{Minimaximal Bilinear plans are maps}
This section provides the foundation to establish the uniqueness result of the coupled NETP $(\pi_{1}^{*}, \pi_{2}^{*})$ for \eqref{inf sup problem}. In particular, we show that a NETP solution of \eqref{inf sup problem} in the path space projects to a solution in Euclidean space.  The lemmas in this section are essential in establishing that the coupled $(\pi_{1}^{*}, \pi_{2}^{*})$ are given by maps, which are uniquely determined by a gradient of a convex (respectively, concave) function. This section provides a foundational account for obtaining the maps in Corollary \ref{cor: maps T} -- see Section \ref{sec: bilinear transport maps}.

\begin{definition}\label{def: min and max paths}
   Recall the pair of effective cost functions given in \eqref{eqn: eff path}. Given a pair of admissible measures $(\pi_{1}, \pi_{2})\in \Pi_{\text{path}}(\mu_{1}, \nu_{1})\otimes \Pi_{\text{path}}(\mu_{2},\nu_{2})$, let $\gamma_{0}, \xi_{0}: [0, 1]\to \mathbb{R}^{d}$ be a pair of continuous paths from $x$ to $y$ and $x'$ to $y'$, respectively. Then $\gamma_{0}$ is a \textit{minimal path} with respect to $c_{\pi_{2}^{*}}$ if $c_{\pi_{2}^{*}}(\gamma_0)\leq c_{\pi_{2}^{*}}(\gamma)$ for all $\gamma\in \Omega_{1}$, and $\xi_{0}$ is a \textit{maximal path} with respect to $c_{\pi_{1}^{*}}(\xi)$ if $c_{\pi_{1}^{*}}(\xi_{0})\geq c_{\pi_{1}^{*}}(\xi)$ for all $\xi\in\Omega_{2}$.
\end{definition}

The proof of the following lemma follows similar arguments as in \cite[Lemma 3.11]{cabrera2022optimaltransportationprincipleinteracting}. The following will prove essential. For a Borel set $B$, let $\pi\!\restriction B$ denote the restriction of $\pi$ to $B$, that is, the measure defined by
\[
(\pi\!\restriction B)[A] = \pi[B \cap A],
\]
for every Borel set $A$. 

\begin{lemma}[Support of minimax plans]\label{lem:support_saddle}
Let $c_{\pi_{2}^{*}}$ be a lower semicontinuous cost function and $c_{\pi_{1}^{*}}$ an upper semicontinuous cost function, and $\alpha<\pi^{2}/4$.
Let $\pi_{1}^{*}, \pi_{2}^{*}$ be a NETP plan for the bilinear functional \eqref{inf sup problem} with cost $c: \Omega_{1}\times \Omega_{2}\to \mathbb{R}$ given as in Condition \ref{condition 2}. If $\gamma_{0}\in\text{spt}(\pi_{1}^{*})$ and $\xi_{0}\in \text{spt}(\pi_{2}^{*})$, then $\gamma_{0}$ is a minimal path with respect to $c_{\pi_{2}^{*}}$ and $\xi_{0}$ is a maximal path with respect to $c_{\pi_{1}^{*}}(\xi)$ as defined in \eqref{eqn: eff path}.
\end{lemma}

\begin{proof} 
The former case was dealt with in \cite[Lemma 3.11]{cabrera2022optimaltransportationprincipleinteracting}. The idea is to argue by contradiction: assume that $\pi_{1}^{*}$ is minimal, but $\gamma_{0}\in \text{spt}(\pi_{1}^{*})$ is not minimal path. In this case one constructs a tubular neighborhood of paths and a measure given by a unique map associating the minimal path in which it connects the end points of $\gamma$. Through this measure one contradicts the optimality of $\pi_{1}^{*}$. For the latter, we apply similar arguments to show that if $\pi_{2}^{*}$ is a maximal plan and if $\xi_{0}\in \text{spt}(\pi_{2}^{*})$, then $\xi_{0}$ is a maximal path.

Fix $\pi_{1}\in \mathcal{P}(\Omega_{1})$, and suppose $\pi_{2}^{*}\in \Pi_{\text{path}}(\mu_{2}, \nu_{2})$ is maximal with respect to $c_{\pi_{1}^{*}}(\xi)$ (defined in \eqref{eqn: eff path}) and $\xi_{0}\in \text{spt}(\pi_{2}^{*})$ is not a maximal path. In particular,
\[
\int_{\Omega_{2}} c_{\pi_{1}^{*}}(\xi)d\pi_{2}(\xi)\leq\int_{\Omega_{2}}c_{\pi_{1}^{*}}(\xi)d\pi_{2}^{*}(\xi)\quad\forall\;\xi\in\Omega_{2},\quad c_{\pi_{1}^{*}}(\xi)>c_{\pi_{1}^{*}}(\xi_{0})\quad \pi_{2}\text{-a.e.}
\]
 Let $W:=B_{1}(x_{0}')\times B_{1}(y_{0}')$, $(e_{0}, e_{1}):\Omega_{2}\to X\times X$ defined by $\xi\mapsto (\xi(0), \xi(1))$, and define a \textit{tubo} by the set
 \[
 \mathcal{T}:=(e_{0}, e_{1})^{-1}(W).
 \]
 This is an open subset of $\Omega_{2}$ containing the path $\xi_{0}$. Consider the measure $\overline{\pi}_{2}:=(\pi_{2}^{*}\restriction\mathcal{T})/\pi_{2}^{*}[\mathcal{T}]$. Note that $\pi_{2}^{*}[\mathcal{T}]>0$ since $\xi_{0}$ is in the support of $\pi_{2}^{*}$, and that $\mathcal{T}$ is open and $\xi_{0}\in\mathcal{T}$.

 Now construct a new measure $\widetilde{\pi}_{2}\in \Pi_{\text{path}}(\nu_{x'},\nu_{y'})$ as follows. Let $h_{t}:X\times X\to \Omega_{2}$ be a map that associates to any two points $x', y'$ the maximal path connecting them; namely, 
 \[
 h_{t}(x',y')=\xi_{x',y'}(t).
 \]
 According to Proposition \ref{boundary value problem}, $h_{t}$ is well defined. Set
 \[
 \widetilde{\pi}_{2}:=(h_{t})_{\sharp}\left((e_{0}, e_{1})_{\sharp}\overline{\pi}_{2}\right).
 \]
In other words, $\widetilde{\pi}_{2}$ is the probability measure corresponding to all the maximal paths connecting the end points of the path in the support of $\overline{\pi}_{2}$. Then for any $\varepsilon_{0}\in \left(0, \pi_{2}^{*}[\mathcal{T}]\right)$ define 
\[
\pi_{2}^{\prime}:=\pi_{2}^{*}-\varepsilon_{0}\overline{\pi}_{2}+\varepsilon_{0}\widetilde{\pi}_{2}.
\]
That $\pi_{2}^{\prime}$ is positive follows from the fact that $\pi_{2}^{*}-\varepsilon_{0}\overline{\pi}_{2}$ is positive, which follows from the fact that $\varepsilon_{0}<\pi_{2}^{*}[\mathcal{T}]$. The marginals of $\pi_{2}^{\prime}$ share the same marginals of $\pi_{2}^{*}$. Indeed, we have
\[
(e_{0})_{\sharp}\pi_{2}^{\prime}=\mu_{2}-\varepsilon_{0}(e_{0})_{\sharp}\overline{\pi}_{2}+\varepsilon_{0}(e_{0})_{\sharp}\widetilde{\pi}_{2}=\mu_{2},
\]
as $(e_{0})_{\sharp}\overline{\pi}_{2}=(e_{0})_{\sharp}\widetilde{\pi}_{2}$ since each measure is the projection under the evaluation map $e_{0}$ of $(e_{0}, e_{1})_{\sharp}(\overline{\pi}_{2})$. Similarly, the second marginal gives $\nu_{2}$. Therefore $\pi_{2}^{\prime}$ is an admissible plan. 

Finally, we will show $\int c_{\pi_{1}^{*}}(\xi)d\pi_{2}^{*}(\xi)-\int c_{\pi_{1}^{*}}(\xi)d\pi_{2}^{\prime}(\xi)<0$, thereby contradicting the maximality of $\pi_{2}^{*}$. By definition
\[
c_{\pi_{1}^{*}}(\xi)\leq c_{e, \pi_{1}^{*}}(\xi(0), \xi(1))\quad\forall\xi\in \Omega_{2}
\]
with equality  if and only if $\xi\in \Omega_{\text{2,max}}$; the set of all maximal paths satisfying Definition \ref{def: min and max paths} which is given by 
\begin{align}\label{max set}
\Omega_{2, \text{max}}:=\left\{\xi\in \Omega_{2}| \;c_{\pi_{1}^{*}}(\xi)=c_{e, \pi_{1}^{*}}(\xi(0), \xi(1))\right\}.
\end{align}
That the function $\xi\mapsto c_{e, \pi_{1}^{*}}(\xi(0), \xi(1))-c_{\pi_{1}^{*}}(\xi)$ is upper semicontinuous in $\xi$ follows from Condition \ref{condition 2} through Proposition \ref{boundary value problem}. By our assumption $\xi_{0}\in \text{spt}(\pi_{2}^{*})$ and $\xi_{0}\notin\Omega_{2,\text{max}}$, and so $c_{\pi_{1}^{*}}(\xi)<c_{e, \pi_{1}^{*}}(\xi(0), \xi(1))$ in an open set intersecting the support of $\pi_{2}^{*}$, and hence of $\overline{\pi}_{2}$, implying the strict inequality
\[
\int_{\Omega_{2}}c_{\pi_{1}^{*}}(\xi)d\overline{\pi}_{2}(\xi)<\int_{\Omega_{2}}c_{e, \pi_{1}^{*}}(\xi(0), \xi(1))d\overline{\pi}_{2}(\xi).
\]
Note that the integrand on the right is a function of only $(\xi(0), \xi(1))$. This last integral equals the integral with respect to the joint marginal $(e_{0}, e_{1})_{\sharp}\overline{\pi}_{2}$, so from the definition of $\widetilde{\pi}_{2}$ through the push-forward of $(e_{0}, e_{1})$, we have
\[
\int_{\Omega_{2}}c_{\pi_{1}^{*}}(\xi)d\overline{\pi}_{2}(\xi)<\int_{X}c_{e, \pi_{1}^{*}}(x^{\prime}, y^{\prime})d\left((e_{0}, e_{1})_{\sharp}\overline{\pi}_{2}\right)(x',y')=\int_{\Omega_{2}}c_{\pi_{1}^{*}}(\xi)d\widetilde{\pi}_{2}(\xi).
\]
Putting this together with the definition of $\pi_{2}^{\prime}$ leads to the desired strict inequality
\begin{align*}
\int_{\Omega_{2}}c_{\pi_{1}^{*}}(\xi)d\pi_{2}^{*}(\xi)-\int_{\Omega_{2}}c_{\pi_{1}^{*}}(\xi)d\pi_{2}^{\prime}(\xi)&=\varepsilon_{0}\int_{\Omega_{2}}c_{\pi_{1}^{*}}(\xi)d\overline{\pi}_{2}(\xi)-\varepsilon_{0}\int_{\Omega_{2}}c_{\pi_{1}^{*}}(\xi)d\widetilde{\pi}_{2}(\xi)\\
&<0.
\end{align*}
Therefore $\pi_{2}^{\prime}$ is an admissible plan with smaller total cost, contradicting the maximality of $\pi_{2}^{*}$.
\end{proof}

The next lemma is vital in ensuring that minimal and maximal plans in $\Omega$ have $c_{e, \pi_{i}^{*}}$-cyclical monotone support, for $i\neq j$. Which will be crucial in establishing that the support of $\pi_{i}^{*},$ for $ i=1,2$, is contained in the $c_{e, \pi_{2}^{*}}$-superdifferential and $c_{e, \pi_{1}^{*}}$-subdifferential, of $\varphi_{1}$ and $\varphi_{2}$, respectively. This guarantees that the (minimaximal/maximinimal) plans are thus given by maps $\Gamma^{1}, \Gamma^{2}$ which are uniquely determined by gradients of convex/concave functions, Corollary \ref{cor: maps T}.

\begin{lemma}\label{lemma: monotonicity}Let $\mu_{i}, \nu_{i}$, $i=1,2$, be compactly supported and absolutely continuous with respect to Lebesgue measure. Suppose $c: \Omega:=\Omega_{1}\times\Omega_{2} \to \mathbb{R}$ is a continuous cost function  satisfying Condition \ref{condition 2}, and $(\pi^{*}_{1}, \pi_{2}^{*})\in \Pi_{\text{path}}(\mu_{1},\nu_{1})\otimes \Pi_{\text{path}}(\mu_{2}, \nu_{2})$ a NETP (Definition \ref{def: NETP}) with respect to $c$. Let $\check{\pi}_{1}:=(e_{0},e_{1})_{\sharp}\pi_{1}^{*} \in \Pi(\mu_{1},\nu_{1})$ and $\check{\pi}_{2}:=(e_{0},e_{1})_{\sharp}\pi_{2}^{*} \in \Pi(\mu_{2},\nu_{2})$. Stipulate further that $c_{e, \pi_{2}^{*}}$ defined by \eqref{euc integral}; through  \eqref{eqn:new cost}, satisfies the hypothesis of Theorem \ref{thm:coupled duality} (resp. for $c_{e, \pi_{1}^{*}}$). Then the support of $\check{\pi}_{1}$, spt$\;(\check{\pi}_{1})$, is $c_{e, \pi_{2}^{*}}$-cyclically monotone, while the support of $\check{\pi}_{2}$ is $c_{e, \pi_{1}^{*}}$-cyclically anti-monotone (see Def. \ref{def. cyclical monotone}). Moreover, $\check{\pi}_{1}$ is minimal while $\check{\pi}_{2}$ is maximal with respect to $c_{e, \pi_{2}^{*}}$ and $c_{e, \pi_{1}^{*}}$, respectively defined by \eqref{eqn: effective costs}.
\end{lemma}

\begin{proof}
We reprise the proof in \cite[Lemma 3.12]{cabrera2022optimaltransportationprincipleinteracting} for $c_{e, \pi_{2}^{*}}$-cyclical monotone, which is based on a classical proof from standard optimal transport which may be found in \cite[Theorem 1.38]{santambrogio2015optimal}. Indeed, let $\pi_{1}^{*}$ be the minimal plan of $\inf_{\Pi_{\text{path}}(\mu_{1}, \nu_{1})}\int_{\Omega_{1}}c_{\pi_{2}^{*}}d\pi_{1}(\gamma)$, and apply the proof in \cite[Lemma 3.12]{cabrera2022optimaltransportationprincipleinteracting} with the following data: $c_{e, \pi_{2}^{*}}$ in place of $c_{e}(x,y)$ and $\check{\pi}_{1}$ in place of $\check{\pi}$, and $c_{\pi_{2}^{*}}$ in place of $c(\gamma)$ while keeping $\pi_{2}^{*}$ and $\xi$ fixed.
    
The subtlety comes from  $c_{e, \pi_{1}^{*}}$-cyclical \emph{anti-monotonicity}. This is because  $c_{e, \pi_{1}^{*}}$ corresponds to a maximal plan $\pi_{2}^{*}$. To this end, let $\pi_{1}^{*}$ and $\gamma$ be fixed and suppose $\pi_{2}^{*}$ solves \[\sup_{\pi_{2}\in\Pi_{\text{path}}(\mu_{2}, \nu_{2})}\int_{\Omega_{2}}c_{\pi_{1}^{*}}(\xi)d\pi_{2}(\xi)=-\inf_{\pi_{2}\in\Pi_{\text{path}}(\mu_{2}, \nu_{2})}\int_{\Omega_{2}}-c_{\pi_{1}^{*}}(\xi)d\pi_{2}(\xi),
\]
and that $\text{spt}\;(\pi_{2}^{*})\subset \Omega_{2,\text{max}}$. This latter subset is defined in \eqref{max set} in the proof of Lemma \ref{lem:support_saddle}. According to Lemma \ref{lem:support_saddle}, $\xi$ in the support of $\pi_{2}^{*}$ is maximal. Let $\check{\pi}_{2}=(e_{0}, e_{1})_{\sharp}\pi_{2}^{*}$ be a transport plan from $\mu_{2}$ to $\nu_{2}$ obtained by pushing forward $\pi_{2}^{*}$ through the coupled evaluation mapping $(e_{0}, e_{1}): \Omega_{2}\to X\times X$.

Let $\widetilde{c}_{\pi_{1}^{*}}=-c_{\pi_{1}^{*}}$. Then by the classical theory of optimal transport, we have that $\pi_{2}^{*}$ is optimal for $\sup\int c_{\pi_{1}^{*}}d\pi_{2}$ if and only if it is optimal for $\inf \int -c_{\pi_{1}^{*}}d\pi_{2}$, if and only if its support is $\widetilde{c}_{\pi_{1}^{*}}$-cyclically monotone, that is $\sum_{i=1}^{n}\widetilde{c}_{\pi_{1}^{*}}(\xi_{i})\leq \sum_{i=1}^{n}\widetilde{c}_{\pi_{1}^{*}}(\widetilde{\xi}_{i})$ if and only if $\sum_{i=1}^{n}c_{\pi_{1}^{*}}(\xi_{i})\geq \sum_{i=1}^{n}c_{\pi_{1}^{*}}(\widetilde{\xi}_{i})$. Here $\xi_{i}$ is a maximal path from $\xi_{i}(0)=x_{i}^{\prime}$ to $\xi_{i}(1)=y_{i}^{\prime}$ and $\widetilde{\xi}_{i}$ a  maximal path from $\widetilde{\xi}_{i}(0)=x_{i}^{\prime}$ to $\widetilde{\xi}_{i}(1)=y_{\iota(i)}^{\prime}$ for some permutation $\iota$. 

Suppose towards a contradiction that the support of $\pi_{2}^{*}$ is not $c_{e, \pi_{1}^{*}}$-cyclically \emph{anti-monotone}. Then there exist an $n\geq 0$, a cyclical permutation $\iota$, and maximal paths $\xi_{i}$ from $x_{i}^{\prime}$ to $y_{i}^{\prime}$, and $\widetilde{\xi}_{i}$ from $x_{i}^{\prime}$ to $y_{\iota(i)}^{\prime}$ in $\text{spt}\;(\pi_{2}^{*})$, respectively, and $\{(x_{i}^{\prime}, y_{i}^{\prime})\}\subset \text{spt}\;(\check{\pi}_{2})$, such that 
\[
\sum_{i=1}^{n}c_{\pi_{1}^{*}}\left(\xi_{i}\right)<\sum_{i=1}^{n}c_{\pi_{1}^{*}}\left(\widetilde{\xi}_{i}\right).
\]
The latter inequality is due to Definition \ref{def. cyclical monotone}.
Here the \emph{shifted paths} are $\widetilde{\xi}_{i}(t):=\xi_{i}(t)+th_{i}(t),$ $0\leq t\leq 1,\;h_{i}(t)\neq 0\;\forall i,$ and $\widetilde{\xi}_{i}(0)=x_{i}^{\prime}$, $\xi_{i}(1)+h_{i}(1):=\xi_{i+1}(t)$, with the convention $\xi_{n+1}(1)=\xi_{1}(1)$. In other words, $\widetilde{\xi}_{i}(0)=\xi_{i}(0)$ while $\widetilde{\xi}_{i}(1)=\xi_{\iota(i)}(1)$ for all $i=1,2,.\dots, n$. So since $\widetilde{\xi}_{i}$ and $\xi_{i}$ are in the support of $\pi_{2}^{*}$, then they are both maximal, and therefore, $c_{\pi_{1}^{*}}(\xi_{i})=c_{e, \pi_{1}^{*}}(x_{i}^{\prime}, y_{i}^{\prime})$ and $c_{\pi_{1}^{*}}(\widetilde{\xi}_{i})=c_{e, \pi_{1}^{*}}(x_{i}^{\prime}, y_{\iota(i)}^{\prime})$. Thus, the above strict inequality is equivalent to 
\[
\sum_{i=1}^{n}c_{e, \pi_{1}^{*}}(x_{i}^{\prime}, y_{i}^{\prime})<\sum_{i=1}^{n}c_{e, \pi_{1}^{*}}(x_{i}^{\prime}, y_{\iota(i)}^{\prime}).
\]

Given $\varepsilon>0$, take 
\begin{align}\label{inq: continuous}
    \varepsilon<\frac{1}{2n}\left(\sum_{i=1}^{n}c_{e, \pi_{1}^{*}}(x_{i}^{\prime}, y_{\iota(i)}^{\prime})-c_{e, \pi_{1}^{*}}(x_{i}^{\prime}, y_{i}^{\prime})\right).
\end{align}
By continuity of $c_{\pi_{1}^{*}}$, there exists an $r>0$ such that 
\[
c_{\pi_{1}^{*}}(\xi)<c_{\pi_{1}^{*}}(\xi_{i})+\varepsilon \quad\text{for all }\quad \xi\in \mathcal{T}_{i}, \; i=1,\ldots, n,
\]
where for each $i$, the set $\mathcal{T}_{i}$, called a \emph{tubo}, is defined by 
\[
\mathcal{T}_{i}:=(e_{0}, e_{1})^{-1}\left(B_{r}(x_{i}^{\prime})\times B_{r}(y_{i}^{\prime})\right)\cap \Omega_{2,\text{max}}.
\]
Similarly for each $i$ (and the same $r$ to be chosen later sufficiently small) and by Proposition \ref{boundary value problem}, $\xi_{x', y'}$ is continuous in $(x',y')$. So we have that
\[
c_{\pi_{1}^{*}}(\xi)>c_{\pi_{1}^{*}}(\widetilde{\xi}_{i})-\varepsilon\quad\text{for all }\quad \xi\in\widetilde{\mathcal{T}_{i}},
\]
where $\widetilde{\mathcal{T}}_{i}:=(e_{0}, e_{1})^{-1}\left(B_{r}\left(x_{i}^{\prime}\right)\times B_{r}\left(y_{\iota(i)}^{\prime}\right)\right)\cap \Omega_{2, \text{max}}$. By definition, the sets $\mathcal{T}_{i}$ and $\widetilde{\mathcal{T}}_{i}$ are non-empty and relatively open subsets of $\Omega_{2, \text{max}}$.

Next we define new measures, 
\[\pi_{2}^{i}:=(\pi_{2}^{*}\restriction\mathcal{T}_{i})/\pi_{2}^{*}[\mathcal{T}_{i}],\]
\[
\nu_{x'}^{i}=(e_{0})_{\sharp}\pi_{2}^{i},\quad\text{and}\quad \nu_{y'}^{i}=(e_{1})_{\sharp}\pi_{2}^{i}.
\]
Observe that  since $\mathcal{T}_{i}$ is relatively open with respect to $\Omega_{2,\text{max}}$, and $\xi_{i}\in \mathcal{T}_{i}$ is contained in the support of $\pi_{2}^{*}$, $\pi_{2}^{*}\left[\mathcal{T}_{i}\right]$ positive. This is equivalent to $(x_{i}^{\prime}, y_{i}^{\prime})\in \text{spt}\;\left(\check{\pi}_{2}\right)$. Take $0<\varepsilon_{0}<\frac{\text{min}_{i}\pi_{2}^{*}[\mathcal{T}_{i}]}{n}$.

Construct a measure $\widetilde{\pi}_{2}^{i}\in \Pi\left(\nu_{x'}^{i}, \nu_{y'}^{\iota(i)}\right)$, for every $i$, as follows. As in the previous proof, let $h: X\times X\to \Omega_{2}$ be a map defined by $(x',y')\mapsto \xi_{x',y'}$. That is, $h_{t}(x' ,y')=\xi_{x', y'}(t)$ is the \emph{maximal path} connecting $\xi_{x',y'}(0)=x'$ to $\xi_{x', y'}(1)=y'$. Then the estimates $c_{\pi_{1}^{*}}(\xi)<c_{\pi_{1}^{*}}(\xi_{i})+\varepsilon$ for all $\xi\in \mathcal{T}_{i}$ and $c_{\pi_{1}^{*}}(\xi)>c_{\pi_{1}^{*}}(\widetilde{\xi}_{i})-\varepsilon$ for all $\xi\in \widetilde{\mathcal{T}}_{i}$  coincide with the following estimates $c_{e, \pi_{1}^{*}}(x^{\prime}, y^{\prime})<c_{e, \pi_{1}^{*}}(x_{i}^{\prime}, y_{i}^{\prime})+\varepsilon$ and $c_{e, \pi_{1}^{*}}(x^{\prime}, y^{\prime})>c_{e, \pi_{1}^{*}}(x_{i}^{\prime}, y_{\iota(i)}^{\prime})-\varepsilon$, respectively, for all pairs $(x', y')\in B_{r}(x_{i}^{\prime})\times B_{r}(y_{i}^{\prime})$ and all $(x', y')\in B_{r}(x_{i}^{\prime})\times B_{r}(y_{\iota(i)}^{\prime})$. Now take $\widetilde{\pi}_{2}^{i}:=(h_{t})_{\sharp}\left(\nu_{x'}^{i}\otimes\nu_{y'}^{\iota(i)}\right)$.

Now define 
\[
\widetilde{\pi}_{2}:=\pi_{2}^{*}-\varepsilon_{0}\sum_{i=1}^{n}\pi_{2}^{i}+\varepsilon_{0}\sum_{i=1}^{n}\widetilde{\pi}_{2}^{i}.
\]
That $\widetilde{\pi}_{2}$ is positive follows from the fact that $\pi_{2}^{*}-\varepsilon_{0}\sum_{i=1}^{n}\pi_{2}^{i}$ is positive as $\varepsilon_{0}<\text{min}_{i}\pi_{2}^{*}[\mathcal{T}_{i}]/n$. To be explicit, checking $\widetilde{\pi}_{2}$ is positive, it suffices to check $\pi_{2}^{*}-\varepsilon_{0}\sum_{i=1}^{n}\pi_{2}^{i}>0$. Indeed, the condition $\varepsilon_{0}\pi_{2}^{i}<\pi_{2}^{*}/n$ is sufficient to check the latter strict inequality as $\varepsilon_{0}\pi_{2}^{i}=\frac{\varepsilon_{0}}{\pi_{2}^{*}[\mathcal{T}_{i}]}(\pi_{2}^{*}\restriction\mathcal{T}_{i})$ and $\varepsilon_{0}/\pi_{2}^{*}[\mathcal{T}_{i}]<1/n$.

That the marginals of $\widetilde{\pi}_{2}$ share the same marginals of $\pi_{2}^{*}$ follows by identical arguments applied in the proof of Lemma \ref{lem:support_saddle}. To that end, we have 

\begin{align*}
    (e_{0})_{\sharp}\widetilde{\pi}_{2}&=\mu_{2}-\varepsilon_{0}\sum_{i=1}^{n}\nu_{x'}^{i}+\varepsilon_{0}\sum_{i=1}^{n}(e_{0})_{\sharp}(h_{t})_{\sharp}\left(\nu_{x'}^{i}\otimes\nu_{y'}^{\iota(i)}\right)\\
    &=\mu_{2}-\varepsilon_{0}\sum_{i=1}^{n}\nu_{x'}^{i}+\varepsilon_{0}\sum_{i=1}^{n}\nu_{x'}^{i}\\
    &=\mu_{2};
\end{align*}
where we used $(\nu_{x'}^{i}\otimes\nu_{y'}^{\iota(i)})\left[h_{t}^{-1}(e_{0}^{-1}(B))\right]=\nu_{x'}^{i}[B]$ for all Borel sets $B$, which follows from the fact that  $(e_{0}\circ h_{t})(x',y')=x'$ for all pairs of points $(x', y')$. On the other hand, 
\begin{align*}
    (e_{1})_{\sharp}\widetilde{\pi}_{2}&=\nu_{2}-\varepsilon_{0}\sum_{i=1}^{n}\nu_{y'}^{i}+\varepsilon_{0}\sum_{i=1}^{n}(e_{1})_{\sharp}(h_{t})_{\sharp}\left(\nu_{x'}^{i}\otimes\nu_{y'}^{\iota(i)}\right)\\
    &=\nu_{2}-\varepsilon_{0}\sum_{i=1}^{n}\nu_{y'}^{i}+\varepsilon_{0}\sum_{i=1}^{n}\nu_{y'}^{\iota(i)}\\
    &=\nu_{2},
\end{align*}
where the last equality follows from the permutation $\iota$ and that  $(\nu_{x'}^{i}\otimes\nu_{y'}^{\iota(i)})\left[h^{-1}_{t}(e_{1}^{-1}(A))\right]=\nu_{y'}^{\iota(i)}[A]$ for all Borel sets $A$, which follows from $(e_{1}\circ h_{t})(x', y')=y'$ for all pairs of points $x', y'$.

Finally, we prove the strict inequality $\int c_{\pi_{1}^{*}}d\pi_{2}^{*}-\int c_{\pi_{1}^{*}}d\widetilde{\pi}_{2}<0,$ thereby contradicting the maximality of $\pi_{2}^{*}$. To see this substitute $\widetilde{\pi}_{2}$,
\begin{align*}
    \int c_{\pi_{1}^{*}}(\xi)d\pi_{2}^{*}(\xi)-\int c_{\pi_{1}^{*}}(\xi)d\widetilde{\pi}_{2}(\xi)&=\varepsilon_{0}\sum_{i=1}^{n}\int c_{\pi_{1}^{*}}(\xi)d\pi_{2}^{i}(\xi)-\varepsilon_{0}\sum_{i=1}^{n}\int c_{\pi_{1}^{*}}(\xi)d\widetilde{\pi}_{2}^{i}(\xi)\\
    &\leq \varepsilon_{0}\sum_{i=1}^{n}\left(c_{\pi_{1}^{*}}(\xi_{i})+\varepsilon\right)-\varepsilon_{0}\sum_{i=1}^{n}\left(c_{\pi_{1}^{*}}(\widetilde{\xi}_{i})-\varepsilon\right)\\
    &=\varepsilon_{0}\left(\sum_{i=1}^{n}c_{e, \pi_{1}^{*}}(x_{i}^{\prime}, y_{i}^{\prime})-c_{e, \pi_{1}^{*}}(x_{i}^{\prime}, y_{\iota(i)}^{\prime})+2n\varepsilon\right)\\
    &<0.
\end{align*}
The last inequality is due to the strict inequality \eqref{inq: continuous}. But this contradicts the maximality of $\pi_{2}^{*}$.

To conclude the proof, we first apply the classical theory of optimal transport to $\check{\pi}_{1}$. Namely, since the support of $\check{\pi}_{1}$ is $c_{e, \pi_{2}^{*}}$-cyclically monotone on the end points of the minimal path $\gamma$, $(x_{i}, y_{i})$, contained in $\text{spt}\;(\check{\pi}_{1})$, and since we can bound the function as follows, $c_{e, \pi_{2}^{*}}(x, y)\leq f(x)+g(y)$ for some $f\in L^{1}(d\mu_{1}), g\in L^{1}(d\nu_{1})$, then according to the \emph{fundamental theorem of optimal transport}, $\check{\pi}_{1}$ is optimal with respect to $c_{e, \pi_{2}^{*}}(x, y)$ by \cite[Theorem 2.13]{ambrosio2013user}.

For the maximal case of $\check{\pi}_{2}$, let $\widetilde{c}_{e, \pi_{1}^{*}}:=-c_{e, \pi_{1}^{*}}$. We apply the fundamental theorem of optimal transport to $\widetilde{c}_{e, \pi_{1}^{*}}$ \cite[Theorem 2.13]{ambrosio2013user}. Indeed, since we showed that the support of $\check{\pi}_{2}$ is $c_{e, \pi_{1}^{*}}$-cyclically antimonotone on the end points $(x_{i}^{\prime}, y_{i}^{\prime})$ contained in the support of $\check{\pi}_{2}$, and since we can bound $\widetilde{c}_{e, \pi_{1}^{*}}$ as follows, $\widetilde{c}_{e, \pi_{1}^{*}}(x', y')<a(x')+b(y')$ for some $a\in L^{1}(d\mu_{2}), b\in L^{1}(d\nu_{2})$; then, $c_{e, \pi_{1}^{*}}(x^{\prime}, y^{\prime})>-a(x')-b(y')$. Then $\check{\pi}_{2}$ is optimal for $\sup\int\widetilde{c}_{e, \pi_{1}^{*}}(x', y')d\check{\widetilde{\pi}}_{2}(x', y')$, if and only if it is optimal for $\inf\int -c_{e, \pi_{1}^{*}}(x^{\prime}, y^{\prime})d\check{\widetilde{\pi}}_{2}(x', y')$. Thus, $\check{\pi}_{2}$ is \emph{maximal} with respect to $c_{e, \pi_{1}^{*}}(x^{\prime}, y^{\prime})$.
\end{proof}

Having knowledge of the previous lemmas allows us to prove Theorem \ref{thm:pathtostationary}.

\begin{proof}[\textit{Proof of Theorem \ref{thm:pathtostationary}}] Lemma \ref{lem:support_saddle} establishes the first item \emph{(1)} of the theorem, while
Lemma \ref{lemma: monotonicity} establishes that any pair of NETP plans $(\pi_{1}^{*},\pi_{2}^{*} )\in \Pi_{\mathrm{path}}(\mu_1,\nu_1)\otimes\Pi_{\mathrm{path}}(\mu_2,\nu_2) $ project onto a  coupled solution, $\check{\pi}_{1}, \check{\pi}_{2}$, of the Nash-Monge--Kantorovich problem, \eqref{eqn:inf sup euclidean}-\eqref{min max monge}, with cost $c_{e, \pi_{i}^{*}}$ via the pairs of couplings
\[
\check{\pi}_{i} := (e_0,e_1)_{\#}\pi_{i}^{*}, \; i=1,2.
\]
\end{proof}
\begin{remark}
Recall also the notion of dynamical couplings introduced in Section \ref{section: path formulation}. Moreover, as shown in Villani~\cite[Theorem 7.21]{villani2008optimal}, a dynamical optimal coupling can be interpreted as a minimizing path in the space of probability measures; respectively, as a maximizing path of measures. In both constructions, one obtains the same dynamical optimal coupling; in particular, the present framework provides an alternative proof of the result in~\cite{villani2008optimal}. Our novelty is to include the maximinimal/minimaximal frame work to the bilinear transport problem.
\end{remark}

\subsection{Bilinear transport maps}\label{sec: bilinear transport maps} That the bilinear plans from the previous section are given by maps follows from the following construction in this section. Moreover, the coupled plans are given by maps, $\Gamma^{1}(x)$ and $\Gamma^{2}(x')$, provided $\mu_{1}, \mu_{2}$ are absolutely continuous with respect to Lebesgue measure. This uses and extends results from the more classical setting of Brenier \cite{brenier1991polar}, Gangbo-McCann \cite{gangbo1998geometry}, and the more recent work of the first author \cite{cabrera2022optimaltransportationprincipleinteracting}.  In this case we apply the classical theory of optimal transport to acquire uniquely determined maps $T_{1}$ and $T_{2}$ mapping $\mu_{1}$ to $\nu_{1}$ and $\mu_{2}$ to $\nu_{2}$, respectively; which solve the Nash-Monge-Kantorovich bilinear transport problems \eqref{inf sup problem}, \eqref{eqn:inf sup euclidean}-\eqref{min max monge}. In other words, the NETP plans $(\pi_{1}^{*}, \pi_{2}^{*})$ are given by the  maps $\Gamma^{i}$ for $i=1,2$, from Definition \ref{def: NMK}, and these maps solve the Nash-Monge-Kantorovich bilinear problem ( see Proposition \ref{prop:NE} ). Furthermore, these maps are given by minimal and maximal paths, $\Gamma^{1}(x,t)=\gamma_{x,T_{1}(x)}(t)$ and $\Gamma^{2}(x', t)=\xi_{x', T_{2}(x')}(t)$, respectively. 

We look for mappings
  \begin{equation}\label{Gamma maps}
      \begin{split}
          \Gamma^{i}: \mathbb{R}^{d} \times [0, 1] \to \mathbb{R}^{d},\; i=1,2,
      \end{split}
  \end{equation}
of the form $\Gamma^{i}(x,t)$ for every $x$ with the following properties
\begin{equation} \label{eqn: gammamap}
\begin{split}
\left\{ \begin{array}{rcl}
\Gamma^{1}(x,0)& =x \mbox{ for all }
& x \in \mathbb{R}^{d}\\ \Gamma^{1}(x,1) & =T_{1}(x)  \mbox{ for } &  (T_{1})_{\sharp}\mu_{1}=\nu_{1} ,
\end{array}\right. \quad \left\{ \begin{array}{rcl}
\Gamma^{2}(x',0)& =x \mbox{ for all }
& x' \in \mathbb{R}^{d}\\ \Gamma^{2}(x',1) & =T_{2}(x')  \mbox{ for } &  (T_{2})_{\sharp}\mu_{2}=\nu_{2} ,
\end{array}\right.
\end{split} 
\end{equation}
where, for each $i=1,2$, $T_{i}: \mathbb{R}^{d} \to \mathbb{R}^{d}$ is a measurable map pushing $\mu_{i} \mapsto \nu_{i}$. Note that the mappings in \eqref{Gamma maps} are in one-to-one correspondence with the mappings $\Gamma^{i}:\mathbb{R}^{d}\to\Omega_{i}$.

Instead of considering the Euclidean space $\mathbb{R}^{d}$, we will restrict to $X$ a simply connected, bounded domain. For all intents and purposes, $X$ can be thought of as a sufficiently large and closed ball.
We attain these maps through the compositions: 
\begin{center}
\begin{tikzcd}
 X\ar[r,"(\text{Id}\times T_{1})"]\ar[rr,out=-30,in=210,swap,"\Gamma^{1}"] & X\times X\ar[r,"\gamma_{x,y}"] & \Omega_{1,\text{min}}
\end{tikzcd},\quad\begin{tikzcd}
 X\ar[r,"(\text{Id}\times T_{2})"]\ar[rr,out=-30,in=210,swap,"\Gamma^{2}"] & X\times X\ar[r,"\xi_{x',y'}"] & \Omega_{2,\text{max}}
\end{tikzcd}
\end{center}
by first applying $(\text{Id}, T_{1}):X\to X\times X$, then $\gamma_{x,y}:X\times X\to \Omega_{1,\text{min}}$ to acquire,
\[
\Gamma^{1}(x,t)=\gamma_{x, T_{1}(x)}(t)\quad \mu_{1}\;\text{a.e.}\; x
\]
This composition can be thought of as the following mapping,
$\Gamma^{1}: X\to \Omega_{1, \text{min}}$. Identical calculations provide the second composition, $\Gamma^{2}(x', t)=\xi_{x', T_{2}(x')}(t)$ for $\mu_{2}$ a.e. $x'$. 

That these mappings $\Gamma^{i}$ map $\mu_{1}, \mu_{2}$ to an admissible measure in the path space follows from \cite[Lemma 3.14]{cabrera2022optimaltransportationprincipleinteracting}. In particular, if $T_{1}$ and $T_{2}$ map $\mu_{1}\mapsto \nu_{1}$ and $\mu_{2}\mapsto\nu_{2}$, respectively, then  for 
\begin{align}\label{monge plans}
    \pi_{\Gamma^{1}}:=(\Gamma^{1})_{\sharp}\mu_{1};\quad \pi_{\Gamma^{2}}:=(\Gamma^{2})_{\sharp}\mu_{2}, \quad\text{we have}\quad \pi_{\Gamma^{1}}\otimes\pi_{\Gamma^{2}}\in \Pi_{\text{path}}(\mu_{1}, \nu_{1})\otimes\Pi_{\text{path}}(\mu_{2}, \nu_{2}).
\end{align}
Thus, these probability measures are projected from the path space, $\Omega$.

In addition, an application of Proposition \ref{boundary value problem} to $\gamma_{x,y}$  and $\xi_{x', y'}$, respectively,  tell us that we have the following equalities, $c_{\pi_{2}^{*}}(\gamma_{x,y})=c_{e, \pi_{2}^{*}}(x, y)$ and $c_{\pi_{1}^{*}}(\xi_{x',y'})=c_{e, \pi_{1}^{*}}(x^{\prime}, y^{\prime})$, for fixed $\pi_{2}^{*}; \;\pi_{1}^{*}$, respectively. Armed with this knowledge we are prepared to prove Theorem \ref{thm:NashBrenier}.

\begin{proof}[\textit{Proof of Theorem  \ref{thm:NashBrenier}}(Latter statement)]
    Theorems  \ref{thm:minimaxExistence},  \ref{thm:pathtostationary}, and \ref{thm:minimax=maximin} show there is a minimaximal, or maximinimal solution pair of NETP $(\pi_{1}^{*}, \pi_{2}^{*})$ in the path space to \eqref{inf sup problem}.  Furthermore, according to Lemma \ref{lemma: monotonicity}, the pair of solutions $(\pi_{1}^{*}, \pi_{2}^{*})$ projects to the pair of coupled solutions $(\check{\pi}_{1}, \check{\pi}_{2})$, in Euclidean space, with respect to $c_{e, \pi_{2}^{*}}(x,y)$ and $c_{e, \pi_{1}^{*}}(x', y')$, respectively. That is, the pairs $\check{\pi}_{1}=(e_{0}, e_{1})_{\sharp}\pi_{1}^{*}, \check{\pi}_{2}=(e_{0}, e_{1})_{\sharp}\pi_{2}^{*}$ solve the Nash-Monge-Kantorovich bilinear transport problem \eqref{eqn:inf sup euclidean}. Indeed, for $\gamma\in\Omega_{1,\text{min}}$, Proposition \ref{boundary value problem} and Lemma \ref{lemma:saddlepoint-endpoint} apply to show that integrating  \eqref{euc integral} with respect to $\check{\pi}_{1}$, we get 
    \begin{align*}
        \int_{\Omega_{1}}c_{\pi_{2}^{*}}(\gamma)d\pi_{1}^{*}(\gamma)&=\int_{\Omega_{1}}c_{e, \pi_{2}^{*}}(\gamma(0), \gamma(1))d\pi_{1}^{*}(\gamma)\\
        &=\int_{\mathbb{R}^{2d}} c_{e, \pi_{2}^{*}}(x,y)d\check{\pi}_{1}(x,y)\\
       (\;\text{Lemma \ref{lemma:saddlepoint-endpoint}}\;) &=\iint_{\mathbb{R}^{2d}\times\mathbb{R}^{2d}} c_{e}(x,y,x',y')d\check{\pi}_{2}(x',y')d\check{\pi}_{1}(x,y)\\
        &=\iint_{\Omega} c_{e}(\gamma(0), \gamma(1), \xi(0), \xi(1))d\pi_{2}^{*}(\xi)d\pi_{1}^{*}(\gamma).
    \end{align*}
Similar computations are done for $c_{e, \pi_{1}^{*}}$. 

According to \cite[Lemma 2.6]{cabrera2022optimaltransportationprincipleinteracting}, $c_{e, \pi_{2}^{*}}(x,y)$, which is given by \eqref{eqn:cost pi_2}-\eqref{integrtal cost 1}, is differentiable with gradient
\[
\nabla_{y}c_{e, \pi_{2}^{*}}(x,y)=x-y+\frac{\alpha}{3}\int_{\mathbb{R}^{d}\times\mathbb{R}^{d}}(x-x')+2y\;d\pi_{2}^{*}(x', y').
\]
So since, $c_{e, \pi_{2}^{*}}(x,y)$  is given through $c_{e}$ via \eqref{integrtal cost 1}, and as $\alpha<3$, then Lemma \ref{twist condition} says $c_{e,\pi_{2}^{*}}$ satisfies the twist condition.
As $\pi_{1}^{*}$ is minimal with respect to $\inf_{\widetilde{\pi}_{1}\in \Pi_{\text{path}}(\mu_{2}, \nu_{2})}\int_{\Omega_{1}}c_{\pi_{2}^{*}}(\gamma)d\widetilde{\pi}_{1}(\gamma)$, Lemma \ref{lemma: monotonicity} applies to show that the support of $\check{\pi}_{1}$ is $c_{e, \pi_{2}^{*}}$-cyclically monotone. Moreover, the classical theory of optimal transport \cite[Theorem 2.13]{ambrosio2013user} says that  $\text{spt}\;\check{\pi}_{1}$ is contained in $\partial^{c_{e,\pi_{2}^{*}}}\varphi_{1}$ $\mu_{1}$-a.e. In addition the potential $\varphi_{1}$ is locally Lipschitz, and an application of Rademacher's theorem with $\mu_{1}\ll\;dx$ show $\varphi_{1}$ is differentiable $\mu_{1}$-a.e.

All of the above show $c_{e,\pi_{2}^{*}}$ satisfies the hypothesis of Theorem 10.28 in Villani's book \cite{villani2008optimal}, and so it applies to give a unique transport map $T_{1}$ pushing $\mu_{1}$ forward to $\nu_{1}$, solving the bilinear transport problem  \eqref{eqn:inf sup euclidean} with respect to $c_{{e}, \pi_{2}^{*}}(x,y)$ through the equality \eqref{integrtal cost 1}. Now we will incorporate the geometry of paths.

The paragraph preceding this proof guarantees the existence of a mapping $\Gamma^{1}(x,t)$ $\mu_{1}$-a.e. containing the data on $T_{1}$. An application of Proposition \ref{boundary value problem} gives the unique minimal path    $\gamma_{x,y}(t)$ connecting $x$ to $y$. As we have seen before, $c_{\pi_{2}^{*}}(\gamma_{x,y})=c_{e, \pi_{2}^{*}}(x,y)$ and $\Gamma^{1}(x,t)=\gamma_{x,T_{1}(x)}(t)$ only defined for $\mu_{1}$-a.e. $x$, which is given by the composition above. Lemma \ref{lem:support_saddle} says $\text{spt}\;\pi_{1}^{*}$ lies in $\Omega_{\text{min}}$.

Next we need show that $\text{spt}\;\pi_{1}^{*}$ is concentrated on the graph of $\Gamma^{1}$. Previously, we already demonstrated that $\varphi_{1}$ is differentiable for $\mu_{1}$-a.e. $x$ and the support of $\check{\pi}_{1}$ is contained in $\partial^{c_{e, \pi_{2}^{*}}}\varphi_{1}$. Then at each point of differentiability of $\varphi_{1}$; applying one of the coupled ``slackness" constraints of Theorem \ref{thm:coupled duality}, namely $\varphi_{1}(x)+\psi_{1}(y)=c_{e,\pi_{2}^{*}}(x,y)$ for $\pi_{2}^{*}$-a.e. $\gamma$, we have
\begin{align*}
    \nabla_{x}\varphi_{1}(x)=\nabla_{x}c_{e, \pi_{2}^{*}}(x,y)=x-y+\frac{\alpha}{3}\int_{\mathbb{R}^{d}\times\mathbb{R}^{d}}2(x-x')+(y-y')d\pi_{2}^{*}(x',y').
\end{align*}
For any such $y$, keeping $x',$ and $y'$ fixed, such that $(x,y)\in \text{spt}\;\check{\pi}_{1}$, the twist condition on $c_{e}$, and therefore on $c_{e, \pi_{2}^{*}}$, this uniquely defines $y$ as a function of $x$. That is, $y=T_{1}(x)$, and for $\mu_{1}$-a.e. $x$, there exists a unique $y=T_{1}(x)$ such that $(x,y)\in \text{spt}\;\check{\pi}_{1}$. Equivalently, $\gamma_{x, T_{1}(x)}(t)\in \text{spt}\;\pi_{1}^{*}$, and since  $\Gamma^{1}(1, x)=T_{1}(x)=y$, $\pi_{1}^{*}$ is concentrated on the graph of $\Gamma^{1}$. The maps $\Gamma^{i}$ are given by the maps $T_{i}$ in Corollary \ref{cor: maps T}.

The same arguments applied to $c_{e, \pi_{1}^{*}}(x', y')=c_{\pi_{1}^{*}}(\xi_{x',y'})$ with $\check{\pi}_{2}:=(e_{0}, e_{1})_{\sharp}\pi_{2}^{*}$, acquires a map $T_{2}$ and a composition $\Gamma^{2}(x', t)=\xi_{x', T_{2}(x')}(t)\in \text{spt}\;\pi_{2}^{*}$ such that $y'=T_{2}(x')$ is uniquely determined, and that thus $\pi_{2}^{*}$ is concentrated on the graph of $\Gamma^{2}$. Consequently, the coupled bilinear plans $\pi_{\Gamma^{1}}^{*}=(\text{Id}, \Gamma^{1})_{\sharp}\mu_{1}$ and $\pi_{\Gamma^{2}}^{*}=(\text{Id}, \Gamma^{2})_{\sharp}\mu_{2}$ defined in \eqref{monge plans}  are contained in $\Pi_{\text{path}}(\mu_{1}, \nu_{1})\otimes\Pi_{\text{path}}(\mu_{2}, \nu_{2})$ and induce a \emph{Monge} solution \eqref{min max monge}. That the bilinear plans are unique follows from the \emph{proof of Theorem 1.5} in \cite{cabrera2022optimaltransportationprincipleinteracting} applied to each plan, $\check{\pi}_{1}, \check{\pi}_{2}$.
\end{proof}

\begin{remark}
    In the above proof, the regime $0 < \alpha < 4$ ($\alpha \neq 3$) was necessary and sufficient to acquire uniquely determined maps. However, Theorem \ref{thm:NashBrenier} for the stationary case was proven for the more broader regime $\alpha < \frac{\pi^2}{2}$. This is slightly more general than the above result. It would be prudent to investigate why this regime exhibits a discrepancy in the parameter $\alpha$. We leave this for future work.
\end{remark}

As a consequence of the proof of Theorem \ref{thm:NashBrenier}, the result in Corollary \ref{cor: maps T} applies for the maps $\Gamma^{1}, \Gamma^{2}$. Moreover, the dual problem \eqref{eqn:new dual pi_2} on paths is recovered, and we obtain Theorem \ref{thm:coupled duality} and its Corollary \ref{corollary: zero sum dual} associated to paths.  

\begin{corollary}\label{thm: dual on paths}
Let the following cost functions $c_{\pi_{1}^{*}}(\gamma)$ and $c_{\pi_{2}^{*}}(\xi)$ be given by \eqref{eqn: eff path}.  The coupled measures $(\pi_{1}^{*}, \pi_{2}^{*})$ is a NETP for the bilinear transport problem \eqref{inf sup problem} if and only if there exist potentials $\varphi_{i}. \psi_{i}:\mathbb{R}^{d}\to\mathbb{R}$ such that
\begin{align*}
\varphi_{1}(\gamma(0))+\psi_{1}(\gamma(1))&\leq c_{\pi_{2}^{*}}(\gamma)\quad \forall \gamma\in \Omega_{1}\\
\varphi_{1}(\gamma(0))+\psi_{1}(\gamma(1))&=c_{\pi_{2}^{*}}(\gamma)\quad \text{for}\; \pi_{2}^{*}-a.e.\;\gamma,
\end{align*}
and similarly,
\begin{align*}
-\varphi_{2}(\xi(0))-\psi_{2}(\xi(1))&\geq c_{\pi_{1}^{*}}(\xi)\quad \forall \xi\in \Omega_{2}\\
-\varphi_{2}(\xi(0))-\psi_{2}(\xi(1))&=c_{\pi_{1}^{*}}(\xi)\quad \text{for}\; \pi_{1}^{*}-a.e.\;\xi.
\end{align*}
\end{corollary}
\begin{proof}
The proof is based on a paper of the first author \cite[Corollary 4.5]{cabrera2022optimaltransportationprincipleinteracting}, which was based on his dissertation \cite{Cabrera2022}. Thus, a detailed proof can be found there, but in order for this manuscript to be self contained, we provide a proof. 

Suppose $\pi_{2}^{*}$ is a maximizer of \eqref{inf sup problem}. Using Lagrangian multipliers, $\varphi_{1}, \psi_{1}\in C_{c}^{0}(X)\times C_{c}^{0}(X)$ and $\pi_{1}^{*}$ in the set of nonnegative Borel measures, consider the Lagrangian function
\begin{align*}
\Lambda(\pi_{1}^{*}, \varphi_{1}, \psi_{1}, \lambda)&:=\int_{\Omega_{1}}c_{\pi_{2}^{*}}(\gamma)d\pi_{1}^{*}(\gamma)+\int_{X}\varphi_{1}(x)d\mu_{1}(x)-\int_{X}\varphi_{1}(\gamma(0))d\pi_{1}^{*}(\gamma)\\
&+\int_{\Omega_{1}}\psi_{1}(y)d\nu_{1}(y)-\int_{\Omega_{1}}\psi_{1}(\gamma(1))d\pi_{1}^{*}(\gamma)+\int_{\Omega_{1}}\lambda(\gamma)d\pi_{1}^{*}(\gamma).
\end{align*}
Rearranging, we get
\begin{align*}
    \Lambda(\pi_{1}^{*}, \varphi_{1}, \psi_{1}, \lambda)&:=\int_{\Omega_{1}}c_{\pi_{2}^{*}}(\gamma)d\pi_{1}^{*}(\gamma)+\int_{\Omega_{1}}\lambda(\gamma)-\left(\varphi_{1}(\gamma(0))+\psi_{1}(\gamma(1))\right)d\pi_{1}^{*}(\gamma)\\
    &+\int_{X}\varphi_{1}(x)d\mu_{1}(x)+\int_{X}\psi_{1}(y)d\nu_{1}(y)\\
  \eqref{eqn: eff path}&=\iint_{\Omega_{1}\times\Omega_{2}}c(\gamma, \xi)d\pi_{1}^{*}(\gamma)d\pi_{2}^{*}(\xi)+\int_{\Omega_{1}}\lambda(\gamma)-\left(\varphi_{1}(\gamma(0))+\psi_{1}(\gamma(1))\right)d\pi_{1}^{*}(\gamma)\\
  &+\int_{X}\varphi_{1}(x)d\mu_{1}(x)+\int_{X}\psi_{1}(y)d\nu_{1}(y),
\end{align*}
where we included the marginal constraints $(e_{0})_{\sharp}\pi_{1}^{*}=\mu_{1}$ and $(e_{1})_{\sharp}\pi_{1}^{*}=\nu_{1}$, by \cite{cabrera2022optimaltransportationprincipleinteracting}, and the nonnegativity constraint due to the measure, $\int \lambda(\gamma) d\pi_{1}^{*}(\gamma)$.  Then we compute the gradient with respect to $\pi_{1}^{*}$ in $\gamma$. Note that we get linear conditions on $\pi_{1}^{*}$. So according to \cite[Corollary 4.5]{cabrera2022optimaltransportationprincipleinteracting}, we take a ``smooth" curve $\pi_{1}^{*}(s)$ and take derivatives to get an expression regarding  tangent vectors $\dot{\pi}_{1}^{*}$ and an expression in the bilinear term regarding an integral against $\pi_{2}^{*}$, namely, 
\begin{align*}
    \frac{d}{ds}\Big|_{s=0}\Lambda(\pi_{1}^{*}(s), \varphi_{1}, \psi_{1}, \lambda)&=\iint_{\Omega_{1}\times\Omega_{2}}c(\gamma,\xi)d\pi_{2}^{*}(\xi)\dot{\pi}_{1}^{*}(\gamma)\\
    &+\int_{\Omega_{1}}\lambda(\gamma)-\left(\varphi_{1}(\gamma(0))+\psi_{1}(\gamma(1))\right)d\dot{\pi}_{1}^{*}(\gamma).
\end{align*}
Therefore, since we have a \emph{critical point}, this indicates that we get a zero functional. Then minimality---the KKT conditions \cite{rockafellar1970convex}, actually---there exist a pair $\varphi_{1}, \psi_{1}$ and $\lambda\geq 0 $ such that 
\begin{align*}
\int_{\Omega_{2}}c(\gamma, \xi)d\pi_{2}^{*}(\xi)+\lambda(\gamma)-\left(\varphi_{1}(\gamma(0))+\psi_{1}(\gamma(1))\right)&=0 \\
\Longrightarrow c_{\pi_{2}^{*}}(\gamma)+\lambda(\gamma)-\left(\varphi_{1}(\gamma(0))+\psi_{1}(\gamma(1))\right)&=0.
\end{align*}
Moreover, $\lambda\equiv 0$ in the support of $\pi_{1}^{*}$.

Notice that from the fact that $\lambda$ is positive in the above equality, one can see then that the sum of the $\varphi_{1}$ and $\psi_{1}$ is less than the integral term, $c_{\pi_{2}^{*}}(\gamma)$, everywhere, and exactly equal wherever $\lambda$ vanishes in the support of $\pi_{1}^{*}$.

Identical arguments hold for the maximality case.
\end{proof}

The calculation of the last proof helps recover Theorem \ref{thm:coupled duality} in the path spaces. Moreover, since $(\pi_{1}^{*}, \pi_{2}^{*})$ is a NETP the conclusion of Corollary \ref{thm: dual on paths} implies the conclusion of the \emph{zero-sum identity on paths}  of Corollary \ref{corollary: zero sum dual},
\[
\int_{\Omega_{1}}\varphi_{1}(\gamma(0))d\pi_{1}^{*}(\gamma)+\int_{\Omega_{1}}\psi_{1}(\gamma(1))d\pi_{1}^{*}(\gamma)+\int_{\Omega_{2}}\varphi_{2}(\xi(0))d\pi_{2}^{*}(\xi)+\int_{\omega_{2}}\psi_{2}(\xi(1))d\pi_{2}^{*}(\xi)=0,
\]
using Fubini-Tonelli and the fact that $\pi_{1}^{*}, \pi_{2}^{*}$ are probability measures. Furthermore,
this is seen from the definitions of the effective and end-point cost functions defined on \eqref{eqn: eff path}-\eqref{eqn: effective costs}, and an application of Lemma \ref{lemma:saddlepoint-endpoint}. Indeed, integration against $\pi_{1}^{*}$ on the one hand and integration against $\pi_{2}^{*}$, on the other hand, of the  effective costs, are equal, $\int_{\Omega_{2}}c_{\pi_{1}^{*}}(\xi)d\pi_{2}^{*}(\xi)=\int_{\Omega_{1}}c_{\pi_{2}^{*}}(\gamma)d\pi_{1}^{*}(\gamma)$; for $\pi_{1}^{*}$-a.e. $\xi$ (resp. for $\pi_{2}^{*}$-a.e. $\gamma$), and hence 
\[\iint_{\Omega}c_{e, \pi_{1}^{*}}(\xi(0), \xi(1))d\pi_{2}^{*}(\xi)d\pi_{1}^{*}(\gamma)=\iint_{\Omega}c_{e, \pi_{2}^{*}}(\gamma(0), \gamma(1))d\pi_{1}^{*}(\gamma)d\pi_{2}^{*}(\xi),\] which follows from Theorem \ref{thm:minimax=maximin}.


\bibliography{biblio}
\bibliographystyle{plain}
\end{document}